\documentclass[10pt]{amsart}
\usepackage{amsmath,amssymb,amsthm,amsfonts,graphicx,color,mathtools}
\usepackage{amssymb}
\usepackage{amsfonts,amsthm,amsmath,latexsym,bm,graphics}
\usepackage{mathrsfs}

\makeatother
\numberwithin{equation}{section}

\newtheorem{lemma}{Lemma}[section]
\newtheorem{coro}[lemma]{Corollary}
\newtheorem{definition}[lemma]{Definition}
\newtheorem{teo}[lemma]{Theorem}
\newtheorem{proposition}[lemma]{Proposition}
\newtheorem{corollary}[lemma]{Corollary}
\theoremstyle{remark}
\newtheorem{remark}[lemma]{Remark}

\newcommand{\ve} {\varepsilon}
\newcommand{\R}{{\mathbb R}}

\newcommand{\be}{\begin{equation}}
\newcommand{\ben}{\begin{equation*}}
\newcommand{\ee}{\end{equation}}
\newcommand{\een}{\end{equation*}}
\newcommand{\BL}{\begin{lemma}}
\newcommand{\EL}{\end{lemma}}
\newcommand{\BT}{\begin{theorem}}
\newcommand{\ET}{\end{theorem}}
\newcommand{\BP}{\begin{proposition}}
\newcommand{\EP}{\end{proposition}}
\newcommand{\BC}{\begin{corollary}}
\newcommand{\EC}{\end{corollary}}
\def\bs{\begin{split}}
\def\es{\end{split}}
\def\s{\mathbb S}
\def\n{\mathbb N}
\DeclareMathOperator{\Hyperg}{\mbox{ }_2 F_{1}}
\DeclareMathOperator{\divergence}{div}
\DeclareMathOperator{\dist}{dist}
\DeclareMathOperator{\Real}{Re}
\DeclareMathOperator{\Imag}{Im}
\DeclareMathOperator{\Res}{Res}
\DeclareMathOperator{\Ker}{Ker}
\DeclareMathOperator{\rango}{Rg}

\newcommand{\COMMENT}[1]{}
\DeclarePairedDelimiter\abs{\lvert}{\rvert}
\makeatletter
\let\oldabs\abs
\def\abs{\@ifstar{\oldabs}{\oldabs*}}
\newcommand{\eps}{\varepsilon}

\newcommand{\p}{\partial}
\newcommand{\pnu}[1]{\frac{\partial{#1}}{\partial\nu}}

\newcommand{\norm}[2][]{\left\|{#2}\right\|_{#1}}

\newcommand{\set}[1]{\left\{#1\right\}}

\vbadness=\maxdimen

\begin{document}

\title[A fractional Mazzeo-Pacard Program]
{On higher dimensional singularities for the fractional Yamabe problem: a non-local Mazzeo-Pacard program}

\author[W. Ao]{Weiwei Ao}

\address{Weiwei Ao
\hfill\break\indent
Wuhan University
\hfill\break\indent
Department of Mathematics and Statistics, Wuhan, 430072, PR China}
\email{wwao@whu.edu.cn}

\author[H. Chan]{Hardy Chan}

\address{Hardy Chan
\hfill\break\indent
University of British Columbia,
\hfill\break\indent
Department of Mathematics, Vancouver, BC V6T1Z2, Canada}
\email{hardy@math.ubc.ca}

\author[A. DelaTorre]{Azahara DelaTorre}

\address{Azahara DelaTorre
\hfill\break\indent
Albert-Ludwigs-Universit\"{a}t Freiburg
\hfill\break\indent
Mathematisches Institut, Ernst-Zermelo-Str. 1,  D-79104 Freiburg im Breisgau, Germany.}
\email{azahara.de.la.torre@math.uni-freiburg}

\author[M. Fontelos]{Marco A. Fontelos}

\address{Marco A. Fontelos
\hfill\break\indent
ICMAT,
\hfill\break\indent
Campus de Cantoblanco, UAM, 28049 Madrid, Spain}
\email{marco.fontelos@icmat.es}

\author[M.d.M. Gonz\'alez]{Mar\'ia del Mar Gonz\'alez}

\address{Mar\'ia del Mar Gonz\'alez
\hfill\break\indent
Universidad Aut\'onoma de Madrid
\hfill\break\indent
Departamento de Matem\'aticas, Campus de Cantoblanco, 28049 Madrid, Spain}
\email{mariamar.gonzalezn@uam.es}

\author[J. Wei]{Juncheng Wei}
\address{Juncheng Wei
\hfill\break\indent
University of British Columbia
\hfill\break\indent
 Department of Mathematics, Vancouver, BC V6T1Z2, Canada} \email{jcwei@math.ubc.ca}

\begin{abstract}
We consider the problem of constructing solutions to the fractional Yamabe problem that are singular at a given smooth sub-manifold, for which we establish the classical gluing method of Mazzeo and Pacard (\cite{Mazzeo-Pacard}) for the scalar curvature in the fractional setting. This proof is based on the analysis of the model linearized operator,
which amounts to the study of a fractional order ODE,
and thus our main contribution here is the development of new methods coming from conformal geometry and scattering theory for the study of non-local ODEs. Note, however, that no traditional phase-plane analysis is available here. Instead, first, we provide a rigorous construction of radial fast-decaying solutions by a blow-up argument and a bifurcation method. Second, we use conformal geometry to rewrite this non-local ODE, giving a hint of what a non-local phase-plane analysis should be. Third, for the linear theory, we use complex analysis and some non-Euclidean harmonic analysis to  examine a fractional Schr\"{o}dinger equation with a Hardy type critical potential. We construct its Green's function, deduce Fredholm properties, and analyze its asymptotics at the singular points in the spirit of  Frobenius method. Surprisingly enough, a fractional linear ODE may still have a two-dimensional kernel as in the second order case.
\end{abstract}

\date{}\maketitle


\centerline{AMS subject classification:  35J61, 35R11, 53A30}

\section{Introduction}\label{section:introduction}

We construct singular solutions to the following non-local semilinear problem
\begin{equation}\label{equation00}
(-\Delta_{\mathbb R^n})^\gamma u=u^p \mbox{ in }\R^n,\quad u>0,
\end{equation}
for $\gamma\in(0,1)$, $n\geq 2$, where the fractional Laplacian is defined by
\begin{equation}\label{Laplacian-introduction}
(-\Delta_{\mathbb R^n})^\gamma u(z)=k_{n,\gamma}P.V.\int_{\mathbb R^n}  \frac{u(z)-u(\tilde z)}{|z-\tilde z|^{n+2\gamma}}\,d\tilde z, \quad\text{for}\quad k_{n,\gamma}=\pi^{-n/2}2^{2\gamma}
\frac{\Gamma\left(\frac{n}{2}+\gamma\right)}{\Gamma(1-\gamma)}\gamma.
\end{equation}
Equation \eqref{equation00} for the critical power $p=\frac{n+2\gamma}{n-2\gamma}$ corresponds to the fractional Yamabe problem in conformal geometry,  which asks to find a constant fractional curvature metric in a given conformal class (see \cite{Gonzalez-Qing,Gonzalez-Wang,Fang-Gonzalez,Kim-Musso-Wei,Mayer-Ndiaye}). In particular, for $\gamma=1$ the fractional curvature coincides with the scalar curvature modulo a multiplicative constant, so \eqref{equation00} reduces to the classical  Yamabe problem. However, classical methods for local equations do not generally work here and one needs to develop new ideas.

Non-local equations have attracted a great deal of interest in the community since they are of central importance in many fields, from the points of view of both pure analysis and applied modeling. By the substantial effort made in the past decade by many authors, we have learned that non-local elliptic equations do enjoy good PDE properties such as uniqueness, regularity and maximum principle. However,
not so much is known when it comes to the study of an integro-differential equation such as \eqref{equation00} from an ODE perspective since most of the ODE theory relies on local properties and phase-plane analysis; our first achievement is the development of a suitable theory for the  fractional order ODE \eqref{ODE-introduction}, that arises when studying radial singular solutions to  \eqref{equation00}.

On the one hand, we construct singular radial solutions for \eqref{equation00} directly with a completely different argument. On the other hand, using ideas from conformal geometry and scattering theory we replace phase-plane analysis by a global study to obtain that solutions of the nonlocal ODE \eqref{ODE-introduction} do have a behavior similar in spirit to a classical second-order autonomous ODE, and initiate the study of a non-local phase portrait. In particular, we show that a linear  non-local ODE has a two-dimensional kernel. This is surprising since this non-local ODE has an infinite number of indicial roots at the origin and at infinity, which is very different from the local case where the solution to a homogeneous linear second order problem can be written as a linear combination of two particular solutions and thus, its asymptotic behavior is governed by two pairs of indicial roots.

Then, with these tools at hand, we arrive at our second accomplishment: to develop a Mazzeo-Pacard gluing program \cite{Mazzeo-Pacard} for the construction of singular solutions to \eqref{equation00} in the non-local setting. This gluing method is indeed local by definition; so one needs to rethink the theory from a fresh perspective in order to adapt it for
such
non-local equation. More precisely, the program relies on the fact that the linearization to \eqref{equation00} has good  properties. In the classical case, this linearization has been well studied applying  microlocal analysis (see \cite{Mazzeo:edge}, for instance), and it reduces to the understanding of a second order ODE with two regular singular points.  In the fractional case this is obviously not possible. Instead, we use conformal geometry, complex analysis and some non-Euclidean harmonic analysis coming from representation theory in order to provide a new proof.

Thus conformal geometry is the central core in this paper, but we provide an interdisciplinary approach in order to approach the following analytical problem:

\begin{teo}\label{teo}
Let $\Sigma=\bigcup_{i=1}^K \Sigma_i$ be a disjoint union of smooth, compact sub-manifolds $\Sigma_i$ without boundary of dimensions $k_i$, $i=1,\ldots,K$. Assume, in addition to $n-k_i\geq 2$,  that
\begin{equation*}
\frac{n-k_i}{n-k_i-2\gamma}<p<\frac{n-k_i+2\gamma}{n-k_i-2\gamma},
\end{equation*}
or equivalently,
\begin{equation*}
n-\frac{2p\gamma+2\gamma}{p-1}<k_i<n-\frac{2p\gamma}{p-1}
\end{equation*}
for all $i=1,\ldots,K$. Then there exists a positive solution for the problem
\begin{equation}\label{problem-introduction}
(-\Delta_{\mathbb R^n})^\gamma u=u^p \mbox{ in }\R^n\setminus \Sigma
\end{equation}
that blows up exactly at $\Sigma$.
\end{teo}

 As a consequence of the previous theorem we obtain:

\begin{coro} Assume that the dimensions $k_i$ satisfy
\begin{equation}\label{dim-estimate}
0<k_i<\frac{n-2\gamma}{2}.
\end{equation}
Then there exists a positive solution to the fractional Yamabe equation
\begin{equation}\label{problem-Yamabe}
(-\Delta_{\mathbb R^n})^\gamma u=u^{\frac{n+2\gamma}{n-2\gamma}} \mbox{ in }\R^n\setminus \Sigma
\end{equation}
that blows up exactly at $\Sigma$.
\end{coro}

The dimension estimate \eqref{dim-estimate} is sharp in some sense. Indeed, it was proved by Gonz\'{a}lez, Mazzeo and Sire \cite{Gonzalez-Mazzeo-Sire} that, if such $u$ blows up at a smooth sub-manifold of dimension $k$ and is polyhomogeneous, then $k$ must satisfy the restriction
\begin{equation*}
\Gamma\left(\frac{n}{4} - \frac{k}{2} + \frac{\gamma}{2}\right) \Big/ \Gamma\left(\frac{n}{4} - \frac{k}{2} - \frac{\gamma}{2}\right)\geq 0,
\label{eq:dimrest}
\end{equation*}
which in particular, includes \eqref{dim-estimate}. Here, and for the rest of the paper, $\Gamma$ denotes the Gamma function. In addition, the asymptotic behavior of  solutions to \eqref{problem-Yamabe} when the singular set has fractional capacity zero has been considered in \cite{Jin-Queiroz-Sire-Xiong}.

\bigskip


Let us describe our methods in detail. First, note that it is enough to let $\Sigma$  be a single sub-manifold of dimension $k$, and we will restrict to this case for the rest of the paper. We denote $N=n-k$.

The first step is to construct the building block, i.e, a solution to \eqref{problem-introduction} in $\mathbb R^n\setminus \mathbb R^k$ that blows up exactly at $\mathbb R^k$. For this, we write $\mathbb R^n\setminus \mathbb R^k=(\mathbb R^N\setminus\{0\})\times \mathbb R^k$, parameterized with coordinates $z=(x,y)$, $x\in\mathbb R^N\setminus\{0\}$, $y\in \mathbb R^k$, and construct a solution $u_1$ that only depends on the radial variable $r=|x|$. Then $u_1$ is also a radial solution to
\begin{equation*}\label{equation-introduction-isolated}
(-\Delta_{\mathbb R^N})^\gamma u=A_{N,p,\gamma}u^p \quad\text{in }\mathbb R^N\setminus\{0\},\quad u>0.
\end{equation*}
We write $u=r^{-\frac{2\gamma}{p-1}}v$, $r=e^{-t}$. Then, in the radially symmetric case, this equation can be written as the integro-differential ODE
\begin{equation}\label{ODE-introduction}
P.V. \int_{\mathbb R}K(t-t')[v(t)-v(t')]\,dt'+A_{N,p,\gamma}v(t)=A_{N,p,\gamma}v^p\quad\text{in }\mathbb R,\quad v>0,
\end{equation}
where the kernel $K$ is given precisely in \eqref{kernel-K}. However, in addition to having the right blow up rate at the origin, $u_1$ must decay fast as $r\to\infty$ in order to perform the Mazzeo-Pacard gluing argument later. The existence of such fast-decaying singular solutions in the case of $\gamma=1$ is an easy consequence  of phase-plane analysis as \eqref{ODE-introduction} is reduced to a second order autonomous ODE (see Proposition 1 of \cite{Mazzeo-Pacard}). The analogue in the fractional case turns out to be quite non-trivial.    To show the existence, we first use Kelvin transform to reduce our  problem for entire solutions to a supercritical one \eqref{entire}. Then we consider an auxiliary non-local problem \eqref{problem}, for which we show that the minimal solution $w_\lambda$ is unique using Schaaf's argument as in \cite{Esposito-Ghoussoub} and a fractional Poho\v{z}aev identity \cite{RosOton-Serra1}. A blow up argument, together with a Crandall-Rabinowitz bifurcation scheme yields the existence of this $u_1$. This is the content of Section \ref{section:fast-decaying-solution}.

Then, in Section \ref{section:conformal}, we exploit the conformal properties of the equation  to produce a geometric interpretation for \eqref{problem-introduction} in terms of scattering theory and conformally covariant operators.
Singular solutions for the standard fractional Laplacian in $\mathbb R^n\setminus\mathbb R^k$ can be better understood by considering the conformal metric $g_k$ from \eqref{metric-gk}, that is the product of a sphere $\mathbb S^{N-1}$ and a half-space $\mathbb H^{k+1}$. Inspired by the arguments by DelaTorre and Gonz\'{a}lez \cite{DelaTorre-Gonzalez}, our point of view is to rewrite the well known extension problem in $\mathbb R^{n+1}_+$ for the fractional Laplacian in $\mathbb R^n$ due to \cite{Caffarelli-Silvestre}, as a different, but equivalent, extension problem and to consider the corresponding Dirichlet-to-Neumann operator $P_\gamma^{g_k}$, defined in  $\mathbb S^{N-1}\times \mathbb H^{k+1}$. Here $\mathbb R^{n+1}_+$ is replaced by anti-de Sitter (AdS) space, but the arguments run in parallel.

This $P_\gamma^{g_k}$ turns out to be a conjugate operator for $(-\Delta_{\mathbb R^n})^\gamma$, (see \eqref{relation-conjugation}), and behaves well when the nonlinearity in \eqref{problem-introduction} is the conformal power. However, the problem \eqref{problem-introduction} is not conformal for a general $p$, so we need  to perform a further conjugation \eqref{tilde-P} and to consider the new operator $\tilde P_\gamma^{g_k}$. Then the original equation \eqref{problem-introduction} in $\mathbb R^n\setminus\mathbb R^k$ is equivalent to
\begin{equation}\label{equation-introduction-tilde-P}
\tilde P_\gamma^{g_k}(v)=v^{p}\quad \text{in }\mathbb S^{N-1}\times \mathbb H^{k+1},\quad v=r^{\frac{2\gamma}{p-1}}u, \quad v>0 \text{ and smooth.}
\end{equation}
Rather miraculously, both
$P_\gamma^{g_k}$ and $\tilde P_\gamma^{g_k}$ diagonalize under the spherical harmonic decomposition of $\mathbb S^{N-1}$. In fact, they can be understood as pseudo-differential operators on hyperbolic space $\mathbb H^{k+1}$, and we calculate their symbols in  Theorem \ref{thm:symbol} and Proposition \ref{thm:symbolV}, respectively, under the
Fourier-Helgason transform
(to be denoted by \,$\widehat{\cdot}$\,)
on the hyperbolic space understood as the symmetric space  $M=G/K$ for $G=SO(1,k+1)$ and $K=SO(k+1)$ (see the Appendix for a short introduction to the subject). This is an original approach that yields new results even in the classical case $\gamma=1$, simplifying some of the arguments in \cite{Mazzeo-Pacard}.
The precise knowledge of their symbols allows, as a consequence, for
the development of the linear theory for our problem, as we will comment below.

Section  \ref{section:ODE-methods} collects these ideas in order to develop new methods for the study of the non-local ODE \eqref{ODE-introduction}, which is precisely the projection  of equation \eqref{equation-introduction-tilde-P} for $k=0$, $n=N$, over the zero-eigenspace when projecting over spherical harmonics of $\mathbb S^{N-1}$. The advantage of shifting from $u$ to $v$ is that we obtain a new equation that behaves very similarly to a second order autonomous ODE. This includes the existence of a Hamiltonian quantity along trajectories.

Moreover, one can take the spherical harmonic decomposition of $\mathbb S^{N-1}$ and consider all projections $m=0,1,\ldots$. In Proposition \ref{prop:all-kernels} we are able to write every projected equation as an integro-differential equation very similar to the $m=0$ projection \eqref{ODE-introduction}. This formulation immediately yields regularity and maximum principles for the solution of \eqref{equation-introduction-tilde-P} following the arguments in \cite{DelaTorre-delPino-Gonzalez-Wei}.

Now, to continue with the proof of Theorem \ref{teo}, one takes the fast decaying solution in $\mathbb R^n\setminus \mathbb R^k$ we have just constructed and,  after some rescaling by $\ve$, glues it to the background Euclidean space in order to have a global approximate solution $\bar u_\ve$ in $\mathbb R^n\setminus\Sigma$. Even though the fractional Laplacian is a non-local operator, one is able to perform this gluing just by carefully estimating the tail terms that appear in the integrals after localizaton. This is done in Section \ref{section:function-spaces} and, more precisely, Lemma \ref{lemma:error2},
where we show
that the error we generate when approximating a  true solution by $\bar u_\ve$,
given by
\begin{equation*}
f_\ve:=(-\Delta_{\mathbb R^n})^\gamma \bar{u}_\ve-\bar{u}_\ve^p,
\end{equation*}
is indeed small in suitable weighted H\"older spaces.

Once we have an approximate solution, we define the linearized operator around it,
\begin{equation*}
L_\ve \phi:=(-\Delta_{\mathbb R^n})^\gamma \phi-p\bar{u}_\ve^{p-1}\phi.
\end{equation*}
The general scheme of Mazzeo-Pacard's method is to  set $u=\bar u_\ve+\phi$ for an unknown perturbation $\phi$ and to rewrite equation \eqref{problem-introduction} as
\begin{equation*}
L_{\ve}(\phi)+Q_{\ve}(\phi)+f_{\ve}=0,
\end{equation*}
where $Q_\ve$ contains the remaining nonlinear terms. If $L_\ve$ is invertible, then we can write
\begin{equation*}
\phi=(L_\ve)^{-1} (-Q_\ve(\phi)-f_\ve),
\end{equation*}
and a standard fixed point argument for small $\ve$ will yield the existence of such $\phi$, thus completing the proof of Theorem \ref{teo} (see Section \ref{section:conclusion}).

Thus, a central argument here is the study of the linear theory for $L_\ve$ and, in particular, the analysis of its indicial roots, injectivity and Fredholm properties. However, while the behaviour of a second order ODE is governed by two boundary conditions (or behavior at the singular points using Frobenius method), this may not be true in general for a non-local operator.

We first consider the model operator $\mathcal L_1$ defined in \eqref{model-linearization} for an isolated singularity at the origin. Near the singularity $\mathcal L_1$ behaves
like
\begin{equation}\label{Hardy-introduction}
(-\Delta_{\mathbb R^N})^\gamma-\frac{\kappa}{r^{2\gamma}}
\end{equation}
or, after conjugation, like $P_\gamma^{g_0}-\kappa$, which is a fractional Laplacian  operator with a Hardy potential of critical type.

The central core of the linear theory deals with the operator \eqref{Hardy-introduction}. In Section \ref{section:Hardy} we perform a delicate study of the Green's function by inverting its Fourier symbol $\Theta_\gamma^m$ (see \eqref{symbol-isolated}). This requires a very careful analysis of the poles of the symbol,
in both the stable and unstable cases. Contrary to the local case $\gamma=1$,
in which there are
only  two indicial roots for each projection $m$, here we find an infinite sequence for each $m$. But in any case, these are controlled. It is also interesting to observe that, even though we have a non-local operator, the first pair of indicial roots governs the asymptotic behavior of the operator and thus, its kernel is two-dimensional in some sense (see, for instance, Proposition \ref{two-dimensional} for a precise statement).

For the interested reader, we have prepared a separate survey paper \cite{survey-ODE} where we summarize the new ODE methods for non-local equations that we have developed here.

Then, in Section \ref{section:linear-theory} we complete the calculation of the indicial roots (see Lemma \ref{indicial}). Next, we show the injectivity for $\mathcal L_1$ in weighted H\"older spaces, and an \emph{a priori} estimate (Lemma \ref{lemma:apriori-estimate}) yields the injectivity for $L_\ve$.

In addition, in Section \ref{section:Fredholm} we work with weighted Hilbert spaces and we prove Fredholm properties for $L_\ve$ in the spirit of the results by Mazzeo \cite{Mazzeo:edge,Mazzeo:edge2} for edge type operators by constructing a suitable parametrix with compact remainder. The difficulty lies precisely in the fact that we are working with a non-local operator, so the localization with a cut-off is the non-trivial step. However, by working with suitable weighted spaces we are able to localize the problem near the singularity; indeed, the tail terms are small. Then we conclude that $L_\ve$ must be surjective by purely functional analysis reasoning. Finally, we construct a right inverse for $L_\ve$, with norm uniformly bounded independently of $\ve$, and this concludes the proof of Theorem \ref{teo}.

The Appendix contains some well known results on special functions and the Fourier-Helgason transform.

\bigskip

As a byproduct of the proof of Theorem \ref{teo}, we will obtain the existence of solutions with isolated singularities in the subcritical regime (note the shift from $n$ to $N$ in the spatial dimension, which will fit better our purposes).

\begin{teo}\label{teo:points}
Let $\gamma\in(0,1)$, $N\geq 2$ and
\begin{equation}\label{exponent-p}
\frac{N}{N-2\gamma}< p< \frac{N+2\gamma}{N-2\gamma}.
\end{equation}
Let $\Sigma$ be a finite number of points, $\Sigma=\{q_1,\ldots,q_K\}$. Then equation
\begin{equation*}\label{problem-points}
(-\Delta_{\mathbb R^N})^\gamma u=A_{N,p,\gamma}u^p \mbox{ in }\R^N\setminus \Sigma
\end{equation*}
 has positive solutions that blow up exactly at $\Sigma$.
\end{teo}

\begin{remark}\label{remark:constants}
The constant $A_{N,p,\gamma}$ is chosen so
that the model
function
$u_\gamma(x)=|x|^{-\frac{2\gamma}{p-1}}$ is a singular solution to \eqref{problem-introduction} that blows up exactly at the origin. In particular,
\begin{equation}\label{Apn}
A_{N,p,\gamma}=\Lambda\big(\tfrac{N-2\gamma}{2}-\tfrac{2\gamma}{p-1}\big)\quad\text{for}\quad
\Lambda(\alpha)=2^{2\gamma}
\frac{\Gamma(\frac{N+2\gamma+2\alpha}{4})\Gamma(\frac{N+2\gamma-2\alpha}{4})}
{\Gamma(\frac{N-2\gamma-2\alpha}{4})\Gamma(\frac{N-2\gamma+2\alpha}{4})}.
\end{equation}
Note that, for the critical exponent $p=\frac{N+2\gamma}{N-2\gamma}$, the constant  $A_{N,p,\gamma}$ coincides with $\Lambda_{N,\gamma}=\Lambda(0)$, the sharp constant  in the fractional Hardy inequality in $\mathbb R^N$. Its precise value is given in \eqref{Hardy-constant}.
\end{remark}

Let us make some comments on the bibliography. First, the Brezis-Nirenberg problem for the fractional Laplacian has been considered in \cite{Servadei-Valdinoci} through variational techniques, which is one of the first papers on semi-linear equations with critical power non-linearity for the fractional Laplacian. In addition, the problem of uniqueness and non-degeneracy for some fractional ODE has been considered in \cite{Frank-Lenzmann,Frank-Lenzmann-Silvestre,dps}, for instance.

The construction of singular solutions in the range of exponents for which the problem is stable, i.e.,
$\frac{N}{N-2\gamma}<p<p_1$ for $p_1<\frac{N+2\gamma}{N-2\gamma}$ defined in \eqref{p1}, was studied in the previous paper by Ao, Chan, Gonz\'{a}lez and Wei \cite{Ao-Chan-Gonzalez-Wei}. In addition, for the critical case $p=\frac{N+2\gamma}{N-2\gamma}$, solutions with a  finite number of isolated singularities were obtained in the article by Ao, DelaTorre, Gonz\'{a}lez and Wei \cite{gluing} using a gluing method. The difficulty there was the presence of a non-trivial kernel for the linearized operator. With all these results, together with Theorem \ref{teo}, we have successfully developed a complete fractional Mazzeo-Pacard program for the construction of singular solutions of the fractional Yamabe problem.

Gluing methods for fractional problems are starting to be developed. A finite dimensional reduction has been applied in \cite{Davila-DelPino-Wei} to construct standing-wave solutions to a fractional nonlinear Schr\"odinger equation and in \cite{Du-Gui-Sire-Wei} to construct layered solutions for a fractional inhomogeneous Allen-Cahn equation.

The next development came in \cite{gluing} for the fractional Yamabe problem with isolated singularities, that we have just mentioned. There the model for an isolated singularity is a Delaunay-type metric  (see also \cite{Mazzeo-Pacard:Delaunay-ends,Mazzeo-Pacard-Pollack,Schoen:isolated-singularities} for the construction of constant mean curvature surfaces with Delaunay ends and \cite{mp,Mazzeo-Pollack-Uhlenbeck} for the scalar curvature case). However, in order to have enough freedom parameters at the perturbation step, for the non-local gluing in \cite{gluing} the authors replace the Delaunay-type solution by a bubble tower (an infinite, but countable, sum of bubbles). As a consequence, the reduction method becomes infinite dimensional.
Nevertheless, it can still be
treated with the tools available in the finite dimensional case and one reduces the PDE to an infinite dimensional Toda type system. The most recent works related to gluing are
\cite{Chan-Liu-Wei,Chan-Davila-DelPino-Liu-Wei} for the construction of counterexamples to the fractional De Giorgi conjecture. This reduction is fully infinite dimensional.

For the fractional De Giorgi conjecture with $\gamma\in[\frac12,1)$ we refer to \cite{CabreSola-Morales, CC, Savin2} and the most recent striking paper \cite{FS}. Related to this conjecture, in the case $\gamma\in(0,\frac12)$ there exists a notion of non-local mean curvature for hypersurfaces in $\mathbb R^n$, see \cite{Caffarelli-Roquejoffre-Savin} and the survey \cite{Valdinoci:review}. Much effort has been made regarding regularity \cite{Caffarelli-Valdinoci, Cabre-Cinti-Serra, Barrios-Figalli-Valdinoci} and various qualitative properties \cite{DSV1, DSV2}.  More recent work on stability of non-local minimal surfaces can be found in \cite{CSV}. Delaunay surfaces for this curvature have been constructed in \cite{Cabre-Fall-Sola-Weth,Cabre-Fall-Weth}. After the appearance of \cite{Davila-delPino-Dipierro-Valdinoci}, Cabr\'e has pointed out that this paper also constructs Delaunay surfaces with constant nonlocal mean curvature.

\section{The fast decaying solution}\label{section:fast-decaying-solution}

We aim to construct a fast-decay singular solution to the fractional Lane--Emden equation
\begin{equation}\label{Lane-Emden}
(-\Delta_{\mathbb R^N})^\gamma{u}=A_{N,p,\gamma}u^p\quad\text{in}~\R^N\setminus \{0\}.
\end{equation}
for $\gamma\in(0,1)$ and $p$  in the range \eqref{exponent-p}.

We consider the exponent $p_1=p_1(N,\gamma)\in(\frac{N}{N-2\gamma},\frac{N+2\gamma}{N-2\gamma})$  defined below by \eqref{p1} such that the singular solution $u_\gamma(x)=\abs{x}^{-\frac{2\gamma}{p-1}}$ is stable if and only if $\frac{N}{N-2\gamma}<p<p_1$.
In the notation of Remark \ref{remark:constants}, $p_1$ as defined as the root of
\begin{equation}\label{p1}
pA_{N,p,\gamma}=\Lambda(0).
\end{equation}

The main result in this section is:

\begin{proposition}\label{existence}
For any $\eps\in(0,\infty)$ there exists a fast-decay entire singular solution $u_\eps$ of \eqref{Lane-Emden} such that
\begin{equation*}
u_\eps(x)\sim
\begin{cases}
\abs{x}^{-\frac{2\gamma}{p-1}}&\quad\text{as}~\abs{x}\to0,\\
\eps\abs{x}^{-(N-2\gamma)}&\quad\text{as}~\abs{x}\to\infty.
\end{cases}
\end{equation*}
\end{proposition}


The proof in the stable case $\frac{N}{N-2\gamma}<p<p_1<\frac{N+2\gamma}{N-2\gamma}$ is already contained in the paper \cite{Ao-Chan-Gonzalez-Wei}, so we will assume for the rest of the section that we are in the unstable regime
\begin{equation*}
\frac{N}{N-2\gamma}<p_1\leq{p}<\frac{N+2\gamma}{N-2\gamma}.
\end{equation*}

We first prove  uniqueness of minimal solutions for the non-local problem \eqref{problem} using Schaaf's argument and a fractional Poho\v{z}aev identity obtained by Ros-Oton and Serra (Proposition \ref{uniqueness} below). Then we perform a blow-up argument on an unbounded bifurcation branch. An application of Kelvin's transform yields an entire solution of the Lane--Emden equation with the desired asymptotics.

Set  $A=A_{N,p,\gamma}$.
Note that the Kelvin transform  $w(x)=\abs{x}^{-(N-2\gamma)}u\left(\frac{x}{\abs{x}^2}\right)$ of $u$ satisfies
\begin{equation}
\label{entire}
(-\Delta)^\gamma{w}(x)=A\abs{x}^{\beta}w^p(x),
\end{equation}
where $\beta=:p(N-2\gamma)-(N+2\gamma)\in(-2\gamma,0)$.

Consider the following non-local Dirichlet problem in the unit ball $B_1=B_1(0)\subset\R^N$,
\begin{equation}\label{problem}\begin{cases}
(-\Delta)^\gamma{w}(x)=\lambda\abs{x}^{\beta}A(1+w(x))^p~&\quad\text{in}~B_1,\\
w=0~&\quad\text{in}~\R^N\setminus{B_1}.
\end{cases}\end{equation}
Since $(-\Delta)^\gamma\abs{x}^{\beta+2\gamma}=c_0\abs{x}^\beta$ and $(-\Delta)^\gamma(1-\abs{x}^2)_+^\gamma=c_1$ for some positive constants $c_0$ and $c_1$, we have that $\abs{x}^{\beta+2\gamma}+(1-\abs{x}^2)_+^\gamma$ is a positive super-solution of \eqref{problem} for small $\lambda$. Then one can follow classical arguments (we refer to \cite{RosOton-Serra2}, for instance), to show that there exists a minimal solution $w_\lambda$ for small $\lambda$. Moreover, $w_\lambda$ is non-decreasing in $\lambda$. Thus one can find a $\lambda^*>0$ such that: (i) the minimal solution $w_\lambda$ exists for each $\lambda\in(0,\lambda^*)$, and $w_\lambda$ is radially symmetric and non-increasing in the radial variable; (ii) for $\lambda>\lambda^*$, \eqref{problem} has no solutions.

  We will show that $w_\lambda$ is the unique solution of \eqref{problem} for all small $\lambda$.

\begin{proposition}\label{uniqueness}
There exists a small $\lambda_0>0$ depending only on $N\geq2$ and $\gamma\in(0,1)$ such that for any $0\leq\lambda<\lambda_0$, $w_\lambda$ is the unique solution to \eqref{problem} among the class
$$\tilde{\mathcal C}^{2}_{0}(B_1)=\set{w\in{\mathcal C^2}(B_1)\cap \mathcal{C}(\R^N):w=0 \text{ in } \R^N\setminus B_1}.$$

\end{proposition}

The idea of the proof of this Proposition follows from \cite{Esposito-Ghoussoub} and similar arguments can be found in \cite{Schaaf}, \cite{Guo-Wei1} and \cite{Guo-Wei2}.

\subsection{Useful inequalities}

The first ingredient is the Poho\v{z}aev identity for the fractional Laplacian. Such identities for integro-differential operators have been recently studied in \cite{RosOton-Serra1}, \cite{RosOton-Serra-Valdinoci} and \cite{Grubb}.

\begin{teo}[Proposition 1.12 in \cite{RosOton-Serra1}]\label{pohozaev}
Let $\Omega$ be a bounded $\mathcal C^{1,1}$ domain, $f\in{\mathcal C}_{loc}^{0,1}(\overline{\Omega}\times\R)$, $u$ be a bounded solution of
\begin{equation}\begin{cases}
(-\Delta)^\gamma{u}=f(x,u)~&\text{in}~\Omega,\\
u=0~&\text{in}~\R^N\setminus\Omega,
\end{cases}\end{equation}
and $\delta(x)=\dist(x,\p\Omega)$. Then
$$u/\delta^\gamma\mid_\Omega\in{\mathcal C}^{\alpha}(\overline\Omega)\qquad\text{for}~\text{some}~\alpha\in(0,1),$$
and there holds
\begin{equation*}
\int_{\Omega}\left(F(x,u)+\dfrac{1}{N}x\cdot\nabla_{x}F(x,u)-\dfrac{N-2\gamma}{2N}uf(x,u)\right)\,dx
=\dfrac{\Gamma(1+\gamma)^2}{2N}\int_{\p\Omega}\left(\dfrac{u}{\delta^\gamma}\right)^2(x\cdot\nu)\,d\sigma
\end{equation*}
where $F(x,t)=\int_0^{t}f(x,\tau)\,d\tau$ and $\nu$ is the unit outward normal to $\p\Omega$ at $x$.
\end{teo}
Using integration by parts (see, for instance, equation (1.5) in \cite{RosOton-Serra1}), it is clear that
$$\int_{\Omega}uf(x,u)\,dx=\int_{\R^N}\abs{(-\Delta)^{\frac{\gamma}{2}}u}^2\,dx,$$
which yields our fundamental inequality:
\begin{coro}
Under the assumptions of Theorem \ref{pohozaev}, we have for any star-shaped domain $\Omega$ and any $\sigma\in\R$,
\begin{equation}\label{poho1}
\int_{\Omega}\left(F(x,u)+\dfrac{1}{N}x\cdot\nabla_{x}F(x,u)-\sigma uf(x,u)\right)\,dx
\geq\left(\dfrac{N-2\gamma}{2N}-\sigma \right)\int_{\R^N}\abs{(-\Delta)^{\frac{\gamma}{2}}u}^2\,dx
\end{equation}
\end{coro}

The second ingredient is the fractional Hardy--Sobolev inequality which, via H\"older inequality, is an interpolation of fractional Hardy inequality and fractional Sobolev inequality:

\begin{teo}[Lemma 2.1 in \cite{Ghoussoub-Shakerian}]
Assume that $0\leq\alpha<2\gamma<\min\{2,N\}$. Then there exists a constant $c$ such that
\begin{equation}\label{HS1}
c\int_{\R^N}\abs{(-\Delta)^{\frac{\gamma}{2}}{u}}^2\,dx
\geq\left(\int_{\R^N}\abs{x}^{-\alpha}\abs{u}^{\frac{2(N-\alpha)}{N-2\gamma}}
\right)^{\frac{N-2\gamma}{N-\alpha}}.
\end{equation}
\end{teo}

\subsection{Proof of Proposition \ref{uniqueness}}

We are now in a position to prove the uniqueness of solutions of \eqref{problem} with small parameter.

\begin{proof}
Suppose $w$ and $w_\lambda$ are solutions to \eqref{problem}. Then $u=w-w_\lambda$ is a positive solution to the Dirichlet problem
\begin{equation*}\label{problem-difference}\begin{cases}
(-\Delta)^\gamma {u}=\lambda A\abs{x}^{\beta}g_{\lambda}(x,u)~&\quad\text{in}~B_1(0),\\
u=0~&\quad\text{in}~\R^N\setminus{B}_1(0),
\end{cases}\end{equation*}
where $g_{\lambda}(x,u)=(1+w_\lambda(x)+u)^{p}-(1+w_\lambda(x))^{p}\geq0$ for $u\geq0$. Denoting
\begin{equation*}\begin{split}
G_\lambda(x,u)
&=\int_0^{u}g_\lambda(x,t)\,dt,\\
\end{split}\end{equation*}
we apply \eqref{poho1} with $f(x,u)=\lambda A\abs{x}^{\beta}g_\lambda(x,u)$ over $\Omega=B_1$ to obtain
\begin{equation}\label{poho2}\begin{split}
&\left(\dfrac{N-2\gamma}{2N}-\sigma \right)\int_{\R^N}\abs{(-\Delta)^{\frac{\gamma}{2}}{u}}^2\,dx\\
&\quad\leq\lambda A\int_{B_1}\left(\abs{x}^{\beta}G_\lambda(x,u)+\dfrac{1}{N}x\cdot\nabla_{x}\left(\abs{x}^{\beta}G_\lambda(x,u)\right)
-\sigma \abs{x}^{\beta}ug_\lambda(x,u)\right)\,dx\\
&\quad=\lambda A\int_{B_1}\abs{x}^{\beta}\left(\left(1+\dfrac{\beta}{N}\right)
G_\lambda(x,u)+\dfrac{1}{N}x\cdot\nabla_{x}G_\lambda(x,u)-\sigma ug_\lambda(x,u)\right)\,dx.
\end{split}\end{equation}
Note that
\begin{equation}\label{G}
G_\lambda(x,u)
=u^2\int_0^{1}\int_0^{1}pt(1+w_\lambda(x)+\tau{t}u)^{p-1}\,d{\tau}dt
\end{equation}
and
\begin{equation*}
\nabla_{x}G_\lambda(x,u)
=u^2\int_0^{1}\int_0^{1}p(p-1)t(1+w_\lambda(x)+\tau{t}u)^{p-2}\,d{\tau}dt\cdot\nabla{w}_\lambda(x).\\
\end{equation*}
Since $w_\lambda$ is radially decreasing, $x\cdot\nabla{w}_\lambda(x)\leq0$ and hence $x\cdot\nabla_{x}G_\lambda(x,u)\leq0$. Then \eqref{poho2} becomes
\begin{equation}\label{poho3}\begin{split}
\quad\,\left(\dfrac{N-2\gamma}{2N}-\sigma \right)\int_{\R^N}\abs{(-\Delta)^{\frac{\gamma}{2}}{u}}^2\,dx
\leq\lambda A\int_{B_1}\abs{x}^{\beta}\left(\left(1+\dfrac{\beta}{N}\right)G_\lambda(x,u)-\sigma ug_\lambda(x,u)\right)\,dx.
\end{split}\end{equation}

Now, since for any $\lambda\in\left[0,\frac{\lambda^*}{2}\right]$ and any $x\in{B_1}$,
\begin{equation*}\begin{split}
\lim_{t\to\infty}\dfrac{G_\lambda(x,t)}{tg_\lambda(x,t)}
=\lim_{t\to\infty}\dfrac{\frac{1}{p+1}
\left((1+w_\lambda(x)+t)^{p+1}-(1+w_\lambda(x))^{p+1}\right)-(1+w_\lambda(x))^{p}t}
{t\left((1+w_\lambda(x)+t)^{p}-(1+w_\lambda(x))^p\right)}
=\dfrac{1}{p+1},
\end{split}\end{equation*}
we deduce that for any $\epsilon>0$ there exists an $M=M(\epsilon)>0$ such that
$$G_\lambda(x,t)\leq\dfrac{1+\epsilon}{p+1}ug_\lambda(x,t)$$
whenever $t\geq{M}$.
From this we estimate the tail of the right hand side of \eqref{poho3} as
\begin{equation*}\begin{split}
\int_{B_1\cap\set{u\geq{M}}}&\abs{x}^{\beta}\left(\left(1+\dfrac{\beta}{N}\right)G_\lambda(x,u)-\sigma ug_\lambda(x,u)\right)\,dx\\
&\leq\int_{B_1\cap\set{u\geq{M}}}\abs{x}^{\beta}\left(\left(1+\frac{\beta}{N}\right)\dfrac{1+\epsilon}{p+1}-\sigma \right)ug_\lambda(x,u)\,dx.
\end{split}\end{equation*}
We wish to choose $\epsilon$ and $\sigma$ such that
$$\left(1+\dfrac{\beta}{N}\right)\dfrac{1+\epsilon}{p+1}<\sigma <\dfrac{N-2\gamma}{2N},$$
so that the above integral is non-positive. Indeed we observe that
\begin{equation*}\begin{split}
\left(\dfrac{N+\beta}{N}\right)\dfrac{1}{p+1}-\dfrac{N-2\gamma}{2N}
&=\dfrac{2(p(N-2\gamma)-2\gamma)-(N-2\gamma)(p+1)}{2N(p+1)}\\
&=\dfrac{(p-1)(N-2\gamma)-4\gamma}{2N(p+1)}\\
&<0
\end{split}\end{equation*}
as $p-1\in\left(\frac{2\gamma}{N-2\gamma},\frac{4\gamma}{N-2\gamma}\right)$. Then there exists a small $\epsilon>0$ such that
$$\left(1+\dfrac{\beta}{N}\right)\dfrac{1+\epsilon}{p+1}<\dfrac{N-2\gamma}{2N},$$
from which the existence of such $\sigma$ follows. With this choice of $\epsilon$ and $\sigma$, \eqref{poho3} gives
\begin{equation*}\begin{split}
&\quad\,\left(\dfrac{N-2\gamma}{2N}-\sigma\right)\int_{\R^N}\abs{(-\Delta)^{\frac{\gamma}{2}}{u}}^2\,dx\\
&\leq\lambda A\int_{B_1\cap\set{u<M}}\abs{x}^{\beta}\left(\left(1+\dfrac{\beta}{N}\right)G_\lambda(x,u)-\sigma ug_\lambda(x,u)\right)dx\\
&\leq\lambda A\left(1+\dfrac{\beta}{N}\right)\int_{B_1\cap\set{u<M}}\abs{x}^{\beta}G_\lambda(x,u)\,dx.
\end{split}\end{equation*}
Recalling the expression \eqref{G} for $G_\lambda(x,u)$, we have
\begin{equation*}\begin{split}
\left(\dfrac12-\dfrac{\sigma}{N}\right)\int_{\R^N}\abs{(-\Delta)^{\frac{\gamma}{2}}{u}}^2\,dx
&\leq\lambda A{C_M}\int_{B_1\cap\set{u<M}}\abs{x}^{\beta}u^2\,dx,
\end{split}\end{equation*}
where
\begin{equation}\label{CM}
C_M=\dfrac{p}{2}\left(1+w_{\frac{\lambda^*}{2}}(0)+M\right)^{p-1}
\end{equation}
by the monotonicity properties of $w_\lambda$.

On the other hand, since $p>\frac{N}{N-2\gamma}$,
$$-\beta=-p(N-2\gamma)+(N+2\gamma)=2\gamma-(N-2\gamma)\left(p-\dfrac{N}{N-2\gamma}\right)<2\gamma,$$
and thus the fractional Hardy--Sobolev inequality \eqref{HS1} implies
\begin{equation*}\begin{split}
c\int_{\R^N}\abs{(-\Delta)^{\frac{\gamma}{2}}{u}}^2\,dx
&\geq\left(\int_{\R^N}\abs{x}^{\beta}u^{2\eta}\,dx\right)^{\frac{1}{\eta}}
=\left(\int_{B_1}\abs{x}^{\beta}u^{2\eta}\,dx\right)^{\frac{1}{\eta}},
\end{split}\end{equation*}
where
$$\eta=\dfrac{N+\beta}{N-2\gamma}=\dfrac{p(N-2\gamma)-2\gamma}{N-2\gamma}=1+\left(p-\dfrac{N}{N-2\gamma}\right)>1.$$
Hence,
$$\left(\int_{B_1}\abs{x}^{\beta}u^{2\gamma}\,dx\right)^{\frac{1}{\gamma}}
\leq\dfrac{2N}{N-2\gamma}
c\,C_{M}\lambda A\int_{B_1}\abs{x}^{\beta}u^2\,dx.$$
However, by H\"older's inequality, we have
\begin{equation*}\begin{split}
\int_{B_1}\abs{x}^{\beta}u^2\,dx
&=\int_{B_1}\abs{x}^{\frac{\beta}{\eta}}u^2\cdot\abs{x}^{\beta\left(1-\frac{1}{\eta}\right)}\,dx
\leq\left(\int_{B_1}\abs{x}^{\beta}u^{2\eta}\,dx\right)^{\frac{1}{\eta}}\left(\int_{B_1}\abs{x}^{\beta}\,dx\right)^{1-\frac{1}{\eta}}\\
&\leq(N+\beta)^{-\frac{N+2\gamma}{N+\beta}}\left(\int_{B_1}\abs{x}^{\beta}u^{2\eta}\,dx\right)^{\frac{1}{\eta}}.
\end{split}\end{equation*}
Therefore, we have
$$\left(\int_{B_1}\abs{x}^{\beta}u^{2\eta}\,dx\right)^{\frac{1}{\eta}}
\leq\dfrac{2Nc A C_M}{(N-2\gamma)(N+\beta)^{\frac{N+2\gamma}{N+\beta}}}
\lambda\left(\int_{B_1}\abs{x}^{\beta}u^{2\eta}\,dx\right)^{\frac{1}{\eta}},$$
which forces $u\equiv0$ for any
\begin{equation}\label{lambda0}
\lambda<\lambda_0=\left(\dfrac{2NcAC_M}{(N-2\gamma)(N+\beta)^{\frac{N+2\gamma}{N+\beta}}}\right)^{-1}.
\end{equation}
\end{proof}

\subsection{Existence of a fast-decay singular solution}

Consider the Banach space of bounded, continuous, non-negative and radially non-increasing functions supported in the unit ball,
$$E=\set{w\in{\mathcal C(\R^n)}:w(x)
=\tilde{w}(\abs{x}),~\tilde{w}(r_1)\leq\tilde{w}(r_2)\text{ for } r_1\geq r_2,~w\geq0~\text{in}~B_1~\text{and}~w\equiv0~\text{in}~\R^N\setminus{B_1}}.$$
We begin with an {\em a priori} estimate followed by a compactness result, from which a bifurcation argument follows.

In the following an equivalent integral formulation of \eqref{problem} will be useful. Using the Green's function for the Dirichlet problem in the unit ball (\cite{Riesz,Bucur}), we see that \eqref{problem} is equivalent to
\begin{equation}\label{eq:T}
w(x)=T(\lambda,w)(x):=\int_{B_1}G(x,y)\lambda A\abs{y}^{\beta}(1+w(y))^p\,dy,
\quad x\in B_1,
\end{equation}
where
$$G(x,y)=C(N,\gamma)\dfrac{1}{\abs{x-y}^{N-2\gamma}}\int_0^{r_0(x,y)}\dfrac{r^{\gamma-1}}{(r+1)^{\frac{N}{2}}}\,dr$$
with
$$r_0(x,y)=\dfrac{(1-\abs{x}^2)(1-\abs{y}^2)}{\abs{x-y}^{2}}.$$
Here $C(N,\gamma)$ is some normalizing constant. $T$ is a continuous operator from $\R\times{E}$ to $E$.

\begin{lemma}[Uniform bound]\label{uniformbound}
There exists a universal constant $C_0=C_0(N,\gamma,p,\lambda^*)$ such that for any function $w\in{E}$, solving \eqref{problem} and for any $x\in B_{1/2}(0)\setminus\set{0}$,
$$w(x)\leq{C_0}\abs{x}^{-\frac{\beta+2\gamma}{p-1}}=C_0\abs{x}^{-\frac{p(N-2\gamma)-N}{p-1}}.$$
\end{lemma}

\begin{proof}
The maximum principle implies that $w>0$ in $B_1$. Let
$$y\in{B_{\frac{\abs{x}}{4}}\left(\dfrac{3x}{4}\right)}
\subset{B_{\frac{\abs{x}}{2}}(x)}\cap{B_{\abs{x}}(0)}\subset{B_1(0)}.$$
From $y\in{B_{\frac{\abs{x}}{2}}(x)}$, we have
$$\abs{x-y}\leq\dfrac{\abs{x}}{2}\leq\dfrac14\quad\text{and}\quad\abs{y}\leq\dfrac{3\abs{x}}{2}\leq\dfrac34$$
and so
$$r_0(x,y)\geq\dfrac{\left(1-\frac14\right)\left(1-\frac{9}{16}\right)}{\frac{1}{16}}\geq\dfrac{21}{4}>5.$$
On the other hand, since $y\in{B_{\abs{x}}(0)}$ and $w$ is radially non-increasing, we have
$$\abs{y}^{\beta}\geq\abs{x}^{\beta}\quad\text{and}\quad{w(y)}\geq{w(x)}.$$
Therefore, we may conclude
$$G(x,y)\geq{C(N,\gamma)}\left(\dfrac{2}{\abs{x}}\right)^{N-2\gamma}\int_0^5\dfrac{r^{\gamma-1}}{(r+1)^{\frac{N}{2}}}\,dr$$
and
\begin{equation*}\begin{split}
w(x)
&\geq A\int_{B_{\frac{\abs{x}}{4}}\left(\frac{3x}{4}\right)}C(N,\gamma)\dfrac{2^{N-2\gamma}}{\abs{x}^{N-2\gamma}}
\left(\int_0^5\dfrac{r^{\gamma-1}}{(r+1)^{\frac{N}{2}}}\,dr\right)\lambda_0\abs{x}^{\beta}w(x)^p\,dy\\
&\geq{C(N,\gamma)}A 2^{N-2\gamma}\left(\int_0^5\dfrac{r^{\gamma-1}}{(r+1)^{\frac{N}{2}}}\,dr\right)\lambda_0
\cdot\dfrac{\abs{x}^{\beta}}{\abs{x}^{N-2\gamma}}w(x)^p
\cdot\abs{B_1}\left(\dfrac{\abs{x}}{4}\right)^N\\
&\geq{C_0^{-(p-1)}}\abs{x}^{\beta+2\gamma}w(x)^p,
\end{split}\end{equation*}
where
\[C_0^{-(p-1)}=\dfrac{C(N,\gamma)\abs{B_1}A\lambda_0}{2^{N+2\gamma}}\int_0^5\dfrac{r^{\gamma-1}}{(r+1)^{\frac{N}{2}}}\,dr.\]
The inequality clearly rearranges to
\[w(x)\leq{C_0}\abs{x}^{-\frac{\beta+2\gamma}{p-1}},\]
as desired. The dependence of the constant $C_0$ follows from \eqref{lambda0} and \eqref{CM}.
\end{proof}

\begin{lemma}[Compactness]\label{compactness}
The non-linear operator $T:\R\times E\to E$ defined in \eqref{eq:T} is compact, i.e. it maps bounded sets to relatively compact sets.
\end{lemma}

\begin{proof}
In a bounded set of $\R\times E$, it suffices to find a convergent subsequence via Arzel\`{a}-Ascoli theorem. Since $p>1$, equi-boundedness and equi-continuity follow immediately once we have the bound

\[
\int_{B_1}\dfrac{|y|^{\beta}}{|x-y|^{N-2\gamma}}\,dy
    \leq{C},
\]
for any $|x|<1$. Indeed, it is true because $2\gamma+\beta>0$. We prove it for the case where $N\geq2$; the lower dimensional case is easier and omitted. Using polar coordinates with $r=|y|$ and $\theta_1\in(0,\pi)$, and then changing variables to $r=|x|\rho$, $t=t_0\sin\frac{\theta_1}{2}$ with $t_0=\frac{2\sqrt{r|x|}}{\abs{r-|x|}}=\frac{2\sqrt{\rho}}{|\rho-1|}$, we compute
\[\begin{split}
\int_{B_1}\dfrac{|y|^{\beta}}{|x-y|^{N-2\gamma}}\,dy
&=\abs{\mathbb{S}^{N-2}}\int_0^1r^{N-1+\beta}
    \int_0^\pi\dfrac{\sin^{N-2}\theta_1\,d\theta_1}
        {\left((r-|x|)^2+2r|x|(1-\cos\theta_1)\right)^{\frac{N-2\gamma}{2}}}\,dr\\
&=\abs{\mathbb{S}^{N-2}}
    \int_0^1\dfrac{r^{N-1+\beta}}
        {\abs{r-|x|}^{N-2\gamma}}
    \int_0^{t_0}\dfrac{
        2^{N-2}\left(\frac{t}{{t_0}}\right)^{N-2}
        \left(1-\frac{t^2}{t_0^2}\right)^{\frac{N-3}{2}}
        \,\left(\frac{2}{t_0}dt\right)}
    {\left(1+t^2\right)^{\frac{N-2\gamma}{2}}}\,dr\\
&=\abs{\mathbb{S}^{N-2}}|x|^{2\gamma+\beta}
    \int_0^{\frac{1}{|x|}}
        \dfrac{\rho^{N-1+\beta}}{\abs{\rho-1}^{N-2\gamma}}
        \left(\dfrac{2}{t_0}\right)^{N-1}
        \int_0^{t_0}\dfrac{t^{N-2}}{(1+t^2)^{\frac{N-2\gamma}{2}}}
            \left(1-\frac{t^2}{t_0^2}\right)^{\frac{N-3}{2}}\,dtd\rho\\
&\leq{C}|x|^{2\gamma+\beta}
    \int_{0}^{\frac{1}{|x|}}
        \dfrac{\rho^{N-1+\beta}}{|\rho-1|^{N-2\gamma}}
        \min\set{1,\left(\frac{\sqrt{\rho}}{|\rho-1|}\right)^{1-N}
        	\int_1^{\frac{\sqrt{\rho}}{|\rho-1|}}
				t^{2\gamma-2}\,dt
		}
    \,d\rho
    \\
&\leq{C}|x|^{2\gamma+\beta}
    \left(\int_{0}^{\frac12}
        \rho^{N-1+\beta}
    \,d\rho
    +\int_{\frac12}^2
        |\rho-1|^{2\gamma-1}
        \left[\int_{1}^{\frac{1}{|\rho-1|}}t^{2\gamma-2}\,dt\right]
    \,d\rho
    +\int_2^{\frac{1}{|x|}}
        \rho^{2\gamma-1+\beta}\,d\rho
    \right)
    \\
&\leq{C},
\end{split}\]
provided that $|x|\leq\frac12$. When $\frac12<|x|<1$, we simply use
\[
\dfrac{|y|^{\beta}}{|x-y|^{N-2\gamma}}
\leq{C}
	\left(
		|y|^{\beta}+|x-y|^{-(N-2\gamma)}+1
	\right).
\]
This proves the claim.

Therefore, if $(\lambda,u)$ and $(\mu,v)$ belong to a bounded set of $\R\times E$, then
\begin{equation*}\begin{split}
\norm[L^\infty]{T(\lambda,u)-T(\mu,v)}
&\leq
	|\lambda-\mu|\int_{B_1}
		\dfrac{A|y|^{\beta}(1+\norm[L^\infty]{u})^p}{|x-y|^{N-2\gamma}}\,dy\\
&\quad\;
	+\norm[L^{\infty}]{u-v}
	\int_{B_1}
		|\mu|Ap\big(1+\norm[L^\infty]{u}+\norm[L^\infty]{v}\big)^{p-1}
		\dfrac{|y|^{\beta}}{|x-y|^{N-2\gamma}}\,dy\\
&\leq{C}\left(|\lambda-\mu|+\norm[L^\infty]{u-v}\right).
\end{split}\end{equation*}
Hence, Arzel\`{a}-Ascoli theorem applies and the proof is complete.
\end{proof}

\begin{lemma}[Bifurcation]\label{bifurcation}
There exists a sequence of solutions $(\lambda_j,w_j)$ of \eqref{problem} in $(0,\lambda^*]\times{E}$ such that
$$\lim_{j\to\infty}\lambda_j=\lambda_\infty\in[\lambda_0,\lambda^*]\quad\text{and}\quad
\lim_{t\to\infty}\norm[L^\infty]{w_j}
=\infty,$$
where $\lambda_0$ is given in Proposition \ref{uniqueness}.
\end{lemma}

\begin{proof}
Consider the continuation $$\mathscr{C}=\set{(\lambda(t),w(t)):t\geq0}$$
of the branch of minimal solutions $\set{(\lambda,w_\lambda):0\leq\lambda\leq\lambda_0}$, where $(\lambda(0),w(0))=(\lambda_0,w_{\lambda_0})$. By Proposition \ref{uniqueness} and the fact that no solutions exist if $\lambda>\lambda^*$, we see that $\mathscr{C}\subset[\lambda_0,\lambda^*]\times{E}$. By the classical bifurcation theory \cite[Theorem 6.2]{Rabinowitz}, $\mathscr{C}$ is unbounded in $[\lambda_0,\lambda^*]\times{E}$ and the existence of the desired sequence of pairs $(\lambda_j,w_j)$ follows.
\end{proof}

We are ready to establish the existence of a fast-decay singular solution.

\begin{proof}[Proof of Proposition \ref{existence}]
Let $(\lambda_j,w_j)$ be as in Lemma \ref{bifurcation}. Define
\begin{equation*}
m_j=\norm[L^\infty(B_1)]{w_j}=w_j(0)\quad\text{and}\quad{R_j}=m_j^{\frac{p-1}{\beta+2\gamma}}=m_j^{\frac{p-1}{p(N-2\gamma)-N}}
\end{equation*}
so that $m_j$, $R_j\to\infty$ as $j\to\infty$. Set also
\begin{equation*}
W_j(x)=\lambda^{\frac{1}{p-1}}m_j^{-1}w_{\lambda_j}\Big(\frac{x}{R_j}\Big).
\end{equation*}
Then $0\leq W_j\leq 1$ and $W_j$ is a bounded solution to
\begin{equation*}\begin{cases}
(-\Delta)^\gamma{W_j}=\lambda_{j}^{\frac{1}{p-1}+1}\lambda_j^{-\frac{p}{p-1}}m_j^{p-1}R_j^{-\beta-2\gamma}A|x|^\beta
\left(\lambda_j^{\frac{1}{p-1}}m_j^{-1}+W_j\right)^p&\text{in}~B_{{R_j}}(0),\\
W_j=0 &\text{in}~\R^N\setminus{B}_{R_j}(0),
\end{cases}\end{equation*}
that is,
\begin{equation*}\begin{cases}
(-\Delta)^\gamma{W_j}=A\abs{x}^\beta\left(\lambda_j^{\frac{1}{p-1}}m_j^{-1}+W_j\right)^p&\text{in}~B_{{R_j}}(0),\\
W_j=0 &\text{in}~\R^N\setminus{B}_{R_j}(0).
\end{cases}\end{equation*}
Recall that by Lemma \ref{uniformbound},
$$w_j(x)\leq{C_0}\abs{x}^{-\frac{\beta+2\gamma}{p-1}}
\quad \text{ for } x\in B_{\frac12}(0).$$
Thus in $B_{\frac{R_j}{2}}(0)$, $W_j(x)$ has the upper bound
\begin{equation}\label{Vjbound}\begin{split}
W_j(x)
&\leq\lambda^{\frac{1}{p-1}}m_j^{-1}\cdot{C_0}\left(\dfrac{x}{R_j}\right)^{-\frac{\beta+2\gamma}{p-1}}
\leq{C_0}\left(\lambda_0^{\frac{1}{p-1}}+\left(\lambda^*\right)^{\frac{1}{p-1}}\right)\abs{x}^{-\frac{\beta+2\gamma}{p-1}}\\
&={C_1}\abs{x}^{-\frac{\beta+2\gamma}{p-1}}=C_1\abs{x}^{\frac{2\gamma}{p-1}-(N-2\gamma)}.
\end{split}\end{equation}
Note that $|x|^{\beta}\in L^q(B_{R_j}(0))$ for any $\frac{N}{2\gamma}<q<\frac{N}{-\beta}$. Hence, for such $q$, by the regularity result in \cite{RosOton-Serra2}, $W_j\in \mathcal C_{loc}^\eta(\R^N)$ for $\eta=\min\left\{\gamma,2\gamma-\frac{N}{q}\right\}\in(0,1)$. Therefore, by passing to a subsequence, $W_j$ converges uniformly on compact sets of $\R^N$ to a radially symmetric and non-increasing function $w$ which satisfies
\begin{equation*}\begin{cases}
(-\Delta)^\gamma{w}=A\abs{x}^{\beta}w^p&\text{in}~\R^N,\\
w(0)=1,\\
w(x)\leq{C_1}\abs{x}^{\frac{2\gamma}{p-1}-(n-2\gamma)},
\end{cases}\end{equation*}
in view of \eqref{Vjbound}.

Now the family of rescaled solutions $w_\eps(x)=\eps{w}\left(\eps^{\frac{p-1}{\beta+2\gamma}}x\right)$ solves
\begin{equation*}\begin{cases}
(-\Delta)^\gamma{w_\eps}=A\abs{x}^{\beta}w_\eps^p&\quad\text{in}~\R^N,\\
w_\eps(0)=\eps,\\
w_\eps(x)\leq{C_1}\abs{x}^{\frac{2\gamma}{p-1}-(N-2\gamma)}&\quad\text{in}~\R^N\setminus\set{0}.
\end{cases}\end{equation*}
Finally, its Kelvin transform $u_\eps(x)=\abs{x}^{-(N-2\gamma)}w_\eps\left(\frac{x}{\abs{x}^2}\right)$ satisfies
\begin{equation*}\begin{cases}
(-\Delta)^\gamma{u_\eps}=Au_\eps^p&\quad\text{in}~\R^N\setminus\set{0},\\
u_\eps(x)\leq{C_1}\abs{x}^{-\frac{2\gamma}{p-1}}&\quad\text{in}~\R^N\setminus\set{0},\\
u_\eps(x)\sim\eps\abs{x}^{-(N-2\gamma)}&\quad\text{as}~\abs{x}\to\infty,
\end{cases}\end{equation*}
as desired.
\end{proof}

\section{The conformal fractional Laplacian in the presence of $k$-dimensional singularities}\label{section:conformal}

\subsection{A quick review on the conformal fractional Laplacian}
Here we review some basic facts on the conformal fractional Laplacian that will be needed in the next sections
(see \cite{Chang-Gonzalez,Gonzalez:survey} for the precise definitions and details).

If $(X,g^+)$ is a $(n+1)$-dimensional conformally compact Einstein manifold (which, in particular, includes the hyperbolic space), one can define a one-parameter family of operators $P_\gamma$ of order $2\gamma$ on its conformal infinity $M^n=\partial_\infty X^{n+1}$. $P_\gamma$ is known as the conformal fractional Laplacian and it can be understood as a Dirichlet-to-Neumann operator on $M$. In the particular case that $X$ is the hyperbolic space $\mathbb H^{n+1}$, whose conformal infinity is $M=\mathbb R^n$ with the Euclidean metric, $P_\gamma$ coincides with the standard fractional Laplacian $(-\Delta_{\mathbb R^n})^\gamma$.

Let us explain this definition in detail. It is known that, having fixed a metric $g_0$ in the conformal infinity $M$,  it is possible to write the metric $g^+$ in the normal form
$g^+=\rho^{-2}(d\rho^2+g_\rho)$ in  a tubular neighborhood $M\times(0,\delta]$. Here $g_\rho$ is a one-parameter family of metrics on $M$ satisfying $g_\rho|_{\rho=0}=g_0$ and $\rho$ is a defining function in $\overline{X}$ for the boundary $M$ (i.e., $\rho$ is a non-degenerate function such that $\rho>0$ in $X$ and $\rho=0$ on $M$).

Fix $\gamma\in(0,n/2)$ not an integer such that $n/2+\gamma$ does not belong to the set of $L^2$-eigenvalues of $-\Delta_{g^+}$. Assume also that the first eigenvalue for $-\Delta_{g^+}$ satisfies  $\lambda_1(-\Delta_{g^+})>n^2/4-\gamma^2$.
It is well known from scattering theory \cite{Graham-Zworski,Guillarmou} that, given $w\in \mathcal C^{\infty}(M)$, the eigenvalue problem
\begin{equation}\label{equation-scattering}
-\Delta_{g^+}\mathcal W-\big(\tfrac{n^2}{2}-\gamma^2\big)\mathcal W=0 \text{ in } X,
\end{equation}
has a unique solution with the asymptotic expansion
\begin{equation}\label{formaw}
\mathcal W=\mathcal W_1 \rho^{\frac{n}{2}-\gamma}+ \mathcal W_2 \rho^{\frac{n}{2}+\gamma}, \quad \mathcal W_1,\mathcal W_2 \in \mathcal C^{\infty}(\overline{X})
\end{equation}
and Dirichlet condition  on $M$
\begin{equation}\label{Dirichlet-condition}
\mathcal W_1|_{\rho=0}=w.
\end{equation}
The conformal fractional Laplacian (or scattering operator, depending on the normalization constant) on $(M,g_0)$ is defined taking the Neumann data
\begin{equation}\label{definition-P}
 P^{g_0}_{\gamma}w=d_{\gamma}\mathcal W_2|_{\rho=0},\quad\text{where }\ d_{\gamma}=2^{2\gamma}\frac{\Gamma(\gamma)}{\Gamma(-\gamma)},
\end{equation}
and the fractional curvature as
$Q^{g_0}_\gamma:=P^{g_0}_\gamma(1).$

$P_\gamma^{g_0}$ is a self-adjoint pseudodifferential operator of order $2\gamma$ on $M$ with the same principal symbol as $(-\Delta_M)^\gamma$. In the case that the order of the operator is an even integer  we recover the conformally invariant GJMS operators on $M$. In addition, for any $\gamma\in(0,\frac{n}{2})$, the operator is conformal. Indeed,
\begin{equation}\label{conformal-property}
 P_{\gamma}^{g_{w}}f=w^{-\frac{n+2\gamma}{n-2\gamma}}P_{\gamma}^{g_0}(wf), \quad \forall f\in \mathcal C^{\infty}(M),
\end{equation}
for a change of metric $$g_{w}:=w^{\frac{4}{n-2\gamma}}g_0,\ w>0.$$
Moreover, \eqref{conformal-property} yields the $Q_\gamma$ curvature equation
\begin{equation*}\label{confQ}
 P_{\gamma}^{g_0}(w)=w^{\frac{n+2\gamma}{n-2\gamma}}Q_{\gamma}^{g_{w}}.
\end{equation*}
Explicit formulas for $P_{\gamma}$ are not known in general. The formula for the cylinder will be given in Section \ref{section:isolated-singularity}, and it is one of the main ingredients for the linear theory arguments of Section \ref{section:linear-theory}.

\bigskip

The extension \eqref{equation-scattering} takes a more familiar form under a conformal change of metric.

\begin{proposition}[\cite{Chang-Gonzalez}]\label{prop:extension}
Fix $\gamma\in(0,1)$ and $\bar g=\rho^2 g^+ $. Let $\mathcal W$ be the solution to the scattering problem \eqref{equation-scattering}-\eqref{formaw} with Dirichlet data \eqref{Dirichlet-condition} set to $w$.  Then $W = \rho^{-n/2+\gamma}  \mathcal W$ is the unique solution to the extension problem
\begin{equation}\label{div}
\left\{\begin{array}{r@{}l@{}l}
-\divergence \left( \rho^{1-2\gamma} \nabla W\right) + E_{\bar g}(\rho) W&\,=0\quad &\mbox{in }(X,\bar g), \medskip\\
W|_{\rho=0}&\,=w\quad &\mbox{on }M,
\end{array}\right.
\end{equation}
where the derivatives are taken with respect to the metric $\bar g$, and the zero-th order term is given by
\begin{equation}\label{Erho}
\begin{split}
E_{\bar g}(\rho)&= -\Delta_{\bar{g}}(\rho^{\frac{1-2\gamma}{2}})\rho^{\frac{1-2\gamma}{2}}
+\left(\gamma^2-\tfrac{1}{4}\right)\rho^{-(1+2\gamma)}+\tfrac{n-1}{4n}R_{\bar {g}}\rho^{1-2\gamma}\\
&=\rho^{-\frac{n}{2}-\gamma-1}\left\{-\Delta_{g^+} -\big(\tfrac{n^2}{4}-\gamma^2\big)\right\}\big(\rho^{\frac{n}{2}-\gamma}\big).
\end{split}\end{equation}
Moreover, we recover the conformal fractional Laplacian as
\begin{equation*}
P_\gamma^{g_0} w=
-\tilde d_\gamma\lim_{\rho\to 0}\rho^{1-2\gamma} \partial_\rho W,
\end{equation*}
where
\begin{equation}\label{tilde-d}\tilde d_\gamma =-\frac{d_\gamma}{2\gamma}=-
\frac{2^{2\gamma-1}\Gamma(\gamma)}{\gamma\Gamma(-\gamma)}.
\end{equation}
\end{proposition}

We also recall the following result, which allows us to rewrite \eqref{div} as a pure divergence equation with no zeroth order term. The more general statement can be found in Lemma \ref{cg-modified}, and it will be useful in the calculation of the Hamiltonian from Section \ref{subsection:Hamiltonian}.

\begin{proposition}[\cite{Chang-Gonzalez,Case-Chang}]\label{prop:new-defining-function} Fix $\gamma\in(0,1)$.
Let $\mathcal W^0$ be the solution to \eqref{equation-scattering}-\eqref{formaw} with Dirichlet data \eqref{Dirichlet-condition} given by $w\equiv 1$, and set $\rho^* = (\mathcal W^0)^\frac 1{n/2-\gamma}$. The function $\rho^*$ is a defining function of $M$ in $X$ such that, if we define the metric $\bar g^*=(\rho^*)^2 g^+$, then
$E_{\bar g^*}(\rho^*)\equiv 0$. Moreover, $\rho^*$ has the asymptotic expansion near the conformal infinity
$$\rho^*(\rho)=\rho\Big[ 1+\frac{Q^{g_0}_\gamma}{(n/2-\gamma)d_\gamma}  \rho^{2\gamma}+O(\rho^2)\Big].$$
By construction, if $W^*$ is the solution to
\begin{equation*}
\left\{
\begin{array}{r@{}l@{}l}
-\divergence \left( (\rho^*)^{1-2\gamma} \nabla W^*\right)&\,=0 \quad &\mbox{in }(X, \bar g^*),\medskip\\
W^*&\,=w \quad&\mbox{on }(M,g_0),
\end{array}
\right.
\end{equation*}
with respect to the metric $\bar g^*$, then
 \begin{equation*}
  P_\gamma^{g_0} w=
-\tilde d_\gamma \lim_{{\rho^*}\to 0} (\rho^*)^{1-2\gamma}\,\partial_{\rho^*} W^*+w\,
Q_\gamma^{g_0}.\end{equation*}
\end{proposition}

\begin{remark}
In the particular case that  $X=\mathbb R^{n+1}_+=\{(x,\ell):x\in\mathbb R^n,\ell>0\}$ is hyperbolic space $\mathbb H^{n+1}$ with the metric $g^+=\frac{d\ell^2+|dx|^2}{\ell^2}$ and $M=\mathbb R^n$, this is just the construction for the fractional Laplacian $(-\Delta_{\mathbb R^n})^\gamma$ as a Dirichlet-to-Neumann operator for a degenerate elliptic extension problem from \cite{Caffarelli-Silvestre}. Indeed, let $U$ be the solution to
\begin{equation}\label{CS1}
\left\{\begin{array}{r@{}l@{}l}
\partial_{\ell\ell}U+\dfrac{1-2\gamma}{\ell}\partial_\ell U+\Delta_{\mathbb R^n} U&\,=0\quad &\mbox{in }\mathbb R^{n+1}_+, \medskip\\
U|_{\ell=0}&\,=u\quad &\mbox{on }\mathbb R^n,
\end{array}\right.
\end{equation}
then
\begin{equation}\label{CS2}
(-\Delta_{\mathbb R^n})^\gamma u=
-\tilde d_\gamma\lim_{\ell\to 0}\ell^{1-2\gamma} \partial_\ell U.
\end{equation}
From now on, $(X,g^+)$ will be fixed to be hyperbolic space with its standard metric. Our point of view in this paper is to rewrite this extension problem \eqref{CS1}-\eqref{CS2} using different coordinates for the hyperbolic metric in $X$, such as \eqref{hyperbolic-metric-rho0}.
\end{remark}

\subsection{An isolated singularity}\label{section:isolated-singularity}

Before we go to the general problem, let us look at positive  solutions to
\begin{equation}\label{equation-conformal-power}
(-\Delta_{\mathbb R^N})^{\gamma} u=\Lambda_{N,\gamma} u^{\frac{N+2\gamma}{N-2\gamma}} \quad\text{in}\quad \mathbb R^N\setminus\{0\}
\end{equation}
that have an isolated singularity at the origin. It is known (\cite{CaffarelliJinSireXiong}) that such solutions have the asymptotic behavior near the origin like $r^{-\frac{N-2\gamma}{2}}$, for $r=|x|$. Thus it is natural to write
\begin{equation}\label{uw}
u=r^{-\frac{N-2\gamma}{2}}w.
\end{equation}

Note that the power of the nonlinearity in the right hand side of \eqref{equation-conformal-power} is chosen so that the equation has good conformal properties. Indeed, let $r=e^{-t}$ and $\theta\in \mathbb S^{N-1}$ and write the Euclidean metric in $\mathbb R^N$ as
$$|dx|^2=dr^2+r^2g_{\mathbb S^{N-1}}$$
in polar coordinates. We use conformal geometry to rewrite equation \eqref{equation-conformal-power}. For this, consider the conformal change
$$g_0:=\frac{1}{r^2} |dx|^2=dt^2+g_{\mathbb S^{N-1}},$$
which is a complete metric defined on the cylinder $M_0:=\mathbb R\times \mathbb S^{N-1}$.
The advantage of using $g_0$ as a background metric instead of the Euclidean one on $\mathbb R^N$ is the following: since the two metrics are conformally related, any conformal change  may be rewritten as
$$\tilde g=u^{\frac{4}{N-2\gamma}}|dx|^2=w^{\frac{4}{N-2\gamma}}g_0,$$
where we have used relation \eqref{uw}. Then, looking at the conformal transformation property \eqref{conformal-property} for the conformal fractional Laplacian $P_\gamma$, it is clear that
\begin{equation}\label{conjugation1}
P_\gamma^{g_0}(w)=r^{\frac{N+2\gamma}{2}} P^{|dx|^2}_{\gamma}(r^{-\frac{N-2\gamma}{2}} w)=r^{\frac{N+2\gamma}{2}}
(-\Delta_{\mathbb R^N})^{\gamma} u,
\end{equation}
and thus equation \eqref{equation-conformal-power} is equivalent to
\begin{equation*}
P_\gamma^{g_0}(w)=\Lambda_{N,\gamma}w^{\frac{N+2\gamma}{N-2\gamma}}\quad\text{in}\quad \mathbb R\times \mathbb S^{N-1}.
\end{equation*}
The operator $P_\gamma^{g_0}$ on  $\mathbb R\times \mathbb S^{N-1}$ is explicit. Indeed, in \cite{DelaTorre-Gonzalez} the authors calculate its principal symbol using the spherical harmonic decomposition for $\mathbb S^{N-1}$. With some abuse of notation, let $\mu_m$, $m=0,1,2,\dots$ be the eigenvalues of $\Delta_{\mathbb S^{N-1}}$, repeated according to multiplicity (this is, $\mu_0=0$, $\mu_1,\ldots,\mu_N=N-1$,\dots). Then any function on $\mathbb R\times \mathbb S^{N-1}$ may be decomposed as $\sum_{m} w_m(t) E_m$, where $\{E_m(\theta)\}$ is the corresponding basis of eigenfunctions. The operator $P_\gamma^{g_0}$ diagonalizes under such eigenspace decomposition, and moreover, it is possible to calculate the Fourier symbol of each projection. Let
\begin{equation}\label{fourier}
\hat{w}(\xi)=\frac{1}{\sqrt{2\pi}}\int_{\mathbb R}e^{-i\xi \cdot t} w(t)\,dt
\end{equation}
be our normalization for the one-dimensional Fourier transform.

\begin{proposition}[\cite{DelaTorre-Gonzalez}]\label{prop:symbol}
 Fix $\gamma\in (0,\tfrac{N}{2})$ and let $P^m_{\gamma}$ be the projection of the operator $P^{g_0}_\gamma$ over each eigenspace $\langle E_m\rangle$. Then
$$\widehat{P_\gamma^m (w_m)}=\Theta^m_\gamma(\xi) \,\widehat{w_m},$$
and this Fourier symbol is given by
\begin{equation}\label{symbol-isolated}
\Theta^m_{\gamma}(\xi)=2^{2\gamma}\frac{\Big|\Gamma\Big(\frac{1}{2}+\frac{\gamma}{2}
+\frac{1}{2}\sqrt{(\frac{N}{2}-1)^2+\mu_m}+\frac{\xi}{2}i\Big)\Big|^2}
{\Big|\Gamma\Big(\frac{1}{2}-\frac{\gamma}{2}+\frac{1}{2}\sqrt{(\frac{N}{2}-1)^2+\mu_m}
+\frac{\xi}{2}i\Big)\Big|^2}.
\end{equation}
\end{proposition}

\begin{proof} Let us give some ideas in the proof because they will be needed in the next subsections. It is inspired in the calculation of the Fourier symbol for the conformal fractional Laplacian on the sphere $\mathbb S^n$  (see the survey \cite{Gonzalez:survey}, for instance). The method is, using spherical harmonics, to reduce the scattering equation \eqref{equation-scattering} to an ODE. For this, we go back to the scattering theory definition for the fractional Laplacian and use different coordinates for the hyperbolic metric $g^+$. More precisely,
\begin{equation}\label{hyperbolic-metric-rho0}
g^+=\rho^{-2}\Big\{\rho^2+\Big(1+\tfrac{\rho^2}{4}\Big)^2 dt^2+\Big(1-\tfrac{\rho^2}{4}\Big)^2 g_{\s^{N-1}}\Big\},\quad\bar g=\rho^2 g^+,
\end{equation}
where $\rho\in(0,2)$, $t\in\mathbb R$. The conformal infinity $\{\rho=0\}$ is precisely the cylinder $(\mathbb R\times \mathbb S^{N-1},g_0)$. Actually, for the particular calculation here it is better to use the new variable $\sigma=-\log(\rho/2)$, and write
\begin{equation}\label{hyperbolic-metric-sigma-k0}
g^+=d\sigma^2+(\cosh \sigma)^2 dt^2+(\sinh \sigma)^2 g_{\mathbb S^{N-1}}.
\end{equation}
Using this metric, the scattering equation \eqref{equation-scattering} is
\begin{equation}\label{equation30}
\partial_{{\sigma}{\sigma}}\mathcal W+R(\sigma)\partial_{\sigma}\mathcal W+
(\cosh{\sigma})^{-2}\partial_{tt}\mathcal W+(\sinh \sigma)^{-2}\Delta_{\s^{N-1}}\mathcal W+\big(\tfrac{N^2}{4}-\gamma^2\big)\mathcal W=0,
\end{equation}
where $\mathcal W=\mathcal W({\sigma},t,\theta)$, $\sigma\in(0,\infty)$, $t\in\mathbb R$, $\theta\in\mathbb S^{N-1}$, and $$R(\sigma)=\frac{\partial_{\sigma}\big(\cosh{\sigma}\sinh^{N-1}{\sigma}\big)}
{\cosh{\sigma}\sinh^{N-1}{\sigma}}.$$
After projection over spherical harmonics, and Fourier transform in $t$, the solution to equation \eqref{equation30} maybe written as
\begin{equation*}
\widehat{\mathcal W_m}=\widehat{w_m}\,\varphi(\tau),
\end{equation*}
where we have used the change of variable $\tau=\tanh(\sigma)$ and $\varphi:=\varphi^{(m)}$
is a solution to the boundary value problem
\begin{equation*}
\begin{cases}
(1-\tau^2)\partial_{\tau\tau}\varphi+
\left(\tfrac{N-1}{\tau}-\tau\right)\partial_\tau\varphi+\Big[-\mu_m\tfrac{1}{\tau^2}
+(\tfrac{n^2}{4}-\gamma^2)\tfrac{1}
{1-\tau^2}-\xi^2\Big]\varphi =0,\medskip\\
\text{has the expansion \eqref{formaw} with }w\equiv 1\text{ near the conformal infinity }\set{\tau=1},  \medskip\\
\varphi \text{ is regular at }\tau=0.
\end{cases}
\end{equation*}
This is an ODE that can be explicitly solved in terms of hypergeometric functions, and indeed,
\begin{equation}\label{varphi0}
\begin{split}
\varphi(\tau)
&=(1+\tau)^{\frac{N}{4}-\frac{\gamma}{2}}(1-\tau)^{\frac{N}{4}-\frac{\gamma}{2}}\tau^{1-\frac{N}{2}
+\sqrt{\big(\tfrac{N}{2}-1\big)^2+\mu_m}}
\cdot\Hyperg(a,b;a+b-c+1;1-\tau^2)\\
&+S (1+\tau)^{\frac{N}{4}+\frac{\gamma}{2}}(1-\tau)^{\frac{N}{4}+\frac{\gamma}{2}}\tau^{1-\frac{N}{2}
+\sqrt{\big(\tfrac{N}{2}-1\big)^2+\mu_m}}
\cdot\Hyperg(c-a,c-b;c-a-b+1;1-\tau^2),
\end{split}
\end{equation}
where
\begin{equation*}
S(\xi)=\frac{\Gamma(-\gamma)}{\Gamma(\gamma)}
\frac{\Big|\Gamma\Big(\frac{1}{2}+\frac{\gamma}{2}
+\frac{1}{2}\sqrt{(\frac{N}{2}-1)^2+\mu_m}+\frac{\xi}{2}i\Big)\Big|^2}
{\Big|\Gamma\Big(\frac{1}{2}-\frac{\gamma}{2}+\frac{1}{2}\sqrt{(\frac{N}{2}-1)^2+\mu_m}
+\frac{\xi}{2}i\Big)\Big|^2},
\end{equation*}
and
\begin{equation*} a=\tfrac{-\gamma}{2}+\tfrac{1}{2}+\tfrac{1}{2}\sqrt{(\tfrac{N}{2}-1)^2+\mu_m}+i\tfrac{\xi}{2},\
b=\tfrac{-\gamma}{2}+\tfrac{1}{2}+\tfrac{1}{2}\sqrt{(\tfrac{N}{2}-1)^2+\mu_m}-i\tfrac{\xi}{2},\
c=1+\sqrt{(\tfrac{N}{2}-1)^2+\mu_m}.
\end{equation*}
The Proposition follows by looking at the Neumann condition in the expansion \eqref{formaw}.
\end{proof}

The interest of this proposition will become clear in Section \ref{section:linear-theory}, where we calculate the indicial roots for the linearized problem. It is also the crucial ingredient in the calculation of the Green's function for the fractional Laplacian with Hardy potential in Section \ref{section:Hardy}.

We finally recall the fractional Hardy's inequality in $\mathbb R^{N}$ (\cite{Prehistory-Hardy,Yafaev,Beckner:Pitts,Frank-Lieb-Seiringer:Hardy})
\begin{equation}\label{Hardy-inequality}
\int_{\R^{N}}u(-\Delta_{\mathbb R^N})^\gamma u  \,dx\geq \Lambda_{N,\gamma}\int_{\R^{N}}\frac{u^2}{r^{2\gamma}}\,dx,
\end{equation}
where $\Lambda_{N,\gamma}$ is the Hardy constant given by
\begin{equation}\label{Hardy-constant}
\Lambda_{N,\gamma}=2^{2\gamma}\frac{\Gamma^2(\frac{N+2\gamma}{4})}{\Gamma^2(\frac{N-2\gamma}{4})}
=\Theta_\gamma^0(0).
\end{equation}
Under the conjugation \eqref{uw}, inequality \eqref{Hardy-inequality} is written as
\begin{equation*}\label{Hardy-inequality2}
\int_{\R\times\mathbb S^{N-1}}wP_\gamma^{g_0} w  \,dtd\theta\geq \Lambda_{N,\gamma}\int_{\R\times\mathbb S^{N-1}}w^2\,dtd\theta.
\end{equation*}

\subsection{The full symbol} \label{subsection:full-symbol}

Now we consider the singular Yamabe problem \eqref{problem-Yamabe} in $\mathbb R^n\setminus\mathbb R^k$. This particular case is important because it is the model for a general higher dimensional singularity  (see \cite{Jin-Queiroz-Sire-Xiong}).

As in the introduction, set $N:=n-k$. We define the coordinates $z=(x,y)$, $x\in\mathbb R^N$, $y\in\mathbb R^k$ in the product space $\mathbb R^n\setminus\mathbb R^k=(\mathbb R^{n-k}\setminus\{0\})\times\mathbb R^k$. Sometimes we will consider polar  coordinates for $x$, which are
$$r=|x|=\dist(\cdot,\mathbb R^k)\in\mathbb R_+,\,\,\theta\in\mathbb S^{N-1}.$$
We write the Euclidean metric in $\mathbb R^n$ as
$$|dz|^2=|dx|^2+|dy|^2=dr^2+r^2g_{\mathbb S^{N-1}}+|dy|^2.$$
Our model manifold $M$ is going to be given by the conformal change
\begin{equation}\label{metric-gk}
g_k:=\frac{1}{r^2} |dz|^2=g_{\mathbb S^{N-1}}+\frac{dr^2+|dy|^2}{r^2}=g_{\mathbb S^{N-1}}+g_{\mathbb H^{k+1}},
\end{equation}
which is a complete metric, singular along $\mathbb R^k$. In particular, $M:=\mathbb S^{N-1}\times \mathbb H^{k+1}$.
As in the previous case, any conformal change  may be rewritten as
$$\tilde g=u^{\frac{4}{n-2\gamma}}|dz|^2=w^{\frac{4}{n-2\gamma}}g_k,$$
where we have used relation
\begin{equation*}\label{wv}
u=r^{-\frac{n-2\gamma}{2}}w,
\end{equation*}
so we may just use $g_k$ as our background metric.
As a consequence, arguing as in the previous subsection, the conformal transformation property \eqref{conformal-property} for the conformal fractional Laplacian yields that
\begin{equation}\label{relation-conjugation}
P_\gamma^{g_k}(w)=r^{\frac{n+2\gamma}{2}} P^{|dz|^2}_{\gamma}(r^{-\frac{n-2\gamma}{2}} w)=r^{\frac{n+2\gamma}{2}}
(-\Delta_{\mathbb R^n})^{\gamma} u,
\end{equation}
and thus the original Yamabe problem \eqref{problem-Yamabe} is equivalent to the following:
\begin{equation*}\label{equation2}P_\gamma^{g_k}(w)=\Lambda_{n,\gamma} w^{\frac{n+2\gamma}{n-2\gamma}}\quad \text{on}\quad M.\end{equation*}

Moreover, the expression for $P_\gamma^{g_k}$ in the metric $g_k$ (with respect to the standard extension to hyperbolic space $X=\mathbb H^{n+1}$) is explicit, and this is the statement of the following theorem. For our purposes, it will be more convenient to write the standard hyperbolic metric as
\begin{equation}\label{hyperbolic-metric-rho}
g^+=\rho^{-2}\left\{d\rho^2+\big(1+\tfrac{\rho^2}{4}\big)^2g_{\mathbb H^{k+1}}+\big(1-\tfrac{\rho^2}{4}\big)^2 g_{\s^{N-1}}\right\},
\end{equation}
for $\rho\in(0,2)$, so its conformal infinity $\{\rho=0\}$ is precisely $(M,g_k)$.

Consider the spherical harmonic decomposition for $\mathbb S^{N-1}$ as in Section \ref{section:isolated-singularity}. Then any function $w$ on $M$ may be decomposed as $w=\sum_{m} w_m E_m$, where $w_m=w_m(\zeta)$ for $\zeta\in\mathbb H^{k+1}$. We show that the operator $P^{g_k}_\gamma$ diagonalizes under such eigenspace decomposition, and moreover, it is possible to calculate the Fourier symbol for each projection. Let
$\widehat{\cdot}$ denote the Fourier-Helgason transform on $\mathbb H^{k+1}$, as described in the Appendix (section \ref{subsection:transform}).

\begin{teo}\label{thm:symbol}
 Fix $\gamma\in (0,\tfrac{n}{2})$ and let $P^m_{\gamma}$ be the projection of the operator $P_\gamma^{g_k}$ over each eigenspace $\langle E_m\rangle$. Then
$$\widehat{P_\gamma^m (w_m)}=\Theta^m_\gamma(\lambda) \,\widehat{w_m},$$
and this Fourier symbol is given by
\begin{equation}\label{symbol}
\Theta^m_{\gamma}(\lambda)=2^{2\gamma}\frac{\Big|\Gamma
\Big(\tfrac{1}{2}+\tfrac{\gamma}{2}
+\tfrac{1}{2}\sqrt{(\tfrac{N}{2}-1)^2+\mu_m}+\tfrac{\lambda}{2}i\Big)\Big|^2}
{\Big|\Gamma\Big(\tfrac{1}{2}-\tfrac{\gamma}{2}+\tfrac{1}{2}\sqrt{(\tfrac{N}{2}-1)^2+\mu_m}
+\tfrac{\lambda}{2}i\Big)\Big|^2}.
\end{equation}
\end{teo}

\begin{proof}
We follow the arguments in Proposition \ref{prop:symbol}, however, the additional ingredient here is to use Fourier-Helgason transform to handle the extra term $\Delta_{\mathbb H^{k+1}}$ that will appear.

For the calculations below it is better to use the new variable
\begin{equation*}\label{sigma_rho}
\sigma=-\log(\rho/2), \quad \rho\in(0,2),
\end{equation*}
and to
rewrite the hyperbolic metric  in $\mathbb H^{n+1}$ from \eqref{hyperbolic-metric-rho} as
\begin{equation*}\label{hyperbolic-metric-sigma}
g^+=d\sigma^2+(\cosh \sigma)^2 g_{\mathbb H^{k+1}}+(\sinh \sigma)^2 g_{\mathbb S^{N-1}},
\end{equation*}
for the variables  $\sigma\in(0,\infty)$, $\zeta\in\mathbb H^{k+1}$ and  $\theta\in\s^{N-1}$. The conformal infinity  is now $\{\sigma=+\infty\}$
and the scattering equation \eqref{equation-scattering} is written as
\begin{equation} \label{eqs}
\partial_{{\sigma}{\sigma}}\mathcal W+R(\sigma)\partial_{\sigma}\mathcal W+
(\cosh{\sigma})^{-2}\Delta_{\mathbb H^{k+1}}\mathcal W+(\sinh \sigma)^{-2}\Delta_{\s^{N-1}}\mathcal W+\big(\tfrac{n^2}{4}-\gamma^2\big)\mathcal W=0,
\end{equation}
where $\mathcal W=\mathcal W({\sigma},\zeta,\theta)$, and $$R(\sigma)=\frac{\partial_{\sigma}\big((\cosh{\sigma})^{k+1}(\sinh{\sigma})^{N-1}\big)}
{(\cosh{\sigma})^{k+1}(\sinh{\sigma})^{N-1}}.$$
The change of variable
\begin{equation}\label{cambioz}
\tau=\tanh(\sigma),\end{equation}
transforms equation \eqref{eqs} into
\begin{equation*}\label{eqz}
\begin{split}
(1-\tau^2)^2\partial_{\tau\tau}\mathcal W+\left(\tfrac{n-k-1}{\tau}+(k-1)\tau\right)(1-\tau^2)\partial_\tau \mathcal W+(1-\tau^2)\Delta_{\mathbb H^{k+1}}\mathcal W&
\\+\left(\tfrac{1}{\tau^2}-1\right)\Delta_{\s^{N-1}}\mathcal W+\big(\tfrac{n^2}{4}-\gamma^2\big)\mathcal W&=0.
\end{split}
\end{equation*}
Now we project onto spherical harmonics. This is, let $\mathcal W_m(\tau,\zeta)$ be the projection of $\mathcal W$ over the eigenspace $\langle E_m\rangle$. Then each $\mathcal W_m$ satisfies
\begin{equation}\label{equ-m}
(1-\tau^2)\partial_{\tau\tau}\mathcal W_m+\left(\tfrac{n-k-1}{\tau}+(k-1)\tau\right)\partial_\tau \mathcal W_m+\Delta_{\mathbb H^{k+1}}\mathcal W_m-\mu_m\tfrac{1}{\tau^2}\mathcal W_m
+\tfrac{\tfrac{n^2}{4}-\gamma^2}{1-\tau^2}\mathcal W_m=0.
\end{equation}
Taking the Fourier-Helgason transform in $\mathbb H^{k+1}$ we obtain
\begin{equation*}\label{extension2}
(1-\tau^2)\partial_{\tau\tau}\widehat{\mathcal W_m}+\left(\tfrac{n-k-1}{\tau}+(k-1)\tau\right)\partial_\tau \widehat{\mathcal W_m}
+\Big[-\mu_m\tfrac{1}{\tau^2}
+(\tfrac{n^2}{4}-\gamma^2)\tfrac{1}{1-\tau^2}-(\lambda^2+\tfrac{k^2}{4})\Big]\widehat{\mathcal W_m}=0
\end{equation*}
for $\widehat{\mathcal W_m}=\widehat{\mathcal W_m}(\tau, \lambda,\omega)$. Fixed $m=0,1,\ldots$, $\lambda\in\mathbb R$ and $\omega\in\mathbb S^k$,  we know that
\begin{equation*}\label{u_k}
\widehat{\mathcal W_m}=\widehat{w_m}\,\varphi_k^{\lambda}(\tau),
\end{equation*}
where $\varphi:=\varphi_k^{\lambda}(\tau)$ is the solution to the following boundary value problem:
\begin{equation}\label{problemphi}
\left\{\begin{array}{@{}l}
(1-\tau^2)\partial_{\tau\tau}\varphi+
\left(\tfrac{n-k-1}{\tau}+(k-1)\tau\right)\partial_\tau\varphi+\Big[-\mu_m\tfrac{1}{\tau^2}
+(\tfrac{n^2}{4}-\gamma^2)\tfrac{1}
{1-\tau^2}-(\lambda^2+\tfrac{k^2}{4})\Big]\varphi =0,\medskip\\
\text{has the expansion (\ref{formaw}) with }w\equiv 1\text{ near the conformal infinity }\set{\tau=1},  \medskip\\
\varphi \text{ is regular at }\tau=0.
\end{array}\right.
\end{equation}
 This is an ODE in $\tau$ that has only regular singular points, and can be explicitly solved. Indeed, from the first equation in \eqref{problemphi} we obtain
\begin{equation}\label{varphi1}
\begin{split}
\varphi(\tau)=&A(1-\tau^2)^{\frac{n}{4}-\frac{\gamma}{2}}
\tau^{1-\frac{n}{2}+\frac{k}{2}+\sqrt{(\tfrac{n-k}{2}-1)^2+\mu_m}}
\Hyperg(a,b;c;\tau^2)\\
+&B(1-\tau^2)^{\frac{n}{4}-\frac{\gamma}{2}}\tau^{1-\frac{n}{2}-\sqrt{(\tfrac{n-k}{2}-1)^2+\mu_m}}
\Hyperg(\tilde{a},\tilde{b};\tilde{c};,\tau^2),
\end{split}
\end{equation}
for any real constants $A,B$, where
\begin{eqnarray*}
&a=\tfrac{-\gamma}{2}+\tfrac{1}{2}+\tfrac{1}{2}\sqrt{(\tfrac{n-k}{2}-1)^2+\mu_m}
            +i\tfrac{\lambda}{2},
            \quad &\tilde{a}=\tfrac{-\gamma}{2}+\tfrac{1}{2}
            -\tfrac{1}{2}\sqrt{(\tfrac{n-k}{2}-1)^2+\mu_m}+i\tfrac{\lambda}{2},\\
&b=\tfrac{-\gamma}{2}+\tfrac{1}{2}+\tfrac{1}{2}\sqrt{(\tfrac{n-k}{2}-1)^2+\mu_m}
            -i\tfrac{\lambda}{2},
            \quad &\tilde{b}=\tfrac{-\gamma}{2}+\tfrac{1}{2}
            -\tfrac{1}{2}\sqrt{(\tfrac{n-k}{2}-1)^2+\mu_m}-i\tfrac{\lambda}{2},\\
&c=1+\sqrt{(\tfrac{n-k}{2}-1)^2+\mu_m},\quad \quad\qquad\,\,\,\qquad&\tilde{c}=1-\sqrt{(\tfrac{n-k}{2}-1)^2+\mu_m},
\end{eqnarray*}
and $\Hyperg$ denotes the standard hypergeometric function described in Lemma \ref{propiedadeshiperg}. Note that we can write $\lambda$ instead of $|\lambda|$ in the arguments of the hypergeometric functions because $a=\overline{b}$, $\tilde{a}=\overline{\tilde{b}}$ and property \eqref{prop5}.

The regularity at the origin $\tau=0$ implies $B=0$ in \eqref{varphi1}.
Moreover, using \eqref{prop4} we can write
\begin{equation*}\label{varphiz}
\begin{split}
\varphi(\tau)
=A&\left[\alpha(1-\tau^2)^{\frac{n}{4}-\frac{\gamma}{2}}\tau^{1-\frac{n}{2}+\frac{k}{2}
+\sqrt{(\tfrac{n-k}{2}-1)^2+\mu_m}} \right.\Hyperg(a,b;a+b-c+1;1-\tau^2)\\
&+\beta (1-\tau^2)^{\frac{n}{4}+\frac{\gamma}{2}}\tau^{1-\frac{n}{2}+\frac{k}{2}
+\sqrt{(\tfrac{n-k}{2}-1)^2+\mu_m}}
\left.\Hyperg(c-a,c-b;c-a-b+1;1-\tau^2)\right],
\end{split}
\end{equation*}
where \begin{align*}\label{alpha}
&\alpha
=\tfrac{\Gamma\Big(1+\sqrt{(\tfrac{n-k}{2}-1)^2+\mu_m}\Big)\Gamma(\gamma)}
{\Gamma\Big(\tfrac{1}{2}
+\tfrac{\gamma}{2}+\tfrac{1}{2}\sqrt{(\tfrac{n-k}{2}-1)^2+\mu_m}
-i\tfrac{\lambda}{2}\Big)\Gamma\Big(\tfrac{1}{2}
+\tfrac{\gamma}{2}+\tfrac{1}{2}\sqrt{(\tfrac{n-k}{2}-1)^2+\mu_m}+i\tfrac{\lambda}{2}\Big)},\\
&\beta=
\tfrac{\Gamma\Big(1+\sqrt{(\tfrac{n-k}{2}-1)^2+\mu_m}\Big)
\Gamma(-\gamma)}{\Gamma\left(\tfrac{1}{2}-\tfrac{\gamma}{2}
+\tfrac{1}{2}\sqrt{(\tfrac{n-k}{2}-1)^2+\mu_m}
+i\tfrac{\lambda}{2}\right)\Gamma\left(
\tfrac{1}{2}-\tfrac{\gamma}{2}+\tfrac{1}{2}\sqrt{(\tfrac{n-k}{2}-1)^2+\mu_m}
-i\tfrac{\lambda}{2}\right)}.
\end{align*}
Note that our changes of variable give
\begin{equation}\label{tau_rho}
\tau=\tanh(\sigma)=\frac{4-\rho^2}{4+\rho^2}=1-\frac{1}{2}\rho^2+\cdots,
\end{equation}
which yields, as $\rho \to 0$,
\begin{equation*}
\begin{split}
\varphi(\rho)\sim
 A\left[\alpha\rho^{\frac{n}{2}-\gamma}+\beta\rho^{\frac{n}{2}+\gamma}+\ldots\right] .
\end{split}
\end{equation*}
Here we have used \eqref{prop2} for the hypergeometric function.

Looking at the expansion for the scattering solution \eqref{formaw} and the definition of the conformal fractional Laplacian \eqref{definition-P}, we must have
\begin{equation}\label{A}
A=\alpha^{-1},\quad \text{and}\quad
\Theta^m_{\gamma}(\lambda)=d_\gamma\beta \alpha^{-1}.
\end{equation}
Property \eqref{prop1g} yields \eqref{symbol} and completes the proof of Theorem \ref{thm:symbol}.
\end{proof}

\subsection{Conjugation}\label{subsection:conjugation}
We now go back to the discussion in Section \ref{section:isolated-singularity} for an isolated singularity but we allow any subcritical power $p\in(\frac{N}{N-2\gamma},\frac{N+2\gamma}{N-2\gamma})$ in the right hand side of \eqref{equation-conformal-power}; this is,
\begin{equation}\label{problem10}
(-\Delta_{\mathbb R^N})^\gamma u=A_{N,p,\gamma}u^p\quad\text{in }\mathbb R^N\setminus\{0\}.
\end{equation}
This equation does not have good conformal properties. But, given
 $u\in \mathcal C^{\infty}(\R^{N}\setminus\{0\})$, we can consider
\begin{equation*}
u=r^{-\frac{N-2\gamma}{2}}w=r^{-\frac{2\gamma}{p-1}}v,\quad r=e^{-t},
\end{equation*}
and
define the conjugate operator
\begin{equation}\label{tilde-P}
\tilde P_\gamma^{g_0}(v):=
r^{-\frac{N-2\gamma}{2}+\frac{2\gamma}{p-1}}P_\gamma^{g_0}\big(r^{\frac{N-2\gamma}{2}-\frac{2\gamma}{p-1}}v\big)
=r^{\frac{2\gamma}{p-1}p}(-\Delta_{\mathbb R^N})^\gamma u.
\end{equation}
Then problem \eqref{problem10} is equivalent to
\begin{equation*}
\tilde P_\gamma^{g_0} (v) =A_{N,p,\gamma}v^{p}\quad\text{in }\mathbb R\times\mathbb S^{N-1},
\end{equation*}
for some $v=v(t,\theta)$ smooth, $t\in\mathbb R$, $\theta\in\mathbb S^{N-1}$.

This $\tilde P_\gamma^{g_0}$ can then be seen from the perspective of scattering theory, and thus be characterized
as a Dirichlet-to-Neumann operator for a special extension problem in Proposition \ref{prop:divV*}, as inspired by the paper of Chang and Gonz\'alez \cite{Chang-Gonzalez}. Note the Neumann condition \eqref{Neumann-V*}, which differs from the one of the standard fractional Laplacian.

In the notation of Section \ref{section:isolated-singularity}, we set $X=\mathbb H^{N+1}$ with the metric given by \eqref{hyperbolic-metric-rho0}. Its conformal infinity is $M=\mathbb R\times\mathbb S^{N-1}$ with the metric $g_0$. We would like to repeat the arguments of Section \ref{section:conformal} for the conjugate operator $\tilde P_\gamma^{g_0}$. But this operator does not have good conformal properties. In any case, we are able to define a new eigenvalue problem that replaces \eqref{equation-scattering}-\eqref{formaw}.

More precisely, let $\mathcal W$ be the unique solution to the scattering problem \eqref{equation-scattering}-\eqref{formaw}
with Dirichlet data \eqref{Dirichlet-condition} set to $w$. We define the function $\mathcal V$ by the following relation
\begin{equation}\label{changeWV}
r^{Q_0}\mathcal{W}=\mathcal{V},\quad
Q_0:=-\tfrac{N-2\gamma}{2}+\tfrac{2\gamma}{p-1},\end{equation}
Substituting into \eqref{equation-scattering}, the new scattering problem is
\begin{equation}\label{equation-scattering-modified}
-\Delta_{g^+}\mathcal V+\left(\tfrac{4+\rho^2}{4\rho}\right)^{-2}\left[-2Q_0\,\partial_t V-Q_0^2\mathcal V\right]-\big(\tfrac{N^4}{2}-\gamma^2\big)\mathcal V=0 \text{ in } X,
\end{equation}
Moreover, if we set
\begin{equation}
\label{formav}\mathcal V=\rho^{\frac{N}{2}-\gamma}\mathcal V_1+\rho^{\frac{N}{2}+\gamma}\mathcal V_2,
\end{equation}
the Dirichlet condition \eqref{Dirichlet-condition} will turn into
\begin{equation}\label{Dirichlet-V}
\mathcal V_1|_{\rho=0}=v,
\end{equation}
and the Neumann one \eqref{definition-P} into
\begin{equation}\label{Neumann-V}
d_\gamma\mathcal V_2|_{\rho=0}=\tilde P^{g_0}_\gamma(v).
\end{equation}

The following proposition is the analogous to Proposition \ref{prop:symbol} for $\tilde P_{\gamma}^{g_0}$:

\begin{proposition}\label{thm:symbolV}
 Fix $\gamma\in (0,\tfrac{n}{2})$ and let $\tilde{P}^m_{\gamma}$ be the projection of the operator $\tilde{P}_\gamma^{g_0}$ over each eigenspace $\langle E_m\rangle$. Then
$$\widehat{\tilde{P}_\gamma^m (v_m)}=\tilde{\Theta}^m_\gamma(\xi) \,\widehat{v_m},$$
and this Fourier symbol is given by
\begin{equation}\label{symbolV}
\tilde{\Theta}^m_{\gamma}(\xi)=2^{2\gamma}\frac{\Gamma \Big(\frac{1}{2}+\frac{\gamma}{2}
+\frac{\sqrt{(\frac{N}{2}-1)^2+\mu_m}}{2}
+\frac{1}{2}(Q_0+\xi i)\Big)\Gamma \Big(\frac{1}{2}+\frac{\gamma}{2}
+\frac{\sqrt{(\frac{N}{2}-1)^2+\mu_m}}{2}
-\frac{1}{2}(Q_0+\xi i)\Big)}
{\Gamma \Big(\frac{1}{2}-\frac{\gamma}{2}
+\frac{\sqrt{(\frac{N}{2}-1)^2+\mu_m}}{2}
+\frac{1}{2}(Q_0+\xi i)\Big)\Gamma \Big(\frac{1}{2}-\frac{\gamma}{2}
+\frac{\sqrt{(\frac{N}{2}-1)^2+\mu_m}}{2}
-\frac{1}{2}(Q_0+\xi i)\Big)}.
\end{equation}
\end{proposition}

\begin{proof}
We write the hyperbolic metric as \eqref{hyperbolic-metric-sigma-k0} using the change of variable $\sigma=-\log(\rho/2)$.  The scattering equation for $\mathcal W$ is \eqref{eqs} in the particular case $k=0$, $n=N$, and thus, we follow the arguments in the proof of Theorem \ref{thm:symbol}. Set $r=e^{-t}$ and project over spherical harmonics as in \eqref{equ-m}, which yields
\begin{equation}\label{extension1}\begin{split}
\partial_{{\sigma}{\sigma}}\mathcal W_m+R(\sigma)\partial_{\sigma}\mathcal W_m+
(\cosh{\sigma})^{-2}
\partial_{tt}\mathcal W_m
-(\sinh \sigma)^{-2}\mu_m \mathcal W_m+\big(\tfrac{N^2}{4}-\gamma^2\big)\mathcal W_m=0
\end{split}\end{equation}
for $$R(\sigma)=\frac{\partial_{\sigma}(\cosh{\sigma}\sinh^{N-1}\sigma)}
{\cosh{\sigma}\sinh^{N-1}\sigma}.$$
Recall the relation \eqref{changeWV} and rewrite the  extension equation \eqref{extension1} in terms of each projection $\mathcal V_m$ of $\mathcal V$. This gives
\begin{equation}\label{extension3}
\begin{split}
\partial_{{\sigma}{\sigma}}\mathcal V_m+R(\sigma)\partial_{\sigma}\mathcal V_m
+(\cosh{\sigma})^{-2}\left\{\partial_{tt}\mathcal V_m
+2 Q_0\partial_t \mathcal V_m\right.
\left.+Q_0^2\mathcal V_m\right\}&\\
-(\sinh \sigma)^{-2}\mu_m \mathcal V_m+\big(\tfrac{N^2}{4}-\gamma^2\big)\mathcal V_m&=0.
\end{split}
\end{equation}
Now we use the change of variable \eqref{cambioz}, and take Fourier transform \eqref{fourier} with respect to the variable $t$. Then
\begin{equation}\label{extension22}
(1-\tau^2)\partial_{\tau\tau}\widehat{\mathcal V_m}+\left(\tfrac{N-1}{\tau}-\tau\right)\partial_\tau \widehat{\mathcal V_m}
+\Big[-\mu_m\tfrac{1}{\tau^2}
+\big(\tfrac{N^2}{4}-\gamma^2\big)\tfrac{1}{1-\tau^2}-\left(\xi-iQ_0\right)^2\Big]\widehat{\mathcal V_m}=0.
\end{equation}
The Fourier symbol \eqref{symbolV} is obtained following the same steps as in the proof of Theorem \ref{thm:symbol}. Note that the only difference is the coefficient of $\widehat{\mathcal V}_m$ in \eqref{extension22}.

We note here than an alternative way to calculate the symbol is by taking Fourier transform in relation $\tilde P_\gamma^{g_0} (v)=e^{-Q_0t} P_\gamma^{g_0}(w)$, as follows:
\begin{equation*}
\widehat{\tilde P_\gamma^{m} v_m(t)}= \widehat{P_\gamma^{m}w_m}(\xi-iQ_0)=\Theta_\gamma^m(\xi-iQ_0)\hat w_m(\xi-iQ_0)=\Theta_\gamma^m(\xi-iQ_0)\hat v_m(\xi).
\end{equation*}
Thus $\tilde \Theta_\gamma^m(\xi)=\Theta_\gamma^m(\xi-iQ_0)$, as desired.
\end{proof}

Now we turn to Proposition \ref{prop:new-defining-function}, and we show that there exists a very special defining function adapted to $\mathcal V$.

\begin{lemma}\label{cg-modified}
Let $\gamma\in(0,1)$. There exists  a new defining function $\rho^*$ such that, if we define the metric $\bar g^*=(\rho^*)^2 g^+$, then
\begin{equation*}
E_{\bar g^*}(\rho^*)= (\rho^*)^{-(1+2\gamma)}\big(\tfrac{4\rho}{4+\rho^2}\big)^2Q_0^2,
\end{equation*} where $E_{\bar g^*}(\rho^*)$ is defined in \eqref{Erho}. The precise expression for $\rho^*$ is
\begin{equation}\label{rho*}
\rho^*(\rho)=
\left[\alpha^{-1}\big(\tfrac{4\rho}{4+\rho^2}\big)^{\tfrac{N-2\gamma}{2}}
\Hyperg\Big(\tfrac{\gamma}{p-1},\tfrac{N-2\gamma}{2}-\tfrac{\gamma}{p-1};
\tfrac{N}{2};\big(\tfrac{4-\rho^2}{4+\rho^2}\big)^2\Big)\right]^{2/(N-2\gamma)},
\quad\rho\in(0,2),
\end{equation}
where
\begin{equation*}
\alpha
=\frac{\Gamma(\frac{N}{2})\Gamma(\gamma)}
{\Gamma\big(\gamma+\frac{\gamma}{p-1}\big)
\Gamma\big(\frac{N}{2}-\frac{\gamma}{p-1}\big)}.
\end{equation*}
The function $\rho^*$ is strictly monotone with respect to $\rho$, and in particular,
$\rho^*\in(0,\rho_0^*)$ for
\begin{equation}\label{rho*0}
\rho_0^*:=\rho^*(2)=\alpha^{\frac{-2}{N-2\gamma}}.
\end{equation}
Moreover, it has the asymptotic expansion near the conformal infinity
\begin{equation}
\label{asymptotics-rho*}\rho^*(\rho)=\rho\left[ 1+O(\rho^{2\gamma})+O(\rho^2)\right].
\end{equation}
\end{lemma}
\begin{proof}
The proof follows Lemma $4.5$ in \cite{Chang-Gonzalez}. The scattering equation \eqref{equation-scattering} for $\mathcal W$ is modified to \eqref{equation-scattering-modified} when we substitute \eqref{changeWV}, but the additional terms do not affect the overall result. Then we know that, given $v\equiv 1$ on $M$, \eqref{equation-scattering-modified}
has a unique solution $\mathcal V^0$ with the asymptotic expansion
\begin{equation*}
\mathcal V^0=\mathcal V_1^0 \rho^{\frac{N}{2}-\gamma}+ \mathcal V_2^0 \rho^{\frac{N}{2}+\gamma}, \quad \mathcal V_1^0,\mathcal V_2^0 \in \mathcal C^{\infty}(\overline{X})
\end{equation*}
and Dirichlet condition  on $M=\mathbb R\times \mathbb S^{N-1}$
\begin{equation}\label{Dirichlet-V0}
\mathcal V_1^0|_{\rho=0}=1.\end{equation}
Actually, from the proof of Proposition  \ref{thm:symbolV} and the modifications of Proposition \ref{prop:symbol} we do obtain an explicit formula for such $\mathcal V^0$. Indeed, from
\eqref{varphi1} and \eqref{A} for $k=0$, $n=N$, $m=0$, replacing $i\lambda$ by $Q_0$, we arrive at
\begin{equation*}
\mathcal V^0(\tau)=\varphi(\tau)=\alpha^{-1}(1-\tau ^2)^{\frac{N}{4}-\frac{\gamma}{2}}
\Hyperg\Big(\tfrac{\gamma}{p-1},\tfrac{N-2\gamma}{2}-\tfrac{\gamma}{p-1};\tfrac{N}{2};\tau^2\Big).
\end{equation*}
Finally, substitute in the relation between $\tau$ and $\rho$ from \eqref{tau_rho}
and set
\begin{equation}\label{definition-rho*}
\rho^*(\rho)=(\mathcal V^0)^{\frac{1}{N/2-\gamma}}(\rho).
\end{equation}
Then, recalling \eqref{Erho}, for this $\rho^*$ we have
\begin{equation*}
E_{\bar g^*}(\rho^*)=(\rho^*)^{-\frac{N}{2}-\gamma-1} \left\{-\Delta_{g^+}-\big(\tfrac{N^2}{4}-\gamma^2\big)\right\}(\mathcal V^0)
=(\rho^*)^{-(1+2\gamma)}\big(\tfrac{4\rho}{4+\rho^2}\big)^2Q_0^2,
\end{equation*}
as desired. Here we have used the scattering equation for $\mathcal V^0$ from \eqref{equation-scattering-modified} and the fact that $\mathcal V^0$ does not depend on the variable $t$.

To show monotonicity, denote $\eta:=\big(\tfrac{4-\rho^2}{4+\rho^2}\big)^2$ for $\eta\in(0,1)$. It is enough to check that
\begin{equation*}
f(\eta):=(1-\eta)^{\frac{N-2\gamma}{4}} \Hyperg\Big(\tfrac{\gamma}{p-1},\tfrac{N-2\gamma}{2}-\tfrac{\gamma}{p-1};\tfrac{N}{2};
\eta\Big)
\end{equation*}
is monotone with respect to $\eta$.
From properties \eqref{prop5} and \eqref{prop6} of the Hypergeometric function  and the possible values for $p$ in \eqref{exponent-p} we can assert that
\begin{equation*}
\begin{split}
\frac{d}{d\eta}f(\eta)&=
\frac{d}{d\eta}\left((1-\eta)^{-\frac{N-2\gamma}{4}+\frac{\gamma}{p-1}}
(1-\eta)^{\frac{N-2\gamma}{2}-\frac{\gamma}{p-1}} \Hyperg\big(\tfrac{N-2\gamma}{2}-\tfrac{\gamma}{p-1},\tfrac{\gamma}{p-1};\tfrac{N}{2};\eta\big)\right)\\
&=\left(\tfrac{N-2\gamma}{4}-\tfrac{\gamma}{p-1}\right)
(1-\eta)^{-\frac{N-2\gamma}{4}+\frac{\gamma}{p-1}-1}(1-\eta)^{\frac{N-2\gamma}{2}-\frac{\gamma}{p-1}} \Hyperg\big(\tfrac{N-2\gamma}{2}-\tfrac{\gamma}{p-1},\tfrac{\gamma}{p-1};\tfrac{N}{2};\eta\big)\\
&
-\tfrac{2}{N}\left(\tfrac{N-2\gamma}{2}-\tfrac{\gamma}{p-1}\right)\left(\tfrac{N}{2}-\tfrac{\gamma}{p-1}\right)
(1-\eta)^{\frac{N-2\gamma}{2}-\frac{\gamma}{p-1}-1} \Hyperg\big(\tfrac{N-2\gamma}{2}-\tfrac{\gamma}{p-1}+1,\tfrac{\gamma}{p-1};\tfrac{N}{2}+1;\eta\big)\\
&<0.
\end{split}
\end{equation*}

\end{proof}

\begin{remark}
For the Neumann condition, note that, by construction,
\begin{equation}\label{Neumann-V0}
\tilde P_{\gamma}^{g_0}(1)=d_\gamma\mathcal V_2^0|_{\rho=0},
\end{equation}
while from \eqref{tilde-P} and the definition of $A_{N,p,\gamma}$ from \eqref{Apn},
\begin{equation*}
\tilde P_\gamma^{g_0}(1)=r^{\frac{2\gamma}{p-1}p}(-\Delta_{\mathbb R^{N}})^\gamma(r^{-\frac{2\gamma}{p-1}})=A_{N,p,\gamma}.\\
\end{equation*}
\end{remark}

The last result in this section shows that the scattering problem for $\mathcal V$ \eqref{equation-scattering-modified} can be transformed into a new extension problem as in Proposition \ref{prop:new-defining-function}, and whose Dirichlet-to-Neumann operator is precisely $\tilde P_\gamma^{g_0}$. For this we will introduce the new metric on $\R^N\setminus\{0\}$
\begin{equation}\label{g*}
\bar{g}^*=(\rho^*)^2g^+,
\end{equation}
where $\rho^*$ is the defining function defined in \eqref{rho*}, and let us denote
\begin{equation}\label{notation*}
V^*=(\rho^*)^{-(N/2-\gamma)}\mathcal V.
 \end{equation}

\begin{proposition} \label{prop:divV*}
Let $v$ be a smooth function on $M=\mathbb R\times\mathbb S^{N-1}$.
The extension problem
\begin{equation}\label{divV*}
\left\{\begin{array}{@{}r@{}l@{}l}
 -\divergence_{\bar{g}^*}((\rho^*)^{1-2\gamma}\nabla_{\bar{g}^*}V^*)
-(\rho^*)^{-(1+2\gamma)}\left(\tfrac{4\rho}{4+\rho^2}\right)^{2}2Q_0\,\partial_t V^*&\,=0\quad &\text{in } (X,\bar g^*), \medskip\\
V^*|_{\rho=0}&\,=v\quad &\text{on }(M,g_0),
\end{array}\right.
\end{equation}
has a unique solution $V^*$. Moreover, for its Neumann data,
\begin{equation}\label{Neumann-V*}
\tilde P_\gamma^{g_0}(v)=-\tilde d_\gamma \lim_{\rho^* \to 0}(\rho^*)^{1-2\gamma} \partial_{\rho^*} (V^*)+A_{N,p,\gamma}v.
\end{equation}
\end{proposition}

\begin{proof}
The
original scattering equation \eqref{equation-scattering}-\eqref{formaw} was rewritten in terms of $\mathcal V$ (recall  \eqref{changeWV}) as
\eqref{equation-scattering-modified}-\eqref{formav} with Dirichlet condition $\mathcal V_1|_{\rho=0}=v$.
Let us rewrite this equation into the more familiar form of Proposition \ref{prop:new-defining-function}. We follow the arguments in \cite{Chang-Gonzalez}; the difference comes from some additional terms that appear when changing to $\mathcal V$.

First use the definition of the classical conformal Laplacian for $g^+$ (that has constant scalar curvature $R_{g^+}=-N(N+1)$),
 $$P^{g^+}_{1}=-\Delta_{g^+}-\tfrac{N^2-1}{4},$$
and the conformal property of this operator \eqref{conformal-property} to assure that
$$P_1^{g^+}(\mathcal V)=(\rho^*)^{\frac{N+3}{2}}P_1^{\bar{g}}((\rho^*)^{-\frac{N-1}{2}}\mathcal V).$$
Using \eqref{notation*}  we can rewrite equation \eqref{equation-scattering-modified} in terms of $V^*$ as
\begin{equation*}
\begin{split}
P_1^{\bar{g}}((\rho^*)^{\frac{1-2\gamma}{2}}V^*)
+(\rho^*)^{\frac{-3-2\gamma}{2}}
\left\{\left(\tfrac{4+\rho^2}{4\rho}\right)^{-2}\left(-2Q_0\,
\partial_t V^*-Q_0^2\,V^*\right)
+\big(\gamma^2-\tfrac{1}{4}\big)V^*\right\}=0
\end{split}
\end{equation*}
or, equivalently,
using that for $\varrho:=\rho^{\frac{1-2\gamma}{2}}$,
\begin{equation*}
\varrho\Delta_{\bar g^*}(\varrho V)=\divergence_{\bar g^*}(\varrho^2 \nabla_{\bar g^*} V)+\varrho V\Delta_{\bar g^*}(\varrho),
\end{equation*}
we have
\begin{equation*}\label{eq_modified_div}
\begin{split}
-\divergence_{\bar{g}^*}((\rho^*)^{1-2\gamma}\nabla_{\bar{g}^*}V^*)+E_{\bar g^*}(\rho^*)V^*
+(\rho^*)^{-(1+2\gamma)}\left(\tfrac{4+\rho^2}{2\rho}\right)^{-2}\left(-2Q_0\,\partial_t V^*-Q_0^2\,V^*\right)=0,
\end{split}
\end{equation*}
with $E_{\bar g^*}(\rho^*)$ defined as in \eqref{Erho}. Finally, note that the defining function $\rho^*$ was chosen as in Lemma \ref{cg-modified}. This yields \eqref{divV*}.

For the boundary conditions let us recall the asymptotics \eqref{asymptotics-rho*}. The Dirichlet condition follows directly from  \eqref{Dirichlet-condition} and the asymptotics. For the Neumann condition, we recall the definition of $\rho^*$ from \eqref{definition-rho*}, so
\begin{equation*}
V^*=(\rho^*)^{-\frac{N}{2}+\gamma}\mathcal V=\frac{\mathcal V}{\mathcal V^0}=
\frac{\mathcal V_1+\rho^{2\gamma} \mathcal V_2}{\mathcal V_1^0+\rho^{2\gamma} \mathcal V_2^0},
\end{equation*}
and thus
\begin{equation*}
-\tilde d_\gamma \lim_{\rho\to 0} \rho^{1-2\gamma}\partial_\rho V^*=d_\gamma\left.\left( \mathcal V_2 \mathcal V_1^0-\mathcal V_1 \mathcal V_2^0\right)\right|_{\rho=0}=
\tilde P_\gamma^{g_0}v-A_{N,p,\gamma}v,
\end{equation*}
where we have used \eqref{Dirichlet-V} and \eqref{Neumann-V} for $\mathcal V$,  and
\eqref{Dirichlet-V0} and \eqref{Neumann-V0} for $\mathcal V^0$. This completes the proof of the Proposition.
\end{proof}

\section{New ODE methods for non-local equations}\label{section:ODE-methods}

In this section we use the conformal properties developed in the previous section to study positive singular solutions to equation
\begin{equation}\label{problem-isolated}
(-\Delta_{\mathbb R^N})^\gamma u=A_{N,p,\gamma}u^p \mbox{ in }\R^N\setminus\{0\}.
\end{equation}
The first idea is, in the notation of Section \ref{subsection:conjugation}, to set $v=r^{\frac{2\gamma}{p-1}}u$ and rewrite this equation as
\begin{equation}\label{problem-vv}
\tilde P_\gamma^{g_0}(v)=A_{N,p,\gamma}v^p, \quad\text{in } \mathbb R\times\mathbb S^{N-1},
\end{equation}
and to consider the projection over spherical harmonics in $\mathbb S^{N-1}$,
\begin{equation*}
\tilde P_\gamma^m(v_m)=A_{N,p,\gamma}(v_m)^p, \quad\text{for } v=v(t),
\end{equation*}
While in Proposition \ref{thm:symbolV} we calculated the Fourier symbol for $\tilde P_\gamma^m$, now we will write it as an integro-differential operator for a well behaved convolution kernel. The advantage of this formulation is that immediately yields regularity for $v_m$ as in \cite{DelaTorre-delPino-Gonzalez-Wei}.

Now we look at the $m=0$ projection, which corresponds to finding radially symmetric singular solutions to \eqref{problem-isolated}. This is a non-local ODE for $u=u(r)$. In the  second part of the section we define a suitable Hamiltonian quantity in conformal coordinates in the spirit a classical second order ODE.

\subsection{The kernel}\label{subsection:kernel}

We consider first the projection $m=0$.
 Following the argument in \cite{DelaTorre-delPino-Gonzalez-Wei}, one can use polar coordinates to rewrite $\tilde P_\gamma^{0}$ as an integro-differential operator with a new convolution kernel. Indeed, polar coordinates $x=(r,\theta)$ and $\bar x=(\bar r,\bar\theta)$ in the definition of the fractional Laplacian \eqref{Laplacian-introduction} give
\begin{equation*}
(-\Delta_{\mathbb R^N})^\gamma u(x)=k_{N,\gamma}P.V. \int_0^\infty\int_{\mathbb S^{N-1}}\frac{r^{-\frac{2\gamma}{p-1}}v(r)-\bar r^{-\frac{2\gamma}{p-1}}v(\bar r)}{|r^2+\bar r^2+2r\bar r\langle \theta, \bar\theta\rangle|^{\frac{N+2\gamma}{2}}}\,\bar r^{N-1}\,d\bar r\,d\bar\theta.
\end{equation*}
After the substitutions $\bar r=r{s}$ and $v(r)=(1-{s}^{-\frac{2\gamma}{p-1}})v(r)+{s}^{-\frac{2\gamma}{p-1}}v(r)$, and recalling the definition for $\tilde P_\gamma^0$ from \eqref{tilde-P} we have
\begin{eqnarray*}
\tilde P_\gamma^0(v)=k_{N,\gamma}P.V.\int_0^\infty\int_{\mathbb S^{N-1}}\frac{{s}^{-\frac{2\gamma}{p-1}+N-1}(v(r)-v(r{s}))}{|1+{s}^2-2{s}\langle \theta, \bar\theta\rangle|^{\frac{N+2\gamma}{2}}}\,d{s}\,d\bar\theta+Cv(r),
\end{eqnarray*}
where \begin{equation*}
C=k_{N,\gamma}P.V.\int_0^\infty\int_{\mathbb S^{N-1}}\frac{(1-{s}^{-\frac{2\gamma}{p-1}}){s}^{N-1}}
{|1+{s}^2-2{s}\langle \theta, \bar\theta\rangle|^{\frac{N+2\gamma}{2}}}\,d{s}\,d\bar\theta.
\end{equation*}
Using the fact that $v\equiv 1$ is a solution, one gets that $C=A_{N,p,\gamma}$. Finally, the change of variables
$r=e^{-t}$, $\bar r=e^{-t'}$ yields
\begin{equation}\label{convolution-K}
\tilde P^{0}_\gamma(v)(t)=P.V. \int_{\mathbb R}\tilde{\mathcal K}_0(t-t')[v(t)-v(t')]\,dt'+A_{N,p,\gamma}v(t)
\end{equation}
for the convolution kernel
\begin{equation}\label{FH}
\tilde{\mathcal K}_0(t)=\int_{\mathbb S^{N-1}}\frac{k_{N,\gamma}e^{-(\frac{2\gamma}{p-1}-N)t}}
{|1+e^{2t}-2e^{t}\langle\theta, \bar\theta\rangle|^{\frac{N+2\gamma}{2}}}\,d\bar\theta
=c\,e^{-(\frac{2\gamma}{p-1}-\frac{N-2\gamma}{2})t}\int_0^\pi \frac{(\sin \phi_1)^{N-2}}{(\cosh t-\cos \phi_1)^{\frac{N+2\gamma}{2}}}\,d\phi_1,
\end{equation}
where $\phi_1$ is the angle between $\theta$ and $\bar\theta$ in spherical coordinates,  and $c$ is a positive constant that only depends on $N$ and $\gamma$. From here we have the explicit expression
\begin{equation}\label{kernel-K}
\tilde{\mathcal K}_0(t)=
c\,e^{-(\frac{2\gamma}{p-1}-\frac{N-2\gamma}{2})t}e^{-\frac{N+2\gamma}{2}|t|}\Hyperg\big( \tfrac{N+2\gamma}{2},1+\gamma;\tfrac{N}{2};e^{-2|t|}\big),
\end{equation}
for a different constant $c$.
As in \cite{DelaTorre-delPino-Gonzalez-Wei}, one can calculate its asymptotic behavior, and we refer to this paper for details:
\begin{lemma} The kernel $\tilde{\mathcal K}_0(t)$ is decaying as $t\to\pm\infty$. More precisely,
\begin{equation*}
\tilde{\mathcal K}_0(t)\sim \begin{cases}
 |t|^{-1-2\gamma} &\mbox{ as }|t|\to 0,\\
e^{-(N-\frac{2\gamma}{p-1})|t|} &\mbox{ as }t\to -\infty,\\
e^{-\frac{2p\gamma}{p-1}|t|} &\mbox{ as }t \to +\infty.
\end{cases}\end{equation*}
\end{lemma}

\begin{remark}
Note that \eqref{FH} is a special case of the Funk-Hecke formula for the $m=0$ spherical  harmonic (see \cite{Muller}, pages 29-30).
\end{remark}

We would like to obtain analogous results for any projection $m=1,2,\ldots$. The main result in this section is that one also has \eqref{convolution-K} for any projection $\tilde P_\gamma^m$. However, we have not been able to use the previous argument (in the spirit of the Funk-Hecke formula) and,  instead, we develop a new approach using conformal geometry and the special defining function $\rho^*$ from Proposition \ref{prop:new-defining-function}.\\

Set $Q_0=\frac{2\gamma}{p-1}-\frac{N-2\gamma}{2}$. In the notation of Proposition \ref{thm:symbolV} we have:

\begin{proposition}\label{prop:all-kernels} For the $m$-th projection of the operator $\tilde P_\gamma^{g_0}$ we have the expression
\begin{equation*}
\tilde P_\gamma^m(v_m)(t)=\int_{\mathbb R} \tilde{\mathcal K}_m(t-t')[v_m(t)-v_m(t')]\,dt'+A_{N,p,\gamma}v_m(t),
\end{equation*}
for a convolution kernel $\tilde{\mathcal K}_m$ on $\mathbb R$
with the asymptotic behavior
\begin{equation*}
\tilde{\mathcal K}_m(t)\sim
\begin{cases}
 |t|^{-1-2\gamma} &\mbox{ as }|t|\to 0,\\
e^{-\big(1+\gamma+\sqrt{(\frac{N-2}{2})^2+\mu_m}+Q_0\big)t} &\mbox{ as }t \to +\infty,\\
e^{\big(1+\gamma+\sqrt{(\frac{N-2}{2})^2+\mu_m}-Q_0\big)t} &\mbox{ as }t \to -\infty.
\end{cases}
\end{equation*}
\end{proposition}

\begin{proof} We first consider the case that $p=\frac{N+2\gamma}{N-2\gamma}$ so that $Q_0=0$, and look at the operator $P_\gamma^{g_0}(w)$ from Proposition \ref{prop:symbol}. Let $\rho^*$ be the new defining function from Proposition \ref{prop:new-defining-function} and write a new extension problem for $w$ in the corresponding metric  $\bar g^*$. In this particular case, we can use \eqref{rho*} to write
\begin{equation*}
\rho^*(\rho)=\left[\alpha^{-1}\big(\tfrac{4\rho}{4+\rho^2}\big)^{\frac{N-2\gamma}{2}}
\Hyperg\Big(\tfrac{N-2\gamma}{4},\tfrac{N-2\gamma}{4}, \tfrac{N}{2},\big(\tfrac{4-\rho^2}{4+\rho^2}\big)^2\Big)\right]^{\frac{2}{N-2\gamma}}, \quad
\alpha=\frac{\Gamma(\frac{N}{2})\Gamma(\gamma)}{\Gamma(\frac{N}{4}+\frac{\gamma}{2})^2}.
\end{equation*}
 The extension problem for $\bar g^*$ is
\begin{equation*}
\left\{\begin{array}{@{}r@{}l@{}l}
 -\divergence_{\bar{g}^*}((\rho^*)^{1-2\gamma}\nabla_{\bar{g}^*}W^*)
    &\,=0\quad
    &\text{in } (X,\bar g^*), \medskip\\
W^*|_{\rho=0}
    &\,=w\quad
    &\text{on }(M,g_0);
\end{array}\right.
\end{equation*}
notice that it does not have a zero-th order term. Moreover, for the Neumann data,
\begin{equation*}
P_\gamma^{g_0}(w)=-\tilde d_\gamma \lim_{\rho^* \to 0}(\rho^*)^{1-2\gamma} \partial_{\rho^*} (W^*)+\Lambda_{N,\gamma}w.
\end{equation*}
From the proof of Proposition \ref{prop:new-defining-function} we know that $W^*=(\rho^*)^{-(N/2-\gamma)}\mathcal W$, where $\mathcal W$ is the solution to \eqref{equation30}. Taking the projection over spherical harmonics, and arguing as in the proof of Proposition \ref{prop:symbol}, we have  that
$\widehat{\mathcal W_m}(\tau,\xi)=\widehat{w_m}(\xi)\,\varphi(\tau)$,
and $\varphi=\varphi_\xi^{(m)}$ is given in \eqref{varphi0}. Let us undo all the changes of variable, but let us keep the notation $\varphi(\rho^*)=\varphi_\xi^{(m)}(\tau)$.

Taking the inverse Fourier transform, we obtain a Poisson  formula
\begin{equation*}
W^*_m(\rho^*,t)=\int_{\mathbb R} \mathcal P_{m}(\rho^*,t-t')w_m(t')\,dt',
\end{equation*}
where
\begin{equation*}
\mathcal P_m(\rho^*,t)=\frac{1}{\sqrt{2\pi}}\int_{\mathbb R} (\rho^*)^{-(N/2-\gamma)}\varphi(\rho^*)e^{i\xi t}\,d\xi.
\end{equation*}
Note that, by construction, $\int_{\mathbb R} \mathcal P_m(\rho^*,t)\,dt=1$ for all $\rho^*$.
Now we calculate
\begin{equation*}\begin{split}
\lim_{\rho^* \to 0}(\rho^*)^{1-2\gamma} \partial_{\rho^*} (W_m^*)&=
\lim_{\rho^* \to 0}(\rho^*)^{1-2\gamma} \frac{W_m^*(\rho^*,t)-W_m^*(0,t)}{\rho^*}\\
&= \lim_{\rho^*\to 0}(\rho^*)^{1-2\gamma}\int_{\mathbb R} \frac{\mathcal P_m(\rho^*,t-t')}{\rho^*}[w_m(t')-w_m(t)]\,dt'.
\end{split}\end{equation*}
This implies that
\begin{equation}\label{equation60.1}
P_\gamma^m(w_m)(t)=\int_{\mathbb R} \mathcal K_m(t-t')[w_m(t)-w_m(t')]\,dt'+\Lambda_{N,\gamma}w_m(t),
\end{equation}
where the convolution kernel is defined as
\begin{equation*}
\mathcal K_m(t)=\tilde d_\gamma\lim_{\rho^*\to 0}(\rho^*)^{1-2\gamma}\frac{\mathcal P_m(\rho^*,t)}{\rho^*}.
\end{equation*}
If we calculate this limit, the precise expression for $\varphi$ from \eqref{varphi0} yields that
\begin{equation*}\label{kernel-m}
\mathcal K_m(t)=\frac{1}{\sqrt{2\pi}}\int_{\mathbb R}(\Theta_\gamma^m(\xi)-\Lambda_{N,\gamma})e^{i\xi t}\,d\xi,\quad \mathcal K_m(-t)=\mathcal K_m(t).
\end{equation*}
which, of course, agrees with Proposition \ref{prop:symbol}.

The asymptotic behavior for the kernel follows from the arguments in Section \ref{section:Hardy}, for instance. In particular, the limit as $t\to 0$ is an easy calculation since Stirling's formula implies that $\Theta_\gamma^m(\xi)\sim |\xi|^{2\gamma}$ as $\xi\to\infty$. For the limit as $|t|\to \infty$ we use that the first pole of the symbol happens at $\pm i(1+\gamma+\sqrt{(\frac{N-2}{2})^2+\mu_m})$ so it extends analytically to a strip that contains the real axis. We have:
\begin{equation*}
\mathcal K_m(t)\sim \begin{cases}
 |t|^{-1-2\gamma} &\mbox{ as }|t|\to 0,\\
e^{-\big(1+\gamma+\sqrt{(\frac{N-2}{2})^2+\mu_m}\big)|t|} &\mbox{ as }t \to \pm\infty.
\end{cases}
\end{equation*}

Now we move on to $\tilde P_\gamma^{g_0}(v)$, whose symbol is calculated in Proposition \ref{thm:symbolV}. Recall that under the change
$w(t)=r^{-Q_0}v(t)$, we have
\begin{equation*}
\tilde P_\gamma^{g_0}(v)=e^{-Q_0t}P^{g_0}_\gamma(e^{Q_0t}v).
\end{equation*}
From \eqref{equation60.1}, if we split
 $e^{Q_0 t}v_m(t)=(e^{Q_0t}-e^{Q_0 t'})v_m(t)+e^{Q_0 t'}v_m(t)$, then
 \begin{equation*}
\tilde P_\gamma^{g_0}(v)(t)=Cv(t)+\int_{\mathbb R}\tilde{\mathcal K}_m(t-t') (v_m(t)-v_m(t'))\,dt'
\end{equation*}
for the kernel
\begin{equation*}
\tilde{\mathcal K}_m(t)=\mathcal K_m(t) e^{-Q_0t}=\frac{1}{\sqrt{2\pi}}e^{-Q_0 t}\int_{\mathbb R}(\Theta_\gamma^m(\xi)-\Lambda_{N,\gamma})e^{i\xi t}\,d\xi,
\end{equation*}
and the constant
\begin{equation*}
C=\Lambda_{N,\gamma}+\int_{\mathbb R} \mathcal K_m(t-t')(1-e^{Q_0 (t'-t)})\,dt'.
\end{equation*}
We have not attempted a direct calculation for the constant $C$. Instead, by noting that $v\equiv 1$ is an exact solution to the equation $\tilde P_\gamma^{g_0}(v)=A_{N,p,\gamma}v^{p}$, we have that $C=A_{N,p,\gamma}$, and this completes the proof of the proposition.
\end{proof}

\subsection{The Hamiltonian along trajectories}\label{subsection:Hamiltonian}

Now we concentrate on positive radial solutions to  \eqref{problem-vv}. These satisfy
\begin{equation}\label{problem-v}
\tilde P_\gamma^{0}(v)=A_{N,p,\gamma}v^p,\quad v=v(t).
\end{equation}
We prove the existence of a Hamiltonian type quantity for  \eqref{problem-v}, decreasing along trajectories when $p$ is in the subcritical range, while this Hamiltonian remains constant in $t$ for critical $p$. Monotonicity formulas for non-local equations in the form of a Hamiltonian have been known for some time (\cite{CabreSola-Morales,Cabre-Sire1,Frank-Lenzmann-Silvestre}). Our main innovation is that our formula \eqref{Hamiltonian} gives a precise analogue of the ODE local case (see Proposition 1 in \cite{Mazzeo-Pacard}, and the notes \cite{Schoen:notas}),
 and hints what the phase portrait for $v$  should be in the non-local setting. We hope to return to this problem elsewhere.

\begin{teo}\label{thctthamiltonian}
Fix $\gamma\in(0,1)$ and $p\in(\frac{N}{N-2\gamma},\frac{N+2\gamma}{N-2\gamma})$. Let $v=v(t)$ be a solution to \eqref{problem-v} and set $V^*$ its extension from Proposition \ref{prop:divV*}.
Then, the Hamiltonian quantity
\begin{equation}\label{Hamiltonian}
\begin{split}
H^*_\gamma(t)=&\frac{A_{N,p,\gamma}}{\tilde{d}_{\gamma}}
\left(-\frac{1}{2}v^2+\frac{1}{p+1}v^{p+1}\right)
+\frac{1}{2}\int_{0}^{\rho^*_0}  {(\rho^*)}^{1-2\gamma}\left\{-e^*_1(\partial_{\rho^*}V^*)^2
+e^*_2(\partial_t V^*)^2\right\}\,d\rho^*\\
=&:H_1(t)+H_2(t)
\end{split}
\end{equation}
is decreasing with respect to $t$. In addition, if $p=\frac{N+2\gamma}{N-2\gamma}$, then $H^*_\gamma(t)$ is constant along trajectories.

Here we write, using Lemma \ref{cg-modified},  $\rho$ as a function of $\rho^*$, and
\begin{equation}\label{e}
\begin{split}
&e^*=\left( \tfrac{\rho^*}{\rho}\right)^2\left( 1+\tfrac{\rho^2}{4}\right)\left(1-\tfrac{\rho^2}{4}\right)^{N-1},\\
&e^*_1=\left( \tfrac{\rho^*}{\rho}\right)^{-2}e^*,\\
&e^*_2=\left( \tfrac{\rho^*}{\rho}\right)^{-2}\left( 1+\tfrac{\rho^2}{4}\right)^{-2}e^*.
\end{split}\end{equation}
The constants $A_{N,p,\gamma}$ and $\tilde{d}_{\gamma}$ are given in \eqref{Apn} and \eqref{tilde-d}, respectively.
\end{teo}

\begin{proof}
In the notation of Proposition \ref{prop:divV*}, let $v$ be a function on $M=\mathbb R\times\mathbb S^{N-1}$ only depending on the variable $t\in\mathbb R$, and let $V^*$ be the corresponding solution to the extension problem \eqref{divV*}. Then $V^*=V^*(\rho,t)$. Use that \begin{equation*}\label{equation40}
\begin{split}
\divergence_{\bar{g}^*}((\rho^*)^{1-2\gamma}\nabla_{\bar{g}^*}V^*)
=\frac{1}{e^*}\partial_{\rho^*}\left(e^*(\rho^*)^{-(1+2\gamma)}\rho^2\partial_{\rho^*}V^*\right)
+(\rho^*)^{1-2\gamma}\left(\tfrac{\rho^*}{\rho}\right)^{-2}\left(1+\tfrac{\rho^2}{4}\right)^{-2}\partial_{tt}V^*,
\end{split}
\end{equation*}
where
$e^*=|\sqrt{\bar{g}^*}|$ is given in \eqref{e}, so
equation \eqref{divV*} reads
\begin{equation*}\label{eqdiv*final}
\begin{split}
-&\partial_{\rho^*}\left(e^*\rho^2(\rho^*)^{-(1+2\gamma)}\partial_{\rho^*}V^*\right)
-(\rho^*)^{1-2\gamma}e^*\left(\tfrac{\rho^*}{\rho}\right)^{-2}
\left(1+\tfrac{\rho^2}{4}\right)^{-2}\partial_{tt}V^* \\
&\qquad-(\rho^*)^{-(1+2\gamma)}e^*\left(\tfrac{4+\rho^{2}}{4\rho}\right)^{-2}
2\left(-\tfrac{N-2\gamma}{2}+\tfrac{2\gamma}{p-1}\right)\partial_t V^*=0.
\end{split}
\end{equation*}
We follow the same steps as in \cite{DelaTorre-Gonzalez}:  multiply this equation by $\partial_t{V^*}$ and integrate with respect to $\rho^*\in(0,\rho^*_0)$, where $\rho^*_0$ is given in \eqref{rho*0}. Using integration by parts in the first term, the regularity of the function $V^*$ at $\rho^*_0$, and the fact that $\tfrac{1}{2}\partial_t\left[(\partial_t V^*)^2\right]=\partial_{tt}V^*\partial_t V^*$ and $\tfrac{1}{2}\partial_t\left[(\partial_{\rho^*} V^*)^2\right]=\partial_{t\rho^*}(V^*)\partial_{\rho^*}V^*$, it holds
\begin{equation*}\label{ham_inter}\begin{split}
&\lim_{\rho^*\rightarrow 0}\left(\partial_t(V^*)e^*(\rho^*)^{-(1+2\gamma)}\rho^2\partial_{\rho^*}V^*\right)\\
&+\int_{0}^{\rho^*_0}\left[\tfrac{1}{2}e^*(\rho^*)^{-(1+2\gamma)}\rho^2\partial_t\left[(\partial_{\rho^*} V^*)^2\right]\right]\,d\rho^*
-\int_{0}^{\rho^*_0}\left[\tfrac{1}{2}(\rho^*)^{1-2\gamma}e^*
\left(\tfrac{\rho}{\rho^*}\right)^{2}\left(1+\tfrac{\rho^2}{4}\right)^{-2}\partial_{t}\left[(\partial_t V^*)^2\right]\right]\,d\rho^* \\
&-\int_{0}^{\rho^*_0}\left[(\rho^*)^{-(1+2\gamma)}e^*
\left(\tfrac{4\rho}{4+\rho^{2}}\right)^{2}2\left(-\tfrac{N-2\gamma}{2}+\tfrac{2\gamma}{p-1}\right)
\left[\partial_t V^*\right]^2\right]\,d\rho^*\\
&=0.
\end{split}
\end{equation*}
But, for the limit as $\rho^*\to 0$, we use \eqref{Neumann-V*} and  equation \eqref{problem-v},
\begin{equation*}\label{limit}\begin{split}
\tilde d_\gamma\lim_{\rho^*\rightarrow 0}\left((\rho^*)^{-(1+2\gamma)}\rho^2e^*\partial_t V^*\partial_{\rho^*}V^*\right)
&=\left[-\tilde P_\gamma^{g_0}v+A_{N,p,\gamma}v\right]\partial_t v
=A_{N,p,\gamma}(v-v^p)\partial_t v\\
&=A_{N,p,\gamma}\partial_t \left(\tfrac{1}{2}v^2-\tfrac{1}{p+1}v^{p+1}\right).
\end{split}\end{equation*}
Then, for $H(t)$ defined as in \eqref{Hamiltonian}, we have
\begin{equation*}\label{ham_der}
\partial_t\left[ H(t)\right]=-2\int_{0}^{\rho^*_0}\left[(\rho^*)^{1-2\gamma}e^*
\left(\tfrac{\rho}{\rho^*}\right)^{-2}\left(1+\tfrac{\rho^2}{4}\right)^{-2}
\left(-\tfrac{N-2\gamma}{2}+\tfrac{2\gamma}{p-1}\right)\left[\partial_t V^*\right]^2\right]\,d\rho^*\leq 0,
\end{equation*}
which proves the result.
\end{proof}

\section{The approximate solution}

\subsection{Function spaces}\label{section:function-spaces}

In this section we define the weighted H\"{o}lder space $\mathcal C_{\mu,\nu}^{2,\alpha}(\R^n \setminus \Sigma)$ tailored for this problem, following the notations and definitions in Section 3 of \cite{mp}. Intuitively, these spaces consist of functions which are products of powers of the distance to $\Sigma$ with functions whose H\"{o}lder norms are invariant under homothetic transformations centered at an arbitrary point on $\Sigma$.

Despite the non-local setting, the local Fermi coordinates are still in use around each component $\Sigma_i$ of $\Sigma$.
When $\Sigma_i$ is a point, these are simply polar coordinates around it. In case $\Sigma_i$ is a higher dimensional sub-manifold, let $\mathcal{T}_\sigma^i$ be the tubular neighbourhood of radius $\sigma$ around $\Sigma_i$. It is well known that $\mathcal{T}_\sigma^i$ is a disk bundle over $\Sigma_i$; more precisely, it is diffeomorphic to the bundle of radius $\sigma$ in the normal bundle $\mathcal N\Sigma_i$.  The Fermi coordinates will be constructed as coordinates in the normal bundle transferred to $\mathcal{T}_\sigma^i$ via such diffeomorphism. Let $r$ be the distance to $\Sigma_i$, which is well defined and smooth away from $\Sigma_i$ for small $\sigma$. Let also $y$ be a local coordinate system on $\Sigma_i$ and $\theta$ the angular variable on the sphere in each normal space $\mathcal N_y\Sigma_i$. We denote by $B^{\mathcal N}_{\sigma}$ the ball of radius $\sigma$ in $\mathcal N_y\Sigma_i$. Finally we let $x$ denote the rectangular coordinate in these normal spaces, so that $r=|x|$, $\theta=\frac{x}{|x|}$.

Let $u$ be a function in this tubular neighbourhood and define
\begin{equation*}
\|u\|_{0,\alpha,0}^{\mathcal{T}_\sigma^i}=\sup_{z\in \mathcal{T}_\sigma^i}|u|+\sup_{z,\tilde{z}\in \mathcal{T}_\sigma^i}\frac{(r+\tilde{r})^\alpha|u(z)-u(\tilde{z})|}
{|r-\tilde{r}|^\alpha+|y-\tilde{y}|^\alpha+(r+\tilde{r})^\alpha|\theta-\tilde{\theta}|^\alpha},
\end{equation*}
where $z,\tilde{z}$ are two points in $\mathcal{T}_\sigma^i$ and $(r,\theta,y), (\tilde{r},\tilde{\theta}, \tilde{y})$ are their Fermi coordinates.

We fix a $R>0$ be large enough such that $\Sigma \subset B_{\frac{R}{2}}(0)$ in $\R^n$. Hereafter the letter $z$ is reserved to denote a point in $\R^n \setminus \Sigma$.  For notational convenience let us also fix a positive function $\varrho\in\mathcal C^\infty(\R^n \setminus \Sigma)$ that is equal to the polar distance $r$ in each $\mathcal{T}_\sigma^i$, and to $|z|$ in $\R^n \setminus B_R(0)$.

\begin{definition}
The space $\mathcal C_0^{l,\alpha}(\R^n \setminus \Sigma)$ is defined to be the set of all $u\in \mathcal C^{l,\alpha}(\R^n \setminus \Sigma)$ for which the norm
\begin{equation*}
\|u\|_{l,\alpha,0}=\|u\|_{\mathcal C^{l,\alpha}(\Sigma_{\sigma/2}^c)}+\sum_{i=1}^K\sum_{j=0}^l\|\nabla^j u\|_{\mathcal C^{0,\alpha}(\mathcal{T}_\sigma^i)}
\end{equation*}
is finite. Here $\Sigma_{\sigma/2}^c=\R^n \setminus \bigcup_{i=1}^K \mathcal{T}_{\sigma/2}^i$.
\end{definition}

Let us define a weighted H\"older space for functions having different behaviors near $\Sigma$ and at $\infty$. With $R>0$ fixed, for any $\mu,\nu\in \R$ we set
\begin{equation*}\begin{split}
\mathcal C_\mu^{l,\alpha}(B_R \setminus \Sigma)&=\{u=\varrho^\mu \bar{u}: \bar{u}\in \mathcal C_0^{l,\alpha}(B_R \setminus \Sigma)\},\\
\mathcal C_\nu^{l,\alpha}(\R^n \setminus B_R)&=\{u=\varrho^\nu \bar{u}: \bar{u}\in \mathcal C_0^{l,\alpha}(\R^n \setminus B_R )\},
\end{split}
\end{equation*}
and thus we can define:

\begin{definition}
The space $\mathcal C^{l,\alpha}_{\mu,\nu}(\R^n \setminus \Sigma)$ consists of all functions $u$ for which the norm
\begin{equation*}
\|u\|_{\mathcal C^{l,\alpha}_{\mu,\nu}}=\sup_{B_R\setminus \Sigma}\|\varrho^{-\mu}u\|_{l,\alpha,0}+\sup_{\R^n\setminus B_R}\|\varrho^{-\nu}u\|_{l,\alpha,0}
\end{equation*}
is finite. The spaces $\mathcal C^{l,\alpha}_{\mu,\nu}(\R^N \setminus\{0\})$ and $\mathcal C^{l,\alpha}_{\mu,\nu}(\R^n \setminus \R^k)$ are defined similarly, in terms of the (global) Fermi coordinates $(r,\theta)$ or $(r,\theta,y)$ and the weights $r^\mu$, $r^\nu$.
\end{definition}

\begin{remark}\label{remark:edge}
From the definition of $\mathcal C^{l,\alpha}_{\mu,\nu}$, functions in this space are allowed to blow up like $\varrho^\mu$ near $\Sigma_i$ and decay like $\varrho^\nu$ at $\infty$.  Moreover, near $\Sigma_i$, their derivatives with respect to up to $l$-fold products of the vector fields $r\partial_r, r\partial_y, \partial_\theta$ blow up no faster than $\varrho^\mu$ while at $\infty$, their derivatives with respect to up to $l$-fold products of the vector fields $|z|\partial_i$ decay at least like $\varrho^{\nu}$.
\end{remark}

\begin{remark}
As it is customary in the analysis of fractional order operators, we write many times, with some abuse of notation, $\mathcal C^{2\gamma+\alpha}_{\mu,\nu}$.
\end{remark}

\subsection{Approximate solution with isolated singularities}

Let $\Sigma =\{q_1, \cdots, q_K\}$ be a prescribed set of singular points. In the next paragraphs we  construct an approximate solution to
\begin{equation*}
(-\Delta_{\mathbb R^N})^\gamma u=A_{N,p,\gamma}u^p \mbox{ in }\R^N\setminus\Sigma,
\end{equation*}
and check that it is indeed a good approximation in certain weighted spaces.

Let $u_1$ be the fast decaying solution to \eqref{problem-isolated} that we constructed in Proposition \ref{existence}.
Now consider the following rescaling
\begin{equation}\label{rescaling}
u_\ve(x)=\ve^{-\frac{2\gamma}{p-1}}u_1\left(\frac{x}{\ve}\right)\quad\text{in}\quad\mathbb R^N\setminus\{0\}.
\end{equation}

Choose  $\chi_d$ to be a smooth cut-off function such that $\chi_d=1$ if $|x|\leq d$ and $\chi_d(x)=0$ for $|x|\geq 2d$, where $d>0$ is a positive constant such that $d<d_0=\inf_{i\neq j}\{\dist(q_i,q_j)/2\}$. Let ${\bar\ve}=\{\ve_1,\cdots,\ve_K\}$ be a $K$-tuple of dilation parameters satisfying $c\ve\leq\ve_i\leq\ve<1$ for $i=1,\ldots,K$. Now define our approximate solution by
\begin{equation*}
\bar{u}_\ve(x)=\sum_{i=1}^K\chi_d(x-q_i)u_{\ve_i}(x-q_i).
\end{equation*}
Set also
\begin{equation}\label{f-epsilon}
f_\ve:=(-\Delta_x)^\gamma \bar{u}_\ve-A_{N,p,\gamma}\bar{u}_\ve^p.
\end{equation}

For the rest of the section, we consider the spaces $\mathcal C^{0,\alpha}_{\tilde\mu,\tilde\nu}$, where
\begin{equation}\label{choice-weights-nonlinear}
-\frac{2\gamma}{p-1}<\tilde \mu<2\gamma\quad \text{and}\quad -(n-2\gamma)<\tilde\nu.
\end{equation}

\begin{lemma}
There exists a constant $C$, depending on $d, \tilde\mu,\tilde\nu$ only, such that
\begin{equation}\label{error1}
\|f_\ve\|_{\mathcal C^{0,\alpha}_{\tilde\mu-2\gamma,\tilde\nu-2\gamma}}\leq C\ve^{N-\frac{2p\gamma}{p-1}}.
\end{equation}

\end{lemma}
\begin{proof}
Using the definition of $(-\Delta)^\gamma$ in $\mathbb R^N$, one has
\begin{equation*}\begin{split}
(-\Delta_x)^\gamma (\chi_iu_{\ve_i})(x-q_i)&=k_{N,\gamma}P.V. \int_{\R^N}\frac{\chi_i(x-q_i)u_{\ve_i}(x-q_i)-\chi_i(\tilde{x}-q_i)u_{\ve_i}(\tilde{x}-q_i)}
{|x-\tilde{x}|^{N+2\gamma}}\,d\tilde{x}\\
&=\chi_i(x-q_i)(-\Delta_x)^\gamma u_{\ve_i}(x-q_i)\\
&\quad+k_{N,\gamma}
P.V.\int_{\R^N}\frac{(\chi_i(x-q_i)-\chi_i(\tilde{x}-q_i))u_{\ve_i}(\tilde{x}-q_i)}
{|x-\tilde{x}|^{N+2\gamma}}\,d\tilde{x}
\end{split}
\end{equation*}
for each $i=1,\ldots,K$. Using the equation \eqref{Lane-Emden} satisfied by $u_{\ve_i}$ we have
\begin{equation*}
\begin{split}
f_\ve(x)&=
A_{N,p,\gamma}\sum_{i=1}^K(\chi_i-\chi_i^p)u_{\ve_i}^p(x-q_i)+k_{N,\gamma}\sum_{i=1}^K
P.V.\int_{\R^N}\frac{(\chi_i(x-q_i)-\chi_i(\tilde{x}-q_i))u_{\ve_i}(\tilde{x}-q_i)}
{|x-\tilde{x}|^{N+2\gamma}}\,d\tilde{x}\\
&=:I_1+k_{N,\gamma}I_2.
\end{split}
\end{equation*}
Let us look first at the term $I_1$. It vanishes unless  $|x-q_{i}|\in [d, 2d]$ for some $i=1,\ldots,K$. But then, one knows from the asymptotic behaviour of $u_{\ve_i}$ that
$$u_{\ve_i}(x)=O\Big(\ve_i^{-\frac{2\gamma}{p-1}}\Big|\frac{x-q_i}{\ve_i}\Big|^{-(N-2\gamma)}\Big)
=O(\ve^{N-2\gamma-\frac{2\gamma}{p-1}})|x-q_i|^{-(N-2\gamma)},$$
so one has
\begin{equation*}
I_1(x)\leq C\ve^{N-2\gamma-\frac{2\gamma}{p-1}} \quad\mbox{if }|x-q_{i}|\in [d, 2d].
\end{equation*}

For the second term $I_2=I_2(x)$, we fix $i=1,\ldots,K$, and divide it into three cases: $x\in B_{d/2}(q_{i})$, $x\in B_{2d}(q_{i})\setminus B_{d/2}(q_{i})$ and $x\in \R^N \setminus B_{2d}(q_{i})$. In the first case, $x\in B_{d/2}(q_{i})$, without loss of generality, assume that $q_{i}=0$, so
\begin{equation*}
\begin{split}
I_2(x)&=P.V.\int_{\R^N}\frac{(\chi_i(x)-\chi_i(\tilde{x}))u_{\ve_i}(\tilde{x})}{|x-\tilde{x}|^{N+2\gamma}}\,d\tilde{x}\\
&=P.V\Big[\int_{B_d(0)}\cdots +\int_{B_{2d}\setminus B_d(0)}\cdots +\int_{\R^N \setminus B_{2d}(0)}\cdots\Big]\\
&\lesssim \int_{\{d<|\tilde{x}|<2d\}}\frac{u_{\ve_i}(\tilde x)}{|x-\tilde{x}|^{N+2\gamma-2}}\,d\tilde x
+\int_{\{|\tilde{x}|>2d\}}\frac{u_{\ve_i}(\tilde{x})}{|x-\tilde{x}|^{N+2\gamma}}\,d\tilde x.
\end{split}
\end{equation*}
Hereafter ``$\cdots$'' carries its obvious meaning, replacing the previously written integrand. Using that $|x-\tilde{x}|\geq \frac{1}{2}|\tilde{x}|$ for $|\tilde{x}|>2d$ when $|x|<\frac{d}{2}$ we easily estimate
\begin{equation*}
I_2(x)\leq O(\ve^{N-\frac{2p\gamma}{p-1}})+\int_{2d}^\infty \frac{u_{\ve_i}(r)}{r^{1+2\gamma}}\,dr=O(\ve^{N-\frac{2p\gamma}{p-1}}).
\end{equation*}

Next, if $x\in B_{2d}(q_i)\setminus B_{d/2}(q_i)$,
\begin{equation*}
\begin{split}
I_2(x)&=P.V.\Big[\int_{B_{d/4}(q_i)}\cdots +\int_{B_{2d}(q_i)\setminus B_{d/4}(q_i)}\cdots +\int_{\R^N \setminus B_{2d}(q_i)}\cdots\Big]\\
&=O\Big(\int_0^{\frac{d}{4\ve}}\ve^{N-\frac{2p\gamma}{p-1}}u_1(\tilde{x})\,d\tilde{x}\Big)
+O\Big(\int_{B_{2d}(q_i)\setminus B_{d/4}(q_i)}\frac{\ve^{N-\frac{2p\gamma}{p-1}}}{|x-\tilde{x}|^{N+2\gamma-2}}\,d\tilde{x}\Big)\\
&\quad+O\Big(\int_{\R^N\setminus B_{2d}(q_i)}\frac{u_{\ve_i}(\tilde{x})}{|\tilde{x}|^{N+2\gamma}}\,d\tilde{x}\Big)\\
&=O(\ve^{N-\frac{2p\gamma}{p-1}}).
\end{split}
\end{equation*}

Finally, if $x\in \R^N\setminus B_{2d}(q_i)$,
\begin{equation*}
\begin{split}
I_2(x)&=P.V.\Big[\int_{B_d(q_i)}\cdots +\int_{B_{2d}\setminus B_d(q_i)}\cdots +\int_{\R^N \setminus B_{2d}(q_i)}\cdots\Big]\\
&=O(\ve^{N-\frac{2p\gamma}{p-1}}|x|^{-(N+2\gamma)}).
\end{split}
\end{equation*}
Combining all the estimates above we get a $\mathcal C^{0}_{\tilde\mu-2\gamma, \tilde\nu-2\gamma}$ bound for a pair of weights satisfying \eqref{choice-weights-nonlinear}. But passing to $\mathcal C^{0,\alpha}_{\tilde\mu-2\gamma, \tilde\nu-2\gamma}$ is analogous and thus we obtain \eqref{error1}.
\end{proof}

\subsection{Approximate solution in general case}

First note that our ODE argument for $u_1$ also yields a fast decaying positive solution to the general problem
\begin{equation}\label{equation0}
(-\Delta_{\mathbb R^n})^\gamma u=A_{N,p,\gamma}u^p \mbox{ in }\R^n\setminus \mathbb R^k.
\end{equation}
that is singular along $\mathbb R^k$. Recall that we have set $N=n-k$.

Indeed, define $\tilde u_1(x,y):=u_1(x)$, where $z=(x,y)\in \R^{n-k}\times \R^k$, and use the Lemma below. For this reason, many times we will use indistinctly the notations $u_1(z)$ and $u_1(x)$. Moreover, after a straightforward rescaling, the constant $A_{N,p,\gamma}$ may be taken to be one.

\begin{lemma} \label{lemma5.6}
If $u$ is defined on $\mathbb R^N$ and we set $\tilde u(z):=u(x)$ in $\mathbb R^n$ in the notation above, then
$$(-\Delta_{\mathbb R^n})^\gamma \tilde u=(-\Delta_{\mathbb R^N})^\gamma u.$$
\end{lemma}

\begin{proof}
We compute, first evaluating the $y$-integral,
\begin{equation*}\begin{split}
(-\Delta_{\mathbb R^n})^\gamma \tilde u(z)&=k_{n,\gamma}P.V.\int_{\R^n}\frac{\tilde u(z)-\tilde u(\tilde{z})}{|z-\tilde{z}|^{n+2\gamma}}\,d\tilde{z}=k_{n,\gamma}P.V.\int_{\mathbb R^k}\int_{\mathbb R^N}\frac{u(x)-u(\tilde x)}{[|x-\tilde{x}|^2+|y-\tilde{y}|^2]^{\frac{n+2\gamma}{2}}}\,d\tilde{x}\,d\tilde{y}\\
&=k_{n,\gamma}P.V.\int_{\mathbb R^N} \frac{u(x)-u(\tilde x)}{|x-\tilde{x}|^{N+2\gamma}}\,d\tilde
x\int_{\mathbb R^k} \frac{1}{(1+|\tilde y|^2)^{\frac{n+2\gamma}{2}}}\,d\tilde y=(-\Delta_{\mathbb R^N})^\gamma u(x).
 \end{split}\end{equation*}
Here we have used
\begin{equation}\label{relation}
k_{n,\gamma}\int_{\R^k}\frac{1}{(1+|\tilde y|^2)^{\frac{n+2\gamma}{2}}}\,d\tilde y=k_{N,\gamma}.
\end{equation}
(See Lemma A.1 and Corollary A.1 in \cite{cw}).
\end{proof}

Now we turn to the construction of an approximate solution for \eqref{problem-introduction}. Let $\Sigma$ be a $k$-dimensional compact sub-manifold in $\R^n$. We shall use local Fermi coordinates around $\Sigma$, as defined in Section \ref{section:function-spaces}. Let $\mathcal T_\sigma$ be the tubular neighbourhood of radius $\sigma$ around $\Sigma$. For a point $z\in \mathcal T_\sigma$, denote it by $z=(x,y)\in \mathcal N\Sigma\times \Sigma$ where $\mathcal N\Sigma $ is the normal bundle of $\Sigma$.
Let $B$ a ball in $\mathcal N\Sigma$. We identify $\mathcal T_\sigma$ with $B\times\Sigma$. In these coordinates, the Euclidean metric is written as (see, for instance, \cite{Mazzeo-Smale})
\begin{equation*}
|dz|^2=
\begin{pmatrix}|dx|^2& O(r)\\
O(r)&g_\Sigma+O(r)\end{pmatrix},
\end{equation*}
where $|dx|^2$ is the standard flat metric in $B$ and $g_\Sigma$ the metric in $\Sigma$. The volume form reduces to
\begin{equation*}
dz=dx\sqrt{\det g_\Sigma}+O(r).
\end{equation*}
In the ball $B$ we use standard polar coordinates $r>0$, $\theta\in\mathbb S^{N-1}$. In addition, near each $q\in\Sigma$, we will consider normal coordinates for $g_\Sigma$ centered at $q$. A neighborhood of $\Sigma\ni q$ is then identified with a neighborhood of $\mathbb R^k\ni 0$ with the metric
\begin{equation*}
g_\Sigma=|dy|^2+O(|y|^2),
\end{equation*}
which yields the volume form
\begin{equation}\label{volume-form}
dz=dx\,dy\,(1+O(r)+O(|y|^2)).
\end{equation}
Note that $\Sigma$ is compact, so we can cover it by a finite number of small balls $B$.

As in the isolated singularity case, we define an approximate solution as follows:
\begin{equation*}
\bar{u}_\ve(x,y)=\chi_d(x)u_\ve(x)
\end{equation*}
where $\chi_d$ is a cut-off function such that $\chi_d=1$ if $|x|\leq d$ and $\chi_d(x)=0$ for $|x|\geq 2d$. In the following we always assume $d<\frac{\sigma}{2}$. Let
\begin{equation*}
f_\ve:=(-\Delta_{\mathbb R^n})^\gamma \bar{u}_\ve-\bar{u}_\ve^p.
\end{equation*}

\begin{lemma}\label{lemma:error2}
Assume, in addition to \eqref{choice-weights-nonlinear},  that $-\frac{2\gamma}{p-1}<\tilde\mu<\min\{\gamma-\frac{2\gamma}{p-1}, \frac{1}{2}-\frac{2\gamma}{p-1}\}$. Then there exists a positive constant $C$ depending only on $d,\tilde\mu,\tilde\nu$ but independent of $\ve$ such that for $\ve\ll1$,
\begin{equation}\label{error2}
\|f_\ve\|_{\mathcal C^{0,\alpha}_{\tilde\mu-2\gamma,\tilde\nu-2\gamma}}\leq C\ve^q,
\end{equation}
where $q=\min\{\frac{(p-3)\gamma}{p-1}-\tilde{\mu}, \frac{1}{2}-\gamma+\frac{(p-3)\gamma}{p-1}-\tilde{\mu},N-\frac{2p\gamma}{p-1}\}>0$.
\end{lemma}

\begin{proof}
Let us fix a point $z=(x,y)\in \mathcal  T_\sigma$, i.e. $|x|<\sigma$. By the definition of the fractional Laplacian,
\begin{equation*}\begin{split}
(-\Delta_z)^\gamma \bar{u}_\ve(z)&=k_{n,\gamma}P.V.\int_{\R^n}\frac{\bar{u}_\ve(z)-\bar{u}_\ve(\tilde{z})}{|z-\tilde{z}|^{n+2\gamma}}\,d\tilde{z}\\
&=k_{n,\gamma}\Big[P.V.\int_{\mathcal T_\sigma}\cdots+\int_{\mathcal T_\sigma^c}\cdots \Big]=:I_1+I_2.
\end{split}\end{equation*}
Note that in this neighborhood we can write $\bar u_\ve(z):=\bar u_\ve(x)$.

For $I_2$, since $\bar{u}_\ve(\tilde{x})=0$ when $\tilde{z}=(\tilde{x},\tilde{y})\in \mathcal T_\sigma^c$, one has
\begin{equation*}\begin{split}
I_2&=k_{n,\gamma}\int_{\mathcal T_\sigma^c}\frac{\bar{u}_\ve(x)-\bar{u}_\ve(\tilde{x})}{|z-\tilde{z}|^{n+2\gamma}}\,d\tilde{z}
=\bar{u}_\ve(x)\,k_{n,\gamma}\int_{\mathcal T_\sigma^c}\frac{1}{|z-\tilde{z}|^{n+2\gamma}}\,d\tilde{z}
\leq C\bar{u}_\ve(x),
\end{split}
\end{equation*}
so $I_2=O(1)\bar{u}_\ve(x)$ (the precise constant depends on $\sigma$). Next, for $I_1$, use normal coordinates $\tilde y$ in $\Sigma$ centered at $y$ in a neighborhood $\{|y-\tilde y|<\sigma_1\}$ for some $\sigma_1$ small but fixed. The constants will also depend on this $\sigma_1$. We have
\begin{equation*}
\begin{split}
I_1&=k_{n,\gamma}P.V.\int_{\mathcal T_\sigma}\frac{\bar{u}_\ve(x)-\bar{u}_\ve(\tilde{x})}{|z-\tilde{z}|^{n+2\gamma}}\,d\tilde{z}\\
&=k_{n,\gamma}P.V.\Big[\int_{\{|y-\tilde{y}|\leq |x|^\beta\}\cap \mathcal T_\sigma}\cdots+\int_{\{\sigma_1>|y-\tilde{y}|>|x|^\beta\}\cap \mathcal T_\sigma}\cdots+\int_{\{|y-\tilde{y}|>\sigma_1\}\cap \mathcal T_\sigma}\cdots\Big]\\
&=:k_{n,\gamma}[I_{11}+I_{12}+I_{13}],
\end{split}
\end{equation*}
where $\beta\in(0,1) $ is to be determined later. The main term will be $I_{11}$; let us calculate the other two. First, for $I_{12}$ we recall the expansion of the volume form \eqref{volume-form}, and approximate $|z-\tilde z|^2=|x-\tilde x|^2+|y-\tilde y|^2$ and $d\tilde z=d\tilde x \,d\tilde y$ modulo lower order perturbations. Then
\begin{equation*}
\begin{split}
I_{12}&=\int_{\{\sigma_1>|y-\tilde{y}|>|x|^\beta\}\cap \mathcal T_\sigma}\frac{\bar{u}_\ve(x)-\bar{u}_\ve(\tilde{x})}{|z-\tilde{z}|^{n+2\gamma}}\,d\tilde{z}\\
&=\int_{\{\sigma_1>|y-\tilde{y}|>|x|^\beta\}}\int_{\{|x-\tilde{x}|\leq |x|^\beta\}\cap \mathcal T_\sigma}\cdots +\int_{\{\sigma_1>|y-\tilde{y}|>|x|^\beta\}}\int_{\{|x-\tilde{x}|>|x|^\beta\}\cap \mathcal T_\sigma}\cdots\\
&\lesssim\int_{\{|x-\tilde{x}|\leq |x|^\beta\}}(\bar{u}_\ve(x)-\bar{u}_\ve(\tilde{x}))\Big(\int_{\{\sigma_1>|y-\tilde{y}|>|x|^\beta\}}
\frac{1}{|z-\tilde{z}|^{n+2\gamma}}\,d\tilde{y}\Big)d\tilde x\\
&\quad+\int_{\{|x-\tilde{x}|>|x|^\beta\}}\frac{\bar{u}_\ve(x)-\bar{u}_\ve(\tilde{x})}{|x-\tilde{x}|^{N+2\gamma}}
\Big(\int_{\big\{|\hat{y}|>\frac{|x|^\beta}{|x-\tilde{x}|}\big\}}\frac{1}{(1+|\hat{y}|^2)^\frac{n+2\gamma}{2}}
\,d\hat{y}\Big)d\tilde{x}.
\end{split}
\end{equation*}
We estimate the above integrals in $d\tilde y$. For instance, for the first term, we have used that
\begin{equation*}
\begin{split}
\int_{\{\sigma_1>|y-\tilde{y}|>|x|^\beta\}}
\frac{1}{|z-\tilde{z}|^{n+2\gamma}}\,d\tilde{y}
&\leq \int_{\{\sigma_1>|y-\tilde{y}|>|x|^\beta\}}\frac{1}{|y-\tilde{y}|^{n+2\gamma}}\,d\tilde{y}
\lesssim \int_{\{|y|\geq |x|^\beta\}}\frac{1}{|y|^{n+2\gamma}}\,dy\\
&\lesssim \int_{|x|^\beta}^{\infty}\frac{r^{k-1}}{r^{n+2\gamma}}\,dr
\lesssim |x|^{-\beta(N+2\gamma)},
\end{split}
\end{equation*}
which yields,
\begin{equation*}
\begin{split}
I_{12}&\lesssim\int_{\{|x-\tilde{x}|\leq |x|^\beta\}}(\bar{u}_\ve(x)-\bar{u}_\ve(\tilde{x}))|x|^{-\beta(N+2\gamma)}\,d\tilde{x}
+\int_{\{|x-\tilde{x}|>|x|^\beta\}}\frac{\bar{u}_\ve(x)-\bar{u}_\ve(\tilde{x})}{|x-\tilde{x}|^{N+2\gamma}}\,d\tilde{x}.\\
\end{split}
\end{equation*}
Now, since  $|x-\tilde{x}|>|x|^\beta$ implies that $|\tilde{x}|>c_0|x|^\beta$ and  $|x-\tilde{x}|\sim |\tilde{x}|$ for some $c_0>0$ independent of $|x|$ small,
\begin{equation*}
\begin{split}
I_{12}&\lesssim |x|^{-2\beta\gamma}\bar{u}_\ve(x)+|x|^{-\beta(N+2\gamma)}\int_{\{|x-\tilde{x}|\leq|x|^\beta\}}\bar{u}_\ve(\tilde{x})\,d\tilde{x}
+\int_{\{|\tilde{x}|>c_0|x|^\beta\}}\frac{\bar{u}_\ve(\tilde{x})}{|\tilde{x}|^{N+2\gamma}}\,d\tilde{x}\\
&+\int_{\{|\tilde{x}|\geq c_0|x|^\beta\}}\frac{\bar{u}_\ve(x)}{|\tilde{x}|^{N+2\gamma}}.
\end{split}
\end{equation*}
 We conclude, using the definition of $\bar{u}_\ve$ and the rescaling \eqref{rescaling}, that
\begin{equation*}
\begin{split}
I_{12}&\lesssim |x|^{-2\beta\gamma}\bar{u}_\ve(x)+\ve^{N-\frac{2\gamma}{p-1}}|x|^{-\beta(N+2\gamma)}
{\int_{\big\{|\tilde{x}|
\leq\frac{|x|^\beta}{\ve}\big\}}u_1(\tilde{x})\,d\tilde{x}}
\quad+\ve^{-\frac{2p\gamma}{p-1}}
{\int_{\big\{|\tilde{x}|>\frac{|x|^\beta}{\ve}\big\}}\frac{u_1(\tilde{x})}{|\tilde{x}|^{N+2\gamma}}
\,d\tilde{x}}\\&
+|x|^{-2\gamma \beta}\bar{u}_\ve(x).
\end{split}
\end{equation*}
For $I_{13}$, one has
\begin{equation*}
\begin{split}
I_{13}\leq C\int_{\mathcal{T}_\sigma}|\bar{u}_\ve(x)|+|\bar{u}_\ve(\tilde{x})|d\tilde{x}
\leq C(\bar{u}_\ve(x)+\ve^{N-\frac{2p\gamma}{p-1}}(1+|x|)^{-(N-2\gamma)}).
\end{split}
\end{equation*}
We look now into the main term $I_{11}$, for which we need to be more precise,
\begin{equation*}\begin{split}
I_{11}&=P.V. \int_{\{|y-\tilde{y}|\leq |x|^\beta\}}\int_{\{|\tilde{x}|<\sigma\}}
\frac{\bar{u}_\ve(x)-\bar{u}_\ve(\tilde{x})}{|z-\tilde{z}|^{n+2\gamma}}\,d\tilde{z}\\
&=P.V.\Big[\int_{\{|y-\tilde{y}|\leq |x|^\beta\}}\int_{\{|x-\tilde{x}|\leq|x|^\beta\}\cap \{|\tilde{x}|<\sigma\}}\cdots +\int_{\{|y-\tilde{y}|\leq |x|^\beta\}}\int_{\{|x-\tilde{x}|>|x|^\beta\}\cap \{|\tilde{x}|<\sigma\}}\cdots\Big]\\
&=:I_{111}+I_{112}.
\end{split}\end{equation*}
Let us estimate these two integrals. First, since for $|x| $ small, $|x-\tilde{x}|<|x|^\beta$ implies that $|\tilde{x}|<\sigma$, we have
\begin{equation*}\begin{split}
k_{n,\gamma}I_{111}&= k_{n,\gamma}\,P.V. \int_{\{|y-\tilde{y}|\leq |x|^\beta\}}\int_{\{|x-\tilde{x}|\leq|x|^\beta\}}
\frac{\bar{u}_\ve(x)-\bar{u}_\ve(\tilde{x})}{|z-\tilde{z}|^{n+2\gamma}}\,d\tilde{z}\\
&=k_{n,\gamma}\,P.V.\int_{\{|y-\tilde{y}|\leq|x|^\beta\}}\int_{\{|x-\tilde{x}|\leq|x|^\beta\}}
\frac{\bar{u}_\ve(x)-\bar{u}_\ve(\tilde{x})}
{[|x-\tilde{x}|^2+|y-\tilde{y}|^2]^{\frac{n+2\gamma}{2}}}
(1+O(|\tilde{x}|)+O(|y-\tilde{y}|))\,d\tilde x\,d\tilde y\\
&=(1+O(|x|^\beta))k_{n,\gamma}\,P.V.
\int_{\{|x-\tilde{x}|\leq|x|^\beta\}}
\frac{\bar{u}_\ve(x)-\bar{u}_\ve(\tilde{x})}{|x-\tilde{x}|^{N+2\gamma}}
\int_{\{|y|\leq\frac{|x|^\beta}{|x-\tilde{x}|}\}}\frac{1}{(1+|y|^2)^{\frac{n+2\gamma}{2}}}\,dy\,d\tilde{x}\\
&=(1+O(|x|^\beta))k_{n,\gamma}\,P.V.\int_{\big\{|x-\tilde{x}|\leq|x|^\beta\big\}}
\frac{\bar{u}_\ve(x)-\bar{u}_\ve(\tilde{x})}{|x-\tilde{x}|^{N+2\gamma}}\\
&\qquad\cdot \Big[\int_{\R^k}\frac{1}{(1+|y|^2)^{\frac{n+2\gamma}{2}}}\,dy
-\int_{\big\{|y|>\frac{|x|^\beta}{|x-\tilde{x}|}\big\}}
\frac{1}{(1+|y|^2)^{\frac{n+2\gamma}{2}}}\,dy\Big]\,d\tilde{x}.\\
\end{split}\end{equation*}
Recall relation \eqref{relation}, then
\begin{equation*}\begin{split}
k_{n,\gamma}I_{111}&=(1+O(|x|^\beta))k_{N,\gamma} P.V.\int_{\{|x-\tilde{x}|\leq|x|^\beta\}}\frac{\bar{u}_\ve(x)-\bar{u}_\ve(\tilde{x})}{|x-\tilde{x}|^{N+2\gamma}}\,d\tilde{x}\\
&\quad+O(1)\int_{\{|x-\tilde{x}|\leq|x|^\beta\}}\frac{\bar{u}_\ve(x)-\bar{u}_\ve(\tilde{x})}
{|x-\tilde{x}|^{N+2\gamma}}\Big(\frac{|x|^\beta}{|x-\tilde{x}|}\Big)^{-(N+2\gamma)}\,d\tilde{x}\\
&=(1+O(|x|^\beta))k_{N,\gamma}P.V.\int_{\R^N}\frac{\bar{u}_\ve(x)-\bar{u}_\ve(\tilde{x})}{|x-\tilde{x}|^{N+2\gamma}}\,d\tilde{x}\\
&\quad+O(1)\int_{\{|x-\tilde{x}|>|x|^\beta\}}\frac{\bar{u}_\ve(x)-\bar{u}_\ve(\tilde{x})}{|x-\tilde{x}|^{N+2\gamma}}\,d\tilde{x}\\
&\quad+O(1)\int_{\{|x-\tilde{x}|\leq|x|^\beta\}}\frac{\bar{u}_\ve(x)-\bar{u}_\ve(\tilde{x})}
{|x-\tilde{x}|^{N+2\gamma}}\Big(\frac{|x|^\beta}{|x-\tilde{x}|}\Big)^{-(N+2\gamma)}\,d\tilde{x}.
\end{split}\end{equation*}
Using the definition of the fractional Laplacian in $\mathbb R^N$,
\begin{equation*}
\begin{split}
k_{n,\gamma}I_{111}&=(1+O(|x|^\beta))(-\Delta_x)^\gamma \bar{u}_\ve(x)\\
&\quad+O(1)\int_{\{|x-\tilde{x}|>|x|^\beta\}}\frac{\bar{u}_\ve(x)-\bar{u}_\ve(\tilde{x})}{|x-\tilde{x}|^{N+2\gamma}}\,d\tilde{x}
+O(1)|x|^{-\beta(N+2\gamma)}\int_{\{|x-\tilde{x}|\leq|x|^\beta\}}(\bar{u}_\ve(x)-\bar{u}_\ve(\tilde{x}))\,d\tilde{x}.\\
\end{split}\end{equation*}
Now we use a similar argument to that of $I_{12}$, which yields
\begin{equation*}
\begin{split}
\int_{\{|x-\tilde{x}|>|x|^\beta\}}\frac{\bar{u}_\ve(x)-\bar{u}_\ve(\tilde{x})}{|x-\tilde{x}|^{N+2\gamma}}\,d\tilde{x}
&\lesssim\int_{\{|\tilde{x}|>|x|^\beta\}}\frac{\bar{u}_\ve(x)}{|\tilde{x}|^{N+2\gamma}}\,d\tilde{x}
+\int_{\{|\tilde{x}|>|x|^\beta\}}\frac{\bar{u}_\ve(\tilde{x})}{|\tilde{x}|^{N+2\gamma}}\,d\tilde{x}\\
&\lesssim|x|^{-2\beta\gamma}\bar{u}_\ve(x)+\ve^{-\frac{2p\gamma}{p-1}}
\int_{\big\{\frac{|x|^\beta}{\ve}<|\tilde{x}|<\frac{\sigma}{\ve}\big\}}
\frac{u_1(\tilde{x})}{|\tilde{x}|^{N+2\gamma}}\,d\tilde{x}
\end{split}\end{equation*}
 and also,
\begin{equation*}
\begin{split}
|x|^{-\beta(N+2\gamma)}&\int_{\{|x-\tilde{x}|\leq|x|^\beta\}}(\bar{u}_\ve(x)-\bar{u}_\ve(\tilde{x}))\,d\tilde{x}\\
&\lesssim|x|^{-\beta(N+2\gamma)}\Big[ |x|^{\beta N}\bar{u}_\ve(x)+\ve^{N-\frac{2\gamma}{p-1}}\int_{\{|\tilde{x}|\leq\frac{|x|^\beta}{\ve}\}}u_1(\tilde{x})\,d\tilde{x}\Big]\\
&\lesssim|x|^{-2\beta\gamma}\bar{u}_\ve(x)+\ve^{N-\frac{2\gamma}{p-1}}|x|^{-\beta(N+2\gamma)}
\int_{\{|\tilde{x}|\leq\frac{|x|^\beta}{\ve}\}}u_1(\tilde{x})\,d\tilde{x}.
\end{split}
\end{equation*}
In conclusion, one has
\begin{equation*}\begin{split}
k_{n,\gamma}I_{111}=(1+&O(|x|^\beta))(-\Delta_x)^\gamma \bar{u}_\ve(x)+O(1)\Big[|x|^{-2\beta\gamma}\bar{u}_\ve(x)\\
&+\ve^{-\frac{2p\gamma}{p-1}}
\int_{\{\frac{|x|^\beta}{\ve}\leq|\tilde{x}|<\frac{\sigma}{\ve}\}}\frac{u_1(\tilde{x})}{|\tilde{x}|^{N+2\gamma}}\,d\tilde{x}
+\ve^{N-\frac{2\gamma}{p-1}}|x|^{-\beta(N+2\gamma)}\int_{\{|\tilde{x}|\leq\frac{|x|^\beta}{\ve}\}}u_1(\tilde{x})\,d\tilde{x}\Big].
\end{split}
\end{equation*}
Next, for $I_{112}$ we calculate similarly
\begin{equation*}\begin{split}
I_{112}&\lesssim\int_{\{|y-\tilde{y}|\leq|x|^\beta\}}\int_{\{|x-\tilde{x}|>|x|^\beta\}}
\frac{\bar{u}_\ve(x)-\bar{u}_\ve(\tilde{x})}{[|x-\tilde{x}|^2+|y-\tilde{y}|^2]^\frac{n+2\gamma}{2}}\,d\tilde{x}\,d\tilde{y}\\
&=\int_{\{|x-\tilde{x}|>|x|^\beta\}}\frac{\bar{u}_\ve(x)-\bar{u}_\ve(\tilde{x})}{|x-\tilde{x}|^{N+2\gamma}}
\int_{\{|y|\leq\frac{|x|^\beta}{|x-\tilde{x}|}\}}\frac{1}{(1+|y|^2)^{\frac{n+2\gamma}{2}}}\,dy\,d\tilde{x}\\
&=\int_{\{|x-\tilde{x}|>|x|^\beta\}}\frac{\bar{u}_\ve(x)-\bar{u}_\ve(\tilde{x})}{|x-\tilde{x}|^{N+2\gamma}}
\Big(\frac{|x|^\beta}{|x-\tilde{x}|}\Big)^k\,d\tilde{x}\\
&=|x|^{-2\beta\gamma}\bar{u}_\ve(x)+\ve^{-k-\frac{2p\gamma}{p-1}}|x|^{\beta k}\int_{\{|\tilde{x}|>\frac{|x|^\beta}{\ve}\}}\frac{u_1(\tilde{x})}{|\tilde{x}|^{n+2\gamma}}\,d\tilde{x}.
\end{split}\end{equation*}
Combining the estimates for $I_{111}, I_{112}$ and $I_{12}, I_{13}$ we obtain
\begin{equation*}\begin{split}
(-\Delta_z)^\gamma \bar{u}_\ve(x)&=(1+O(|x|^\beta))(-\Delta_x)^\gamma \bar{u}_\ve(x)+O(1)\Big[|x|^{-2\beta\gamma}\bar{u}_\ve(x)\\
&\quad+\ve^{-\frac{2p\gamma}{p-1}}\int_{\{\frac{|x|^\beta}{\ve}<|\tilde{x}|
<\frac{\sigma}{\ve}\}}\frac{u_1(\tilde{x})}{|\tilde{x}|^{N+2\gamma}}d\tilde{x}
+\ve^{N-\frac{2\gamma}{p-1}}|x|^{-\beta(N+2\gamma)}\int_{\{|\tilde{x}|<\frac{|x|^\beta}{\ve}\}}u_1(\tilde{x})\,d\tilde{x}\\
&\quad+\ve^{-k-\frac{2p\gamma}{p-1}}|x|^{\beta k}\int_{\{|\tilde{x}|>\frac{|x|^\beta}{\ve}\}}\frac{u_1(\tilde{x})}{|\tilde{x}|^{n+2\gamma}}\,d\tilde{x}
+O(\ve^{N-\frac{2p\gamma}{p-1}}(1+|x|)^{-(N-2\gamma)})\Big]\\
&=(1+O(|x|^\beta))(-\Delta_x)^\gamma \bar{u}_\ve(x)+O(1)|x|^{-2\beta\gamma}\bar{u}_\ve(x)+\mathcal R_1.
\end{split}
\end{equation*}
In order to estimate $\mathcal R_1$ we use the asymptotic behavior of $u_1(x)$ at $0$ and $\infty$. By direct computation one sees that
\begin{equation*}
\mathcal R_1(x)=\begin{cases}
|x|^{-\beta\frac{2p\gamma}{p-1}}, \quad &\text{if } |x|^\beta<\ve,\\
\ve^{N-\frac{2p\gamma}{p-1}}|x|^{-\beta N}, \quad &\text{if }|x|^\beta>\ve.
\end{cases}
\end{equation*}
The choice $\beta=\frac{1}{2}$ yields that $\mathcal R_1=O(|x|^{-\gamma})\bar{u}_\ve(x)$, and thus
\begin{equation*}
(-\Delta_z)^\gamma\bar{u}_\ve(z)=(1+O(|x|^\beta))(-\Delta_x)^\gamma \bar{u}_\ve(x)
+O(1)|x|^{-\gamma}\bar{u}_\ve(x).
\end{equation*}
Finally, recall that $\bar{u}_\ve(x)=\chi_d(x)u_\ve(x)$, then by the estimates in the previous subsection (\ref{error1}), one has
\begin{equation}
|f_{\ve}(z)|\lesssim|x|^{\frac{1}{2}}|(-\Delta_x)^\gamma\bar{u}_\ve(x)|+|x|^{-\gamma}\bar{u}_{\ve}(x)+\mathcal{E},
\end{equation}
where the weighted norm of $\mathcal{E}$ can be bounded by $\ve^{N-\frac{2p\gamma}{p-1}}$.

For $z\in \R^n \setminus \mathcal{T}_{\frac{\sigma}{2}}$, the estimate is similar to the isolated singularity case, we omit the details here. Then we may conclude
\begin{equation*}
\|f_\ve\|_{\mathcal C^{0,\alpha}_{\tilde\mu-2\gamma,\tilde\nu-2\gamma}}\leq C\ve^{q},
\end{equation*}
where $q=\min\big\{\frac{(p-3)\gamma}{p-1}-\tilde\mu, \frac{1}{2}-\gamma+\frac{(p-3)\gamma}{p-1}-\tilde\mu, N-\frac{2p\gamma}{p-1}\big\}$, and it is positive if $-\frac{2\gamma}{p-1}<\tilde\mu<\min\big\{\gamma-\frac{2\gamma}{p-1}, \frac{1}{2}-\frac{2\gamma}{p-1}\big\}$.
\end{proof}

{\begin{remark}\label{remark}
In general, in terms of the local Fermi coordinates $(x,y)$ around a fixed $z_0=(0,0)\in\Sigma$, for $u\in \mathcal C^{\alpha+2\gamma}_{\tilde{\mu},\tilde{\nu}}(\R^n \setminus \Sigma)$, one has the following estimate:
\begin{equation*}
(-\Delta)^\gamma u=(-\Delta_{\R^n\setminus \R^k})^\gamma u(x,y)+|x|^{\tau}\|u\|_*
\end{equation*}
for $|x|\ll1, |y|\ll1$, and some $\tau>\tilde{\mu}-2\gamma$. Indeed, similar to the estimates in Lemma \ref{lemma:error2}, except the main term in $I_{111}$, in the estimates, it suffices to control the terms $u(\tilde{x})$ by $\|u\|_*|x|^{\tilde{\mu}}$.
\end{remark}}

\section{Hardy type operators with fractional Laplacian}\label{section:Hardy}

Here we give a formula for the Green's function for the Hardy type operator in $\mathbb R^N$,
\begin{equation}\label{Hardy-operator}
 L\phi:=(-\Delta_{\mathbb R^N})^{\gamma}\phi-\frac{\kappa}{r^{2\gamma}}\phi,
\end{equation}
where $\kappa\in\mathbb R$. In the notation of Section \ref{section:isolated-singularity},  after the conjugation \eqref{conjugation1} we may study the equivalent operator
\begin{equation*}
\tilde {\mathcal L} w:=
e^{-\frac{N+2\gamma}{2}t}\mathcal L(e^{\frac{N-2\gamma}{2}t}w)=P_\gamma^{g_0} w-\kappa w \quad\text{on }\mathbb R\times\mathbb S^{N-1}
\end{equation*}
for $\phi=e^{\frac{N-2\gamma}{2}}w$. Consider the  projections over spherical harmonics: for $m=0,1\ldots$, let $w_m$ be a solution to
\begin{equation}\label{equation-introduction-m}
\tilde {\mathcal L}_m w:=P_\gamma^m w_m-\kappa w_m=h_m\quad\text{on }\mathbb R.
\end{equation}
Recall Proposition \ref{prop:symbol}. Then, in Fourier variables, equation \eqref{equation-introduction-m} simply becomes
\begin{equation*}
(\Theta_\gamma^m(\xi)-\kappa)\hat w_m=\hat h_m.
\end{equation*}
The behavior of this equation depends on the zeroes of the symbol $\Theta^m_\gamma(\xi)-\kappa$. In any case, we can formally write
\begin{equation}\label{w_m:Fourier}
w_m(t)=\int_{\mathbb R} \frac{1}{\Theta_\gamma^m(\xi)-\kappa} \hat h_m(\xi) e^{i\xi t}\,d\xi=\int_{\mathbb R} h_m(t') \mathcal G_m(t-t')\,dt',
\end{equation}
where the Green's function for the problem is formally given by
\begin{equation}\label{Green100}
\mathcal G_m(t)=\int_{\mathbb R} e^{i\xi t} \frac{1}{\Theta_\gamma^m(\xi)-\kappa}\,d\xi.
\end{equation}
Let us make this statement rigorous in the stable case (this is, below the Hardy constant \eqref{Hardy-constant}), by taking into account the poles that may occur in the above integral. For this, let us define
\begin{equation}\label{Am}
A_m=\tfrac{1}{2}+\tfrac{\gamma}{2}
+\tfrac{1}{2}\sqrt{\big(\tfrac{N}{2}-1\big)^2+\mu_m},\quad B_m=\tfrac{1}{2}-\tfrac{\gamma}{2}+\tfrac{1}{2}\sqrt{\big(\tfrac{N}{2}-1\big)^2+\mu_m},
\end{equation}
and observe that the symbol
\begin{equation*}\label{symbol1}
\Theta_\gamma^m(\xi)=2^{2\gamma}\frac{\big|\Gamma\big(A_m+\tfrac{\xi}{2}i\big)\big|^2}
{\big|\Gamma\big(B_m+\tfrac{\xi}{2}i\big)\big|^2}
=2^{2\gamma}\frac{\Gamma\big(A_m+\tfrac{\xi}{2}i\big)\Gamma\big(A_m-\tfrac{\xi}{2}i\big)}
{\Gamma\big(B_m+\tfrac{\xi}{2}i\big)\Gamma\big(B_m-\tfrac{\xi}{2}i\big)}
\end{equation*}
 can be extended meromorphically to the complex plane; this extension will be denoted simply by
\begin{equation*}
\Theta_m(z)
:=2^{2\gamma}\frac{\Gamma\big(A_m+\frac{z}{2}i\big)\Gamma\big(A_m-\frac{z}{2}i\big)}
{\Gamma\big(B_m+\frac{z}{2}i\big)\Gamma\big(B_m-\frac{z}{2}i\big)},
\end{equation*}
for  $z\in \mathbb C$. We have:

\begin{teo}\label{thm:Hardy-potential}
Let $0\leq \kappa<\Lambda_{N,\gamma}$ and fix $m=0,1,\ldots$. Assume that the right hand side $h_m$ in \eqref{equation-introduction-m} satisfies
\begin{equation}\label{decay-h}
h_m(t)=\begin{cases}
O(e^{-\delta t}) &\text{as } t\to+\infty,\\
O(e^{\delta_0 t}) &\text{as } t\to-\infty,
\end{cases}
\end{equation}
for some real constants $\delta,\delta_0$. It holds:
\begin{itemize}
\item[\emph{i.}] The function $\frac{1}{\Theta_m(z)-\kappa}$ is meromorphic in $z\in\mathbb C$. Its poles are located at points of the form $\tau_j\pm i\sigma_j$ and $-\tau_j\pm i\sigma_j$, where $\sigma_j>\sigma_0>0$ for $j=1,\ldots$, and $\tau_j\geq0$ for $j=0,1,\ldots$. In addition, $\tau_0=0$, and $\tau_j=0$ for $j$ large enough. For such $j$, $\{\sigma_j\}$ is an increasing sequence with no accumulation points.

\item[\emph{ii.}] If $\delta>0$, $\delta_0\geq 0$, then a particular solution of \eqref{equation-introduction-m} can be written as
\begin{equation}\label{Green0}
w_m(t)=\int_{\mathbb R} h_m(t') {\mathcal G}_m(t-t')\,dt',
\end{equation}
where
\begin{equation}\label{Green1}
{\mathcal G}_m(t)=d_0e^{-\sigma_0 |t|}+\sum_{j=1}^\infty  e^{-\sigma_j |t|} [d_j\cos(\tau_j |t|)+d_j'\sin(\tau_j|t|)],
\end{equation}
for some constants $d_j,d_j'$ given precisely in the proof.  Moreover, ${\mathcal G}_m$ is an even $\mathcal C^{\infty}$ function when $t\neq 0$.

\item[\emph{iii.}] Now assume only that $\delta+\delta_0\geq 0$. If $\sigma_J<\delta<\sigma_{J+1}$ (and thus $\delta_0>-\sigma_{J+1})$, then a particular solution is
    \begin{equation}\label{solucion2}
w_m(t)=\int_{\mathbb R} h_m(t') {\widetilde{\mathcal G}}_m(t-t')\,dt'
\end{equation}
where
\begin{equation*}\label{Green2}
\widetilde{\mathcal G}_m(t)=\sum_{j=J+1}^\infty e^{-\sigma_j t} [d_j  \cos(\tau_j t)+d_j'\sin(\tau_j t)], \quad\text{for}\quad t>0.
\end{equation*}
 $\widetilde{\mathcal G}_m$ is a $\mathcal C^{\infty}$ function when $t\neq 0$. In addition, $w_m$ satisfies the asymptotic behavior
\begin{equation}\label{equation100}
w_m(t)=O(e^{-\delta t}) \quad\text{as} \quad t\to +\infty, \quad w_m(t)\sim \sum_{j=0}^J D_j(t) e^{- \sigma_j t}+O(e^{\delta_0 t}) \quad\text{as} \quad t\to -\infty,
\end{equation}
for some functions $D_0(t),\ldots,D_J(t)$. In the case $\tau_j=0$, we have
\begin{equation*}
D_j(t)\to \tilde D_j:=d_j\int_{\mathbb R}e^{\sigma_j t'}h_m(t')\,dt', \quad\text{as}\quad t\to -\infty,
\end{equation*}
while when $\tau_j\neq 0$,
\begin{equation*}
\begin{split}
D_j(t)&=d_j\Big[\cos(\tau_j t)\int_{t}^{+\infty}e^{\sigma_j t'}\cos(\tau_j t')h(t')\,dt'+\sin(\tau_j t)\int_{t}^{+\infty}e^{\sigma_j t'}\sin(\tau_j t')h(t')\,dt'\Big]\\
&+d_j'\Big[\sin(\tau_j t)\int_{t}^{+\infty}e^{\sigma_j t'}\cos(\tau_j t')h(t')\,dt'-\cos(\tau_j t)\int_{t}^{+\infty}e^{\sigma_j t'}\sin(\tau_j t')h(t')\,dt'\Big].
\end{split}
\end{equation*}
\end{itemize}
\end{teo}

\begin{remark}\label{remark:homogeneous}
All solutions of \eqref{equation-introduction-m} are obtained from the particular solution $w_m$ above by adding those of the homogeneous problem $\tilde{\mathcal L}_m w=0$, which are of the form
\begin{equation*}
\begin{split}
w_{m,h}(t)&=C_0^-e^{-\sigma_0 t}+C_0^+e^{\sigma_0 t}+\sum_{j=1}^\infty  e^{-\sigma_j t}
[C_j^-\cos(\tau_j t)+C_j'^-\sin(\tau_j t)]\\
&+\sum_{j=1}^\infty  e^{+\sigma_j t}[C_j^+\cos(\tau_j t)+C_j'^+\sin(\tau_j t)]
\end{split}
\end{equation*}
for some real constants $C_j^-,C_j^+,C_j'^-,C_j'^+$, $j=0,1,\ldots$.
\end{remark}

We also look at the case when $\kappa$ leaves the stability regime. In order to simplify the presentation, we only consider the projection $m=0$ and the equation
\begin{equation}\label{equation-radial}
\tilde{\mathcal L}_0 w=h.
\end{equation}
In addition,  we assume that only the first pole leaves the stability regime, which happens if $\Lambda_{N,\gamma}<\kappa<\Lambda'_{N,\gamma}$  for some $\Lambda'_{N,\gamma}$. Then, in addition to the poles above, we will have two real poles $\tau_0$ and $-\tau_0$. Some study regarding $\Lambda'_{N,\gamma}$ will be given in the next section but we are not interested in its explicit formula.

 \begin{proposition}\label{prop:unstable}
Let $\Lambda_{N,\gamma}<\kappa<\Lambda'_{N,\gamma}$. Assume that $h$ decays like $O(e^{-\delta t})$ as $t\to\infty$, and $O(e^{\delta_0 t})$ as $t\to-\infty$ for some real constants $\delta,\delta_0$. It holds:
\begin{itemize}
\item[{\textit i.}] The function $\frac{1}{\Theta_m(z)-\kappa}$ is  meromorphic in $z\in\mathbb C$. Its poles are located at points of the form $\tau_j\pm i\sigma_j$ and $-\tau_j\pm i\sigma_j$, where $\tau_j,\sigma_j\geq 0$, for $j=0,1,\ldots$. In addition, $\sigma_0=0$, and $\tau_j=0$ for $j$ large enough. For such $j$, $\{\sigma_j\}$ is an increasing sequence with no accumulation points.

\item[\textit{ii.}]
If $\delta>0$, $\delta_0> 0$, then
a particular solution of \eqref{equation-radial} can be written as
\begin{equation}\label{v0}
w_0(t)=\int_{\mathbb R} h(t') {\mathcal G_0}(t-t')\,dt',
\end{equation}
where
\begin{equation*}
\mathcal G_0(t)=d_0 \sin({\tau_0t})\chi_{(-\infty,0)}(t)+\sum_{j=1}^\infty e^{-\sigma_j |t|}[d_j\cos(\tau_j t)+d_j'\sin(\tau_j |t|)]
\end{equation*}
for some constants $d_0,d_j,d_j'$, $j=1,\ldots$. Moreover, $\mathcal G_0$ is an even $\mathcal C^{\infty}$ function when $t\neq 0$.

\item[\textit{iii.}] The analogous statements to  Theorem \ref{thm:Hardy-potential}, {\em iii.},  and Remark \ref{remark:homogeneous} hold.
\end{itemize}
\end{proposition}

Further study of fractional non-linear equations with critical Hardy potential has been done in \cite{AMPP,Dipierro-Montoro-Peral-Sciunzi}, for instance.

\bigskip

Before we start the proof of Theorem \ref{thm:Hardy-potential}, let us give some preliminaries regarding the symbol the symbol $\Theta_m$.

\begin{remark}
It is interesting to observe that
$$\Theta_m(z)=\Theta_m(-z).$$
Moreover, thanks to Stirling formula (expression 6.1.37 in \cite{Abramowitz})
\begin{equation}\label{Stirling}\Gamma(z)\sim e^{-z}z^{z-\frac{1}{2}}(2\pi)^{\frac{1}{2}},\quad \text{as}\quad |z|\to \infty \text{ in } |\arg\, z|<\pi,\end{equation}
one may check that for $\xi\in\mathbb R$,
\begin{equation}\label{symbol-limit}
\Theta_m(\xi)\sim |m+\xi i|^{2\gamma},\quad \text{as}\quad |\xi|\to\infty,
\end{equation}
and this limit is uniform in $m$. Here the symbol $\sim$ means that one can bound one quantity, above and below, by constant times the other. This also shows that, for fixed $m$, the behavior at infinity is the same as the one for the standard fractional Laplacian $(-\Delta)^\gamma$.
\end{remark}

The following proposition uses this idea to study the behavior as $|t|\to 0$. Recall that the Green's function for the fractional Laplacian $(-\Delta_{\mathbb R})^\gamma$ in one space dimension is precisely
$$G(t)=|t|^{-(1-2\gamma)}.$$
We will prove that $\mathcal G_m$ has a similar behavior.

\begin{proposition} Let $\gamma\in(0,1/2)$. Then
\begin{equation*}
\lim_{|t|\to 0} \frac{\mathcal G_m(t)}{|t|^{-(1-2\gamma)}}=c
\end{equation*}
for some positive constant $c$.
\end{proposition}

\begin{proof}
Indeed, recalling \eqref{symbol-limit}, we have
\begin{equation*}
\begin{split}
\lim_{|t|\to 0} \frac{\int_{\mathbb R} \frac{1}{\Theta_m(\xi)-\lambda}\,e^{i\xi t}\,d\xi}{|t|^{-(1-2\gamma)}}&=\lim_{|t|\to 0}\int_{\mathbb R} \frac{1}{|t|^{2\gamma}\big[\Theta_m(\frac{\zeta}{t})-\lambda\big]}e^{i\zeta} \,d\zeta\\
&=\lim_{t\to 0}\left[\int_{\{|\zeta|>t^\delta\}}\ldots+\int_{\{|\zeta|\leq t^\delta\}}\ldots\right]=:\lim_{t\to 0}[I_1+I_2]
\end{split}
\end{equation*}
for some $2\gamma<\delta<1$.

For $I_1$, we use Stirling's formula \eqref{Stirling} to estimate
\begin{equation*}
I_1\sim \int_{\{|\zeta|>t^\delta\}}\frac{\cos(\zeta)}{|\zeta|^{2\gamma}}\,d\zeta\to c\quad\text{as}\quad t\to 0,
\end{equation*}
while for $I_2$,
\begin{equation*}
    |I_2|\leq \int_{\{|\zeta|\leq t^\delta\}} \frac{1}{|t|^{2\gamma}\big[\Theta_m(0)-\lambda\big]}\,d\zeta \to 0\quad\text{as}\quad t\to 0,
\end{equation*}
as desired.
\end{proof}

\begin{lemma}\label{lemma-increasing}
Define the function
$\Phi(x,\xi)=2^{2\gamma}\frac{\big|\Gamma\big(A_m+x+\tfrac{\xi}{2}i\big)\big|^2}
{\big|\Gamma\big(B_m+x+\tfrac{\xi}{2}i\big)\big|^2}.$
Then:
\begin{itemize}
\item[\emph{i.}] Fixed $x>B_m$,  $\Psi (x,\xi)$ is a (strictly) increasing function of $\xi>0$.
\item[\emph{ii.}] $\Psi(x,0)$ is a (strictly) increasing function of $x>0$.
\end{itemize}
\end{lemma}
\begin{proof}
As in \cite{DelaTorre-Gonzalez}, section 7, one calculates using \eqref{expansion-digamma},
\begin{equation*}\begin{split}
\partial_\xi (\log \Theta_m( \xi))&=\text{Im} \left\{ \psi\left( B_m+x+\tfrac{\xi}{2} i) \right)-\psi\left( A_m +x+\tfrac{\xi}{2} i \right)\right\}\\
& = c \,\Imag\sum_{l=0}^\infty \left(\frac{1}{l+A_m+x+\frac{\xi}{2} i}-\frac{1}{l+B_m+x+\frac{\xi}{2} i}\right)>0,
\end{split}\end{equation*}
as claimed. A similar argument yields the monotonicity in $x$.
\end{proof}

Now we give the proof of Theorem \ref{thm:Hardy-potential}. Before we consider the general case, let us study first when $\kappa=0$, for which $\mathcal G_m$ can be computed almost explicitly. Fix $m=0,1,\ldots$. The poles of the function $\frac{1}{\Theta_m(z)}$ happen at points $z\in\mathbb C$ such that
$$\pm \frac{z}{2} i+B_m=-j,\quad \text{for }j\in \mathbb N\cup \{0\},$$
 i.e, at points $\{\pm i\sigma_j\}$ for
 \begin{equation}\label{poles-theta}\sigma_j:=
 2(B_m+j),\quad j=0,1,\ldots.
 \end{equation}
Then the integral in \eqref{Green100}
can be computed in terms of the usual residue formula. Define the region in the complex plane
\begin{equation}\label{contour}
\Omega_+=\{z\in\mathbb C\,:\, |z|<R, \Imag z>0 \}.
\end{equation}
A standard contour integration along $\partial\Omega_+$ gives, as $R\to \infty$, that
\begin{equation}\label{residue-formula}
\mathcal G_m(t)=2\pi i \sum_{j=0}^\infty \Res\Big(e^{i z t} \frac{1}{\Theta_m(z)},i\sigma_j\Big)=2\pi i \sum_{j=0}^\infty e^{- \sigma_j t}c_j,
\end{equation}
where $c_j=c_j(m)$ is the residue  of the function $\frac{1}{\Theta_m(z)}$ at the pole $i\sigma_j$.
This argument is valid as long as the integral in the upper semicircle tends to zero as $R\to \infty$. This happens when $t>0$ since $|e^{iz t}|=e^{-t\Imag z}$. For $t<0$, we need to modify the contour of integration to $\Omega_{-}=\{z\in\mathbb C\,:\, |z|<R, \Imag z<0 \}$,
and we have that, for $t<0$,
\begin{equation*}
\mathcal G_m(t)=2\pi i \sum_{j=0}^\infty c_j e^{\sigma_j t},
\end{equation*}
which of course gives that $\mathcal G_m$ is an even function in $t$. In any case $\mathcal G_m$ is exponentially decaying as  $|t|\to\infty$ with speed given by the first pole $|\sigma_0|= 2B_m$.

In addition, recalling the formula for the residues of the Gamma function from \eqref{residue-Gamma},
we have that
\begin{equation*}
\begin{split}
c_j&=\frac{1}{2^{2\gamma}}\, \frac{\Gamma(2B_m+j)}{\Gamma(A_m-B_m-j)\Gamma(A_m +B_m+j)}\lim_{z\to i \sigma_j} \Gamma(B_m+\tfrac{z}{2}i)(z-i\sigma_j)\\
&=\frac{2}{2^{2\gamma}}\frac{\Gamma\big(1-\gamma+\sqrt{(\frac{n}{2}-1)^2+\mu_m}+j\big)}
{\Gamma(\gamma-j)\Gamma\big(1+\sqrt{(\frac{n}{2}-1)^2+\mu_m}+j\big)}\frac{-i(-1)^{j}}{j!}
\end{split}\end{equation*}
for $j\geq 1$, which yields the (uniform) convergence of the series \eqref{Green1} by Stirling's formula \eqref{Stirling}.

\bigskip

Next, take a general $0<\kappa<\Lambda_{n,\gamma}$. The function $e^{iz t} \frac{1}{\Theta_m(z)-\kappa}$ is meromorphic in the complex plane $\mathbb C$. Moreover, if $z$ is a root of $\Theta_m(z)=\kappa$, so are $-z$, $\bar z$ and $-\bar z$.

Let us check then that there are no poles on the real line. Indeed, the first statement in Lemma \ref{lemma-increasing} implies that is enough to show that
\begin{equation*}\Theta_m(0)-\kappa>0.\end{equation*}
But again, from the second statement of the lemma, $\Theta_m(0)>\Theta_0(0)$, so we only need to look at the case $m=0$.
Finally, just note that $\Theta_0(0)=\Lambda_{N,\gamma}>\kappa$.

Next, we look for poles on the imaginary axis. For $\sigma>0$, $\Theta_m(i\sigma)=\Psi(-\sigma,0)$ and this function is (strictly) decreasing in $\sigma$. Moreover, $\Psi(0,0)=\Theta_m(0)=\Lambda_{N,\gamma}>\kappa$. Let $\sigma_0\in(0,+\infty]$ be the first point where $\Theta_m(i\sigma_0)=\kappa$. Then $\pm i\sigma_0$ are poles on the imaginary axis.
Moreover, the first statement of Lemma \ref{lemma-increasing} shows that there are no other poles in the strip $\{z\,:\,|\Imag(z)|\leq\sigma_0\}$.

\bigskip

Denote the rest of the poles by $z_j:=\tau_j+i\sigma_j$, $\tau_j-i\sigma_j$, $-\tau_j+i\sigma_j$ and $-\tau_j-i\sigma_j$, $j=1,2,\ldots$. Here we take $\sigma_j>\sigma_0>0$, $\tau_j\geq 0$. A detailed study of the poles is given in the Section \ref{subsection:technical}. In particular, for large $j$, all poles lie there on the imaginary axis, and their asymptotic behavior is similar to that of \eqref{poles-theta}.

Now we can complete the proof of Theorem \ref{thm:Hardy-potential}. Since we have shown that there is a spectral gap $\sigma_0$ from the real line, it is possible to modify the contour of integration in \eqref{residue-formula} to prove a similar residue formula:  for $t>0$,
\begin{equation}\label{formula-Green0}\begin{split}
\mathcal G_m(t)&=2\pi i \Res\Big(e^{i z t} \frac{1}{\Theta_m(z)-\kappa},i\sigma_0\Big)\\
&+2\pi i \sum_{j=1}^\infty \left[\Res\Big(e^{i z t} \frac{1}{\Theta_m(z)-\kappa},\tau_j+i\sigma_j\Big)+\Res\Big(e^{i z t} \frac{1}{\Theta_m(z)-\kappa},-\tau_j+i\sigma_j\Big)\right]\\
=&2\pi i c_0e^{-\sigma_0 t}+2\pi i \sum_{j=1}^\infty e^{-\sigma_j t} [c_j\cos(\tau_j t)+c_j'\sin(\tau_j t)].
\end{split}
\end{equation}
Here $-i(c_j'+ic_j)/2$ is the residue of the function $\frac{1}{\Theta_m(z)-\kappa}$ at the point $\tau_j+i\sigma_j$ if $\tau_j\neq0$ (if $\tau_j=0$ then we write instead $ic_j$ for this residue).

By the monotonicity properties of $\Theta_m$, 
we see that
\[\begin{split}
d_0
&=2\pi i c_0
=2\pi i
    \Res\left(\dfrac{1}{\Theta_m(z)-\kappa},i\sigma_0\right)\\
&=2\pi i\lim_{z\to i\sigma_0}
    \dfrac{z-i\sigma_0}{\Theta_m(z)-\kappa}
=-2\pi \lim_{\zeta\to \sigma_0}
    \dfrac{\zeta-\sigma_0}{\Psi(-\frac{\zeta}{2},0)-\kappa}
>0.
\end{split}\]
Using the symmetries of the  function $\Theta_m$, it can be easily shown that $c_j,c_j'$ are purely imaginary and that the residues at the points $\pm \tau_j+i\sigma_j$ (denoted by $r_{\pm}$) are related by $r_+=-\overline{r_-}$. Moreover, the asymptotic behavior for this residue is calculated in \eqref{asymptotics-residue}; indeed, $c_j\sim Cj^{-2\gamma}$. The convergence of the series is guaranteed, and we have proved the desired expression \eqref{Green1} (for $t<0$, $\mathcal G$ it is defined evenly).

We have produced a particular solution to  \eqref{equation-introduction-m}, given explicitly by \eqref{Green0}. Due to the presence of a non-trivial kernel, one has an infinite dimensional family of solutions. The main idea in the proof of \emph{iii.} is to modify \eqref{Green0} by adding elements in the kernel in order to obtain a particular solution that decays as $O(e^{-\delta t})$ as $t\to +\infty$ (at the expense of worsening the behavior as $t\to -\infty$). We obtain what is known as a multipole expansion. We start with a simple lemma:

\begin{lemma}\label{lemma:exponentials}
If $f_1(t)=O(e^{-a |t|})$ as $t\to \infty$, $f_2(t)=O(e^{-a_+ t})$ as $t\to +\infty$ and $f_2(t)=O(e^{a_{-} t})$ as $t\to -\infty$ for some $a,a_+>0$, $a_{-}>-a$, then
$$f_1*f_2 (t)=O(e^{-\min\{a,a_+\}t})\quad\text{as}\quad t\to+\infty.$$
\end{lemma}

\begin{proof}
Indeed, for $t>0$,
\begin{equation*}
\begin{split}
|f_1*&f_2(t)|=\Big|\int_{\mathbb R} f_1(t-t')f_2(t')\,dt'\Big|\\
&\lesssim \int_{-\infty}^0 e^{-a(t-t')} e^{a_{-} t'}\,dt'+\int_0^t  e^{-a (t-t')}e^{-a_+t'}\,ds+\int_t^{+\infty}e^{a (t-t')}e^{-a_+t'}\,dt'.
\end{split}
\end{equation*}
The lemma follows by straightforward computations.
\end{proof}

\begin{remark} It is interesting to observe that $a_{-}$ is not involved in the decay as $t\to +\infty$. Moreover, by reversing the role of $t$ and $-t$, it is possible to obtain the analogous statement for $t\to-\infty$ with the obvious modifications.
\end{remark}

Let us assume the bounds \eqref{decay-h}. Suppose first that $0<\delta\leq\sigma_0$. Then \eqref{Green0} is already our solution, since
\begin{equation*}
|w_m(t)|=\Big|\int_{\mathbb R}\mathcal G_m(t-t') h(t')\,dt'\Big |
\leq \int_{\mathbb R} O(e^{-\sigma_{0}|t-t'|})|h(t')|\,dt'
=O(e^{-\delta t}).
\end{equation*}
Similarly as $t\to -\infty$, $w(t)=O(e^{\min\{\delta_0,\sigma_0\}t})$.

Fix $m=0,1,\ldots$, and drop the subindex $m$ for simplicity. Now we take $\sigma_J<\delta\leq\sigma_{J+1}$ for some $J\geq 0$ and $\delta+\delta_0>0$ (and thus, $\delta_0>-\sigma_{J+1}$). The idea is to change our Fourier transform in order to integrate on a different horizontal line $\zeta\in\mathbb R+i\vartheta$ for some $\vartheta>0$. This is, for $w=w(t)$, set
\begin{equation*}
\widetilde w(\zeta)=\frac{1}{\sqrt{2\pi}}\int_{\mathbb R} e^{-i\zeta t}\, w(t)\,dt
=\frac{1}{\sqrt{2\pi}}\int_{\mathbb R} e^{-i(\xi+i\vartheta) t}\, w(t)\,dt=\hat w(\xi+i\vartheta),
\end{equation*}
whose inverse Fourier transform is
\begin{equation*}
w(t)=\frac{1}{\sqrt{2\pi}}\int_{\mathbb R+i\vartheta} e^{i\zeta t}  \,\widetilde w(\zeta)\,d\zeta,
\end{equation*}
Moreover, in the new variable $\zeta=\xi+i\vartheta$ we have that
\begin{equation*}
\widetilde h(\zeta)=\widetilde{P^{(m)}(w)}(\zeta)=(\Theta_m(\zeta)-\kappa)\widetilde w(\zeta).
\end{equation*}
Inverting this symbol we obtain a particular solution
\begin{equation*}
w(t)=\int_{\mathbb R+i\vartheta} \frac{1}{\Theta_m(\zeta)-\kappa} \widetilde h(\zeta) e^{i\zeta t}\,d\zeta
=\int_{\mathbb R} h(t') \widetilde{\mathcal G}(t-t')\,dt',
\end{equation*}
where
\begin{equation*}
\widetilde{\mathcal G}(t)=\int_{\mathbb R+i\vartheta} e^{i\zeta t} \frac{1}{\Theta_m(\zeta)-\kappa}\,d\zeta.
\end{equation*}
In order to evaluate this integral, for $t>0$ we need to replace the contour of integration from \eqref{contour} to $\partial \Omega_{\vartheta,+}$ for
\begin{equation*}
\Omega_{\vartheta,+}=\{z\in\mathbb C\,:\, |z|<R,\, \Imag z>\vartheta \},
\end{equation*}
in the case $\kappa=0$, and with the obvious modifications, for $\kappa\neq 0$.
We will take $\vartheta\in(\sigma_J,\delta)$.
Letting $R\to \infty$, this yields, similarly to \eqref{formula-Green0}, that
\begin{equation*}
\widetilde{\mathcal G}(t)=2\pi i \sum_{j=J+1}^\infty \Res\Big(e^{i z t} \frac{1}{\Theta_m(z)-\kappa},\pm\tau_j+i\sigma_j\Big)=2\pi i \sum_{j=J+1}^\infty e^{- \sigma_j t}c_j,\quad\text{for}\quad t>0,
\end{equation*}
where, as above, $c_j$ is the residue  of the function $\frac{1}{\Theta_m(z)-\kappa}$ at the pole $\tau_j+i\sigma_j$.

On the other hand,  for $t<0$, we need to change the contour of integration to $\partial \Omega_{\vartheta,-}$, where
\begin{equation*}
\Omega_{\vartheta,-} =\{z\in\mathbb C \;:\; -R<\Real(z)<R,\,-R<\Imag(z)<\vartheta\}.
\end{equation*}
We see again that, for $t<0$, $|e^{izt}|=e^{-t\Imag z}$, so the contour integral is well defined in the lower half-space $\{\Imag(z)<0\}$. However, for the region $\{\Imag(z)>0\}\cap \partial \Omega_{\vartheta,-}$ we need the additional estimate $\frac{1}{\Theta_m(z)-\kappa}=O(R^{-2\gamma})$, which follows directly from \eqref{symbol-limit}. Using the residue formula we obtain
\begin{equation*}
\begin{split}
\widetilde{\mathcal G}(t)&=2\pi i \left[\sum_{j=1}^{\infty}\Res\Big(e^{i z t} \frac{1}{\Theta_m(z)-\kappa},\pm\tau_j-i\sigma_j\Big)
+\sum_{j=0}^{J}\Res\Big(e^{i z t} \frac{1}{\Theta_m(z)-\kappa},\pm\tau_j+i\sigma_j\Big)\right]\\
&= 2\pi i\sum_{j=1}^{\infty} e^{ \sigma_j t} [c_j\cos (\tau_j t)+c_j'\sin (\tau_j t)]+2\pi i\sum_{j=1}^{J} e^{ -\sigma_j t}[c_j\cos (\tau_j t)+c_j'\sin (\tau_j t)]+2\pi ic_0e^{-\sigma_0 t},
\end{split}
\end{equation*}
for $t<0$. In order to check the asymptotic behavior of $w_m$ we write
\begin{equation*}
w_m(t)=\left[\int_{-\infty}^t +\int_{t}^{+\infty}\right]\widetilde{\mathcal  G}(t-t')h(t')\,dt'=W_1(t)+W_2(t),
\end{equation*}
where
\begin{equation*}
\begin{split}
W_1(t)&=\sum_{j=J+1}^{\infty} e^{-\sigma_j t}\int_{-\infty}^t  e^{\sigma_j t'}[d_j\cos (\tau_j (t-t'))+d_j'\sin (\tau_j (t-t'))]h(t')\,dt',\\
W_2(t)&= \sum_{j=1}^{\infty} e^{ \sigma_j t}\int_t^{+\infty} e^{ -\sigma_j t'} [d_j\cos (\tau_j (t-t'))+d_j'\sin (\tau_j (t-t'))]h(t')\,dt'\\
&+\sum_{j=0}^{J} e^{ -\sigma_j t}\int_t^{+\infty} e^{ \sigma_j t'}[d_j\cos (\tau_j (t-t'))+d_j'\sin (\tau_j (t-t'))]h(t')\,dt'.
\end{split}
\end{equation*}
Let us check its asymptotic behavior.  The growth hypothesis on $h$ from \eqref{decay-h} imply that
\begin{equation*}
\Big|\int_{\mathbb R}  h(t')\widetilde{\mathcal G}(t-t')\,dt'\Big |
=O(e^{-\delta t}),
\end{equation*}
as $t\to +\infty$. Thus we have found a different particular solution $w_m$ with the desired decay. On the other hand, when $t\to -\infty$, the worst term above is the second one in $W_2$, as expected. Let
\begin{equation*}
D_j(t)=\int_{t}^{+\infty}e^{\sigma_j t'}[d_j\cos (\tau_j (t-t'))+d_j'\sin (\tau_j (t-t'))]h(t')\,dt', \quad j=1,\ldots,J.
\end{equation*}
Let us assume for a moment, that $\tau_j=0$. Two cases can happen as $t\to -\infty$:
\begin{itemize}
\item either $D_j(t)\to \tilde D_j\neq 0$, so $w_m(t)\sim e^{-\sigma_j t}$, or
\item $D_j(-\infty)=0$, so $w_m(t)=O(e^{\delta_0 t})$,  by L'H\^{o}pital's rule,
\end{itemize}
and this completes the proof of \eqref{equation100}. When $\tau_j\neq 0$ we need to write
\begin{equation}\label{modified}
\begin{split}
D_j(t)&=d_j\Big[\cos(\tau_j t)\int_{t}^{+\infty}e^{\sigma_j t'}\cos(\tau_j t')h(t')\,dt'+\sin(\tau_j t)\int_{t}^{+\infty}e^{\sigma_j t'}\sin(\tau_j t')h(t')\,dt'\Big]\\
&+d_j'\Big[\sin(\tau_j t)\int_{t}^{+\infty}e^{\sigma_j t'}\cos(\tau_j t')h(t')\,dt'-\cos(\tau_j t)\int_{t}^{+\infty}e^{\sigma_j t'}\sin(\tau_j t')h(t')\,dt'\Big],
\end{split}
\end{equation}
but we have an analogous result.

This completes the proof of Theorem \ref{thm:Hardy-potential}.

\qed

\begin{remark}
We now look at the proof of Theorem \ref{thm:Hardy-potential}, \emph{iii.}, in terms of the variation of constants method and Fredholm theory. Starting from equations \eqref{Green0}-\eqref{Green1}, we estimate, for $t>0$,
\begin{equation}\label{equation12}
\begin{split}
\Big|\int_{\mathbb R} \Big[\mathcal G_m(t-t') -\sum_{j=0}^J  e^{-\sigma_j |t-t'|}(d_j\cos(\tau_j |t-t'|)+d_j'\sin(\tau_j |t-t'|)\Big]h(t')\,dt'\Big |&\\
\leq \int_{\mathbb R} O(e^{-\sigma_{J+1}|t-t'|})|h(t')|\,dt'
&=O(e^{-\delta t}).
\end{split}
\end{equation}
Let
\begin{equation*}
\varphi_j(t):=\int_{\mathbb R}e^{-\sigma_j |t-t'|}(d_j\cos(\tau_j |t-t'|)+d_j'\sin(\tau_j |t-t'|))h(t')\,dt', \quad j=0,\ldots,J,
\end{equation*}
so
\begin{equation}\label{formulawm}
w_m(t)=\sum_{j=0}^J  \varphi_j(t)+O(e^{-\delta t}) \quad\text{as}\quad t\to+\infty.
\end{equation}
 The underlying idea is that, if one wishes to improve the decay of \eqref{formulawm} as $t\to +\infty$, one needs to add to this particular solution  elements in the kernel, say, $C_je^{-\sigma_j}$ for a precise constant $C_j$, $j=0,\ldots,J$. But this worsens the decay at $t\to -\infty$. Let us make this computation more precise in the case that $\tau_j=0$. The general case is similar, but lengthy.
\begin{equation}\label{equation10-1}\begin{split}
\varphi_j(t)&=d_j\int_{\mathbb R} e^{-\sigma_j|t-t'|}h(t')\,dt'\\
&=d_j\int_{-\infty}^t e^{-\sigma_j(t-t')}h(t')\,dt'+d_j\int_{t}^{+\infty} e^{\sigma_j(t-t')}h(t')\,dt'.
\end{split}\end{equation}
The first integral in the right hand side above can be rewritten using that, by Fredholm theory, the following compatibility condition must be satisfied:
\begin{equation}\label{Fredholm}
0=\int_{\mathbb R} e^{\sigma_j t'}h(t')\,dt'=\int_{-\infty}^t \ldots +\int_{t}^{+\infty}\ldots.
\end{equation}
Thus we have
\begin{equation*}
\varphi_j(t)=d_j\int_{t}^{+\infty} \left[-e^{-\sigma_j(t-t')}+e^{\sigma_j(t-t')} \right]h(t')\,dt'.
\end{equation*}
Then, except for the factor $d_j$ in front, this is the standard variation of constants formula to produce a particular solution for the second order ODE
\begin{equation*}
\varphi_j''(t)=\sigma_j^2\varphi_j(t)-2\sigma_j h(t),
\end{equation*}
and in particular shows that $\varphi_j(t)$ decays like $h(t)$ as $t\to +\infty$, which is $O(e^{-\delta t})$. One could also check behavior as $t\to -\infty$, similarly to the above.
\end{remark}

\subsection{Beyond the stability regime}

Now we look at the proof of Proposition \ref{prop:unstable}. As we have mentioned, in order to simplify the presentation, we only consider the projection $m=0$. Let $\Lambda_{N,\gamma}<\kappa<\Lambda'_{N,\gamma}$ be the region where we have exactly two real poles at $\tau_0$ and $-\tau_0$, for $\tau>0$. For this, just note that, for real $\xi>0$, Lemma \ref{lemma-increasing} shows that $\Theta_0(\xi)$ is an increasing function in $\xi$, and it is even. Denote the rest of the poles as in the previous subsection, for $j=1,2,\ldots$.

We proceed as in the proof of Theorem \ref{thm:Hardy-potential} and write
\begin{equation}\label{solution-radial}
w_0(t)=\int_{\mathbb R} \frac{1}{\Theta_0(\xi)-\kappa} \hat h(\xi) e^{i\xi t}\,d\xi=\int_{\mathbb R} h(t') \mathcal G_0(t-t')\,dt'.
\end{equation}
There are two real poles now. So in order to give sense to the integral above we can integrate on a different horizontal line, just below the horizontal axis, or to regularize the integral by some $\varepsilon>0$ small. For this, calculate
\begin{equation*}
\mathcal G_0^\varepsilon(t)=\int_{\mathbb R} e^{i\xi t} \frac{1}{\Theta_0(\xi-\varepsilon i)-\kappa}\,d\xi.
\end{equation*}
The poles are now $\tau_0+\varepsilon i$ and $\tau_0-\varepsilon i$.  Define the region
$\Omega_+=\{z\in\mathbb C\,:\, |z-(\tau+\varepsilon i)|<R, \Real{(z)}>0 \}$.
A standard contour integration along $\partial\Omega_+$ gives, as $R\to \infty$, that for $t>0$,
\begin{equation*}
\mathcal G_0^\varepsilon(t)=2\pi i c_0^\varepsilon e^{i(\tau_0+\varepsilon i) t}+2\pi i \sum_{j=1}^\infty e^{-\sigma^\varepsilon_j t}  [c_j^\varepsilon\cos(\tau^\varepsilon_j t) +c_j'^\varepsilon\sin(\tau^\varepsilon_j t)]
\end{equation*}
where
\begin{equation*}
c_0^\varepsilon=\Res\Big( \frac{1}{\Theta_0(z-\varepsilon i)-\kappa},\tau_0+\varepsilon i\Big).
\end{equation*}
Taking the limit $\varepsilon \to 0$,
\begin{equation*}
\mathcal G_0^\varepsilon(t)\to \mathcal G_0(t)=2\pi i [c_0 \cos(\tau_0t)+c_0'\sin(\tau_0 t)]+2\pi i \sum_{j=1}^\infty e^{-\sigma_j t} [c_j \cos(\tau_jt)+c_j'\sin(\tau_j t)],
\end{equation*}
 for $t>0$, and extended evenly to the real line.

Let us simplify this formula for the $j=0$ case. Using Fredholm theory, to have a decaying solution of equation \eqref{equation-radial}, $h$ must satisfy the compatibility condition
\begin{equation*}\begin{split}
0&=e^{-i\tau_0 t}\int_{\mathbb R} h(t')e^{i\tau_0 t'}\,dt'=\int_{\mathbb R } h(t-t') e^{-i\tau_0 t'}\,dt'\\
&=\int_{0}^{+\infty} h(t-t') e^{-i\tau_0 t'}\,dt'+\int_{-\infty}^0 h(t-t') e^{-i\tau_0 t'}\,dt'.
\end{split}
\end{equation*}
Substitute this expression into the formula below (and take imaginary part)
\begin{equation*}\begin{split}
\int_{0}^{+\infty} &h(t-t') e^{-i\tau_0 t'}\,dt'+\int_{-\infty}^0 h(t-t')e^{i\tau_0 t'}\,dt'\\
&=\int_{-\infty}^0 h(t-t') e^{-i\tau_0 t'}\,-\int_{-\infty}^0 h(t-t')e^{i\tau_0 t'}\,dt'
=-2i\int_{t}^{+\infty} h(t')\sin(\tau_0 (t-t'))\,dt'.
\end{split}\end{equation*}
Arguing as in the proof of Theorem \ref{thm:Hardy-potential} we obtain \emph{ii}. The only difference with the stable case is that the $j=0$ term in the summation in formula \eqref{equation12} needs to be replaced by
\begin{equation*}
\int_{t}^{+\infty}\sin(\tau_0 (t-t'))h(t')\,dt'
\end{equation*}
(note that $c_0=0$). A similar argument yields  \emph{iii.} too.

\subsection{A-priori estimates in weighted Sobolev spaces}\label{subsection:estimates-Sobolev}


For $s>0$, we define the norm in $\mathbb R\times\mathbb S^{N-1}$ given by
\begin{equation}\label{norm-0}
\|w\|^2_s=\sum_{m=0}^\infty\int_{\mathbb R} (1+\xi^2+m^2)^{2s} |\hat w_m(\xi)|^2\,d\xi.
\end{equation}
These are homogeneous norms in the variable $r=e^{-t}$, and formulate the Sobolev counterpart to the H\"older norms in $\mathbb R^N\setminus\{0\}$ from Section \ref{section:function-spaces}. That is, for $ w^*(r):=w(t)$ and $s$ integer we have
\begin{equation*}
\begin{split}
&\|w\|^2_0=\sum_{m=0}^\infty\int_0^\infty |\tilde w^*_m|^2 r^{-1}\,dr,\\
&\|w\|^2_1=\sum_{m=0}^\infty\int_0^\infty (|\tilde w^*_m|^2+|\partial_r\tilde w^*_m|^2 r^2)\,r^{-1}\,dr,\\
&\|w\|^2_2=\sum_{m=0}^\infty\int_0^\infty (|\tilde w^*_m|^2+|\partial_r\tilde w^*_m|^2 r^2+|\partial_{rr}\tilde w^*_{m}|^2r^4)\,r^{-1}\,dr.
\end{split}
\end{equation*}
One may also give the corresponding weighted norms, for a weight of the type $r^{-\vartheta}=e^{\vartheta t}$. Indeed, one just needs to modify the norm \eqref{norm-0} to
\begin{equation*}
\|w\|^2_{s,\vartheta}=\sum_{m=0}^\infty\int_{\mathbb R+i\vartheta} (1+\xi^2+m^2)^{2s} |\hat w_m(\xi)|^2\,d\xi.
\end{equation*}
For instance,  in the particular case $s=1$, this is
\begin{equation*}
\|w\|^2_{1,\vartheta}=\sum_{m=0}^\infty\int_0^\infty (| w^*_m|^{2}r^{-2\vartheta}+|\partial_r w^*_m|^2 r^{2-2\vartheta})\,r^{-1}\,dr.
\end{equation*}

\begin{proposition}
Let $s\geq 2\gamma$, and fix $\vartheta\in\mathbb R$ such that the horizontal line $\mathbb R+i\vartheta$ does not cross any pole $\tau_j^{(m)}\pm i\sigma_j^{(m)}$, $j=0,1,\ldots$, $m=0,1,\ldots$. If $w$ is a solution to
\begin{equation*}
\tilde{\mathcal L} w=h\quad\text{in } \mathbb R\times \mathbb S^{N-1}
\end{equation*}
 of the form \eqref{Green0}, then
\begin{equation*}\label{estimate-sobolev}
\|w\|_{s,\vartheta}\leq C\|h\|_{s-2\gamma,\vartheta}
\end{equation*}
for some constant $C>0$.
\end{proposition}

\begin{proof}
We project over spherical harmonics $w=\sum_m w_m E_m$, where $w_m$ is a solution to $\tilde{\mathcal L}_m w_m=h_m$.  Assume, without loss of generality, that $\vartheta=0$, otherwise replace the Fourier transform $\hat \cdot$ by $\widetilde \cdot$ on a different horizontal line. In particular, $\hat w_m(\xi)=(\Theta_m(\xi)-\kappa)^{-1}\hat h_m(\xi)$, and we simply estimate
\begin{equation*}
\begin{split}
\|w\|_s^2&=\sum_{m=0}^\infty\int_{\mathbb R} \frac{(1+|\xi|^2+m^2)^{2s}}{|\Theta_m(\xi)-\kappa|^2} |\hat h_m(\xi)|^2\,d\xi\leq C \sum_{m=0}^\infty\int_{\mathbb R}(1+|\xi|^2+m^2)^{2s-4\gamma}|\hat h_m(\xi)|^2\,d\xi\\
&=C\|h\|^2_{s-2\gamma},
\end{split}
\end{equation*}
where we have used that
\begin{equation*}
\frac{(1+|\xi|^2+m^2)^{2s}}{|\Theta_m(\xi)-\kappa|^2}\leq C (1+|\xi|^2+m^2)^{2s-4\gamma},
\end{equation*}
which follows from \eqref{symbol-limit}.
\end{proof}

\subsection{An application to a non-local ODE}\label{subsection:two-dimensional}

For simplicity, denote by   $u_1$ be the fast decaying solution from Proposition \ref{existence} for $\epsilon=1$. We would like to study the kernel for the linearization of \eqref{Lane-Emden} in the space of radial functions, this is,
\begin{equation}\label{eq:linearized-phi}
L_1\varphi:=(-\Delta)^\gamma \varphi -p(u_1)^{p-1}\varphi=0,\quad \varphi=\varphi(r).
\end{equation}
In terms of the $t$ variable we can write as follows: set $r=e^{-t}$, setting  $u_1=r^{-\frac{2\gamma}{p-1}}v_1$, $w=r^{\frac{N-2\gamma}{2}}\phi$, we can write this equation as
\begin{equation*}
\mathcal L_1 w:= r^{\frac{n+2\gamma}{2}}L_1 \varphi=P_\gamma w-V(t)w,\quad w=w(t),
\end{equation*}
for the potential
\begin{equation}\label{potential11}
V(t)=p(v_1)^{p-1}=
\begin{cases}
&\kappa+O(e^{-qt}),\quad \text{if } t\to+\infty, \\
&O(e^{q_1 t}),\quad \text{if } t\to -\infty,
\end{cases}
\end{equation}
where we have set $\kappa=pA_{N,p,\gamma}^{p-1}$, and for some $q,q_1>0$.

The following Hamiltonian identity will be useful  (the proof is exactly the same as in Theorem \ref{thctthamiltonian}):

\begin{lemma}\label{lemma:Wronskian}
Let $w_i=w_i(t)$, $i=1,2$, be two solutions to $\mathcal L_*w=0$ and set $W_i$, $i=1,2$, the corresponding extensions from Proposition \ref{prop:new-defining-function} (for simplicity, the special defining function $\rho^*$ will just be  denoted by $\rho$).
Then, the Hamiltonian quantity
\begin{equation*}\label{Hamiltonian2}
\begin{split}
\tilde H_\gamma(t):=\int_{0}^{\rho_0}  {\rho}^{1-2\gamma} e(\rho)[W_1\partial_t W_2-W_2\partial_t W_1]\,d\rho\\
\end{split}
\end{equation*}
 is constant along trajectories. Here $e(\rho)$ is a smooth function satisfying $e(\rho)\to 1$ as $\rho \to 0$.
\end{lemma}


First, let us look  at  non-singular solutions of \eqref{Lane-Emden}. Since this equation is invariant under rescaling, it is well known that
\begin{equation*}
\varphi_*:=r\partial_r u_1+\tfrac{p-1}{2\gamma}u_1
\end{equation*}
belongs to the kernel of $L_*$. If we set
\begin{equation*}
w_*=r^{\frac{N-2\gamma}{2}}\varphi_*,
\end{equation*}
then $w_*$ is a solution to $\mathcal L_1 w=0$. Moreover, in the stable case,
\begin{equation*}
w_*(t)\sim \begin{cases}
e^{-\sigma_0 t}\quad\text{as }t\to+\infty,\\
e^{-\frac{N-2\gamma}{2} t}\quad\text{as }t\to-\infty.
\end{cases}
\end{equation*}
In the unstable case, the picture is more complicated but has a similar behavior (see \cite{survey-ODE}).

From Lemma \ref{lemma:Wronskian} above we have a non-degeneracy result, which will be used in the proof of Proposition \ref{injectivity}:

\begin{corollary}\label{cor:other-solutions}
Any other solution to $\mathcal L_* w=0$ that decays both at $\pm \infty$ must be a multiple of $w_*$.
\end{corollary}

We conclude the section which a statement that is not needed in the proof of the main theorem, but we have decided to include it here because it showcases a classical ODE type behavior for a non-local equation, and it motivates the arguments in Section \ref{section:linear-theory}.

Assume that we are in the unstable case, i.e., the setting of Proposition \ref{prop:unstable}.

\begin{proposition}\label{two-dimensional} Let $q>0$ and fix a potential on $\mathbb R$ as in \eqref{potential11}, with the asymptotic behavior
\begin{equation*}
V(t)=\begin{cases}
\kappa+O(e^{-q t})\quad \text{as }t\to+\infty,\\
O(1)\quad \text{as }t\to-\infty,
\end{cases}
\end{equation*}
for $r=e^{-t}$. Then the  space of radial solutions to equation
\begin{equation}\label{eq-two-dimensional}
(-\Delta)^\gamma u-\frac{V}{r^{2\gamma}}u=0\quad\text{in }\mathbb R^N
\end{equation}
that have a bound of the form $|u(r)|\leq Cr^{-\frac{N-2\gamma}{2}}$ is at most two-dimensional.
\end{proposition}

\begin{proof} Let $u$ be one of such solutions, and
write $w=u r^{\frac{N-2\gamma}{2}}$, $w=w(t)$. By assumption, $w$ is bounded on $\mathbb R$. Moreover, $w$ satisfies the equation $P^{(0)}_\gamma w-Vw=0$.

First we look at the indicial roots at $t\to-\infty$, these come from studying the problem $P_\gamma^{(0)}w=0$, this is, for $\kappa=0$. We have
\[\begin{split}
P_\gamma^{(0)}w=\mathcal{V_*}(t)w=:h_0=
\begin{cases}
O(1)
	& \text{ as } t\to+\infty\\
O(e^{-q_1|t|})
	& \text{ as } t\to-\infty.
\end{cases}
\end{split}\]
Thus we get a particular solution decaying as $t\to -\infty$. Using this piece of information, we now look at the indicial roots as $t\to +\infty$, which come from studying the equation
\begin{equation*}
\tilde{\mathcal L}_0 w=h,\quad \text{for}\quad h:=(V-\kappa)w.
\end{equation*}
Then we have the bounds for $h$
\begin{equation*}
h(t)=\begin{cases}
O(e^{-q t})\quad \text{as }t\to+\infty,\\
O(e^{-q_1 |t|})\quad \text{as }t\to-\infty,
\end{cases}
\end{equation*}
so we take $\delta=q>0$, $\delta_0=0$, and apply Proposition \ref{prop:unstable}. Then $w$ must be of the form
\begin{equation*}\begin{split}
w(t)=w_0(t)&+C_0^1\sin(\tau_0 t)+C_0^2\cos (\tau_0 t)\\
&+\sum_{j=1}^\infty e^{-\sigma_j t} \left[ C_j^1 \sin(\tau_j t)+C_j^2 \cos(\tau_j t)\right]
+\sum_{j=1}^\infty e^{\sigma_j t} \left[ D_j^1 \sin(\tau_j t)+D_j^2 \cos(\tau_j t)\right]
\end{split}
\end{equation*}
for some real constants $C_0^1,C_0^2,C_j^1,C_j^2,D_j^1,D_j^2$, $j=1,2,\ldots$,
and $w_0$ is given by \eqref{v0}. The same proposition yields that $w_0$ is decaying as $O(e^{-\delta t})$ when $t\to+ \infty$, so we must have $D_j^1,D_j^2=0$ for $j=1,2,\ldots$.  Moreover, as $t\to -\infty$, if we want a bounded solution, this fixes the values of $C_j^1,C_j^2$ for $j=1,2,\ldots$, which implies that only $C_0^1$ and $C_0^2$ are free, and determined from $w$.

The Hamiltonian from Lemma \ref{lemma:Wronskian} shows that
there are no decaying solutions on both sides $t \to \pm\infty$ except for the trivial one $\phi_*$, which is controlled; indeed, its asymptotic behavior is given by the indicial roots (if it decays exponentially as $O(e^{-\delta t})$ for $t\to +\infty$, we can iterate statement \textit{iii.} with $\delta=lq$, $l=2,3,\ldots$ and $\delta_0=0$, to show that it decays faster than any $O(e^{-\delta t})$, $\delta>0$, as $t\to +\infty$). Next, we use a unique continuation result for equation \eqref{eq-two-dimensional} to show that it must vanish. In the stable case, unique continuation was proved in \cite{Fall-Felli} using a monotonicity formula, while in the unstable case it follows from \cite{Ruland}, where Carleman estimates were the crucial ingredient.

We remark that if, in addition, the potential satisfies a monotonicity condition, one can give a direct proof of unique continuation using Theorem 1 from \cite{Frank-Lenzmann-Silvestre}. Note that, however, in \cite{Frank-Lenzmann-Silvestre} the potential is assumed to be smooth at the origin. But one can check that the lack of regularity of the potential at the origin can be handled by the higher order of vanishing of $u$.

\end{proof}

\subsection{Technical results}\label{subsection:technical}

Here we give a more precise calculation of the poles of the function $\frac{1}{\Theta_m(z)-\kappa}$. For this, given $\kappa\in\mathbb R$, we aim to solve the equation \begin{equation}\label{a1}
\frac{\Gamma (\alpha +iz)\Gamma (\alpha -iz)}{\Gamma (\beta +iz)\Gamma
(\beta -iz)}-\kappa =0
\end{equation}%
with $\left\vert \alpha-\beta \right\vert <1$ and $\beta <\alpha $.

\begin{lemma}
Let
$$z=iR+\zeta$$
with $\left\vert z\right\vert >R_{0}$ and $R_{0}$ sufficiently large. Then
the solutions to \eqref{a1} are contained in balls of radius $\frac{C\kappa
\sin \left( (\alpha -\beta )\pi \right) }{\mathcal N^{2}{}^{(\alpha -\beta )}}$
around the points $z=(\mathcal N+\beta )i$, with $\mathcal N=\left[ R\right] $ and $C$
depending solely on $\alpha$ and $\beta $.
\end{lemma}

\begin{proof}
First we note, by using the identity $\Gamma (s)\Gamma (1-s)=\pi
/\sin (\pi s)$, that
\begin{equation*}
\begin{split}
\frac{\Gamma (\alpha -R+i\zeta )}{\Gamma (\beta -R+i\zeta )}
&=\frac{\Gamma
(1-\beta +R-i\zeta )}{\Gamma (1-\alpha +R-i\zeta )}\frac{\sin (\pi (\beta
-R+i\zeta ))}{\sin (\pi (\alpha -R+i\zeta ))}\\
&=\frac{\Gamma (1-\beta +R-i\zeta )}{%
\Gamma (1-\alpha +R-i\zeta )}\frac{\sin (\pi (\beta -\delta +i\zeta ))}{\sin
(\pi (\alpha -\delta +i\zeta ))},
\end{split}
\end{equation*}
where we have denoted
\begin{equation*}
\delta =R-\left[ R\right].
\end{equation*}%
Then, Stirling's formula \eqref{Stirling} yields
\begin{equation*}\left\vert \Gamma (1+z)\right\vert \sim \left\vert z\right\vert ^{\Real%
z}e^{-(\Imag z)\arg (z)}e^{-\Real z}\sqrt{2\pi }\left\vert
z\right\vert ^{\frac{1}{2}},
\end{equation*}%
which implies
\begin{equation*}
\begin{split}
&\left\vert \frac{\Gamma (1-\beta +R-i\zeta )\Gamma (\alpha +R-i\zeta )}{\Gamma
(1-\alpha +R-i\zeta )\Gamma (\beta +R-i\zeta )}\right\vert \\
&\qquad\sim (R^{2}+\zeta
^{2})^{\alpha -\beta }e^{\zeta \left( \arctan \frac{\zeta }{1-\beta +R}+\arctan
\frac{\zeta }{\alpha +R}-\arctan \frac{\zeta }{1-\alpha +R}-\arctan \frac{\zeta }{%
\beta +R}\right) }e^{-2(\alpha -\beta )}.
\end{split}
\end{equation*}%
Since%
\begin{equation*}\begin{split}
\arctan &\frac{\zeta }{1-\beta +R}+\arctan \frac{\zeta }{\alpha +R}-\arctan
\frac{\zeta }{1-\alpha +R}-\arctan \frac{\zeta }{\beta +R} \\
&=\arctan \frac{\frac{\zeta }{1-\beta +R}-\frac{\zeta }{1-\alpha +R}}{1+\frac{%
\zeta }{1-\beta +R}\frac{\zeta }{1-\alpha +R}}+\arctan \frac{\frac{\zeta }{\alpha
+R}-\frac{\zeta }{\beta +R}}{1+\frac{\zeta }{\alpha +R}\frac{\zeta }{\beta +R}}\\
&\sim -2\arctan \frac{\left( \alpha -\beta \right) \zeta }{R^{2}+\zeta ^{2}}
\sim
-2\frac{\left( \alpha -\beta \right) \zeta }{R^{2}+\zeta ^{2}},
\end{split}\end{equation*}%
we can estimate, for $R^{2}+\zeta ^{2}$ sufficiently large,
\begin{equation*}
\left\vert \frac{\Gamma (1-\beta +R-i\zeta )\Gamma (\alpha +R-i\zeta )}{\Gamma
(1-\alpha +R-i\zeta )\Gamma (\beta +R-i\zeta )}\right\vert \sim (R^{2}+\zeta
^{2})^{\alpha -\beta }e^{-2\frac{\left( \alpha -\beta \right) \zeta ^{2}}{%
R^{2}+\zeta ^{2}}}e^{-2(\alpha -\beta)}.
\end{equation*}%
Therefore, for $R^{2}+\zeta ^{2}>R_{0}^{2}$ with $R_{0}$ sufficiently large, we have the bound
\begin{equation*}
C^{-1}(R^{2}+\zeta ^{2})^{\alpha -\beta }\leq \left\vert \frac{\Gamma (1-\beta
+R-i\zeta )\Gamma (\alpha +R-i\zeta )}{\Gamma (1-\alpha +R-i\zeta )\Gamma (\beta
+R-i\zeta )}\right\vert \leq C(R^{2}+\zeta ^{2})^{\alpha -\beta },
\end{equation*}%
where $C$ depends only on $\alpha $ and $\beta $. Hence,
\begin{equation*}
\frac{\kappa}{C(R^{2}+\zeta ^{2})^{\alpha -\beta }}\leq \left\vert \frac{%
\sin (\pi (\beta -\delta +i\zeta ))}{\sin (\pi (\alpha -\delta +i\zeta ))}%
\right\vert \leq \frac{\kappa }{C^{-1}(R^{2}+\zeta ^{2})^{\alpha -\beta }},
\end{equation*}%
and by writing%
\begin{equation*}
\delta -i\zeta =\beta +\tilde{z},
\end{equation*}%
we conclude that necessarily%
\begin{equation*}
\left\vert \tilde{z}\right\vert \leq \frac{C\kappa \sin \left( (\alpha
-\beta )\pi \right) }{R^{2}{}^{(\alpha -\beta )}},
\end{equation*}%
which implies that solutions to \eqref{a1} lie at
\begin{equation*}
z=iR+\zeta =i\left[ R\right] +i\beta +i\tilde{z}=\left[ R\right] +\beta
+O\left( \frac{C\kappa \sin \left( (\alpha -\beta )\pi \right) }{\left[ R%
\right] ^{2}{}^{(\alpha -\beta )}}\right),
\end{equation*}%
and this proves the Lemma.\\
\end{proof}

Next we write%
\[
z=i(\beta +\mathcal N)+\tilde{z}
\]%
with $\mathcal N$ sufficiently large (according to the previous lemma) natural
number, and equation \eqref{a1} reads%
\begin{equation}\label{a2}
\frac{\Gamma (\alpha -\beta -\mathcal N+i\tilde{z})\Gamma (\alpha +\beta +\mathcal N-i%
\tilde{z})}{\Gamma (-\mathcal N+i\tilde{z})\Gamma (2\beta +\mathcal N-i\tilde{z})}%
-\kappa =0.
\end{equation}%
Since%
\begin{equation*}
\Gamma (-\mathcal N+i\tilde{z})=(-1)^{\mathcal N-1}\frac{\Gamma (-i\tilde{z})\Gamma
(1+i\tilde{z})}{\Gamma (\mathcal N+1-i\tilde{z})}
\end{equation*}%
and%
\begin{equation*}
\Gamma (\alpha -\beta -\mathcal N+i\tilde{z})=(-1)^{\mathcal N-1}\frac{\Gamma (\beta
-\alpha -i\tilde{z})\Gamma (1+\alpha -\beta +i\tilde{z})}{\Gamma
(\mathcal N+1-\alpha +\beta -i\tilde{z})},
\end{equation*}%
we can write%
\begin{equation*}\begin{split}
&\frac{\Gamma (\alpha -\beta -\mathcal N+i\tilde{z})\Gamma (\alpha +\beta +\mathcal N-i%
\tilde{z})}{\Gamma (-\mathcal N+i\tilde{z})\Gamma (2\beta +\mathcal N-i\tilde{z})}\\
&\qquad=\frac{\Gamma (\mathcal N+1-i\tilde{z})\Gamma (\alpha +\beta +\mathcal N-i\tilde{z})}{%
\Gamma (\mathcal N+1-\alpha +\beta -i\tilde{z})\Gamma (2\beta +\mathcal N-i\tilde{z})}%
\frac{\Gamma (\beta -\alpha -i\tilde{z})\Gamma (1+\alpha -\beta +i%
\tilde{z})}{\Gamma (1+i\tilde{z})\Gamma (-i\tilde{z})}.
\end{split}\end{equation*}
Using that
\[
\frac{\Gamma (\beta -\alpha -i\tilde{z})\Gamma (1+\alpha -\beta +i%
\tilde{z})}{\Gamma (1+i\tilde{z})\Gamma (-i\tilde{z})}=\frac{%
\sin (-\pi i\tilde{z})}{\sin \left( \pi (\beta -\alpha -i\tilde{z}%
)\right) },
\]%
as well as Stirling's formula \eqref{Stirling} to estimate%
\begin{equation*}\begin{split}
&\frac{\Gamma (\mathcal N+1-i\tilde{z})\Gamma (\alpha +\beta +\mathcal N-i\tilde{z})}{%
\Gamma (\mathcal N+1-\alpha +\beta -i\tilde{z})\Gamma (2\beta +\mathcal N-i\tilde{z})}\\
&\qquad\sim \frac{(\mathcal N-i\tilde{z})^{\mathcal N-i\tilde{z}}(\alpha +\beta +\mathcal N-1-i%
\tilde{z})^{\alpha +\beta +\mathcal N-1-i\tilde{z}}}{(\mathcal N-\alpha +\beta -i%
\tilde{z})^{\mathcal N-\alpha +\beta -i\tilde{z}}(2\beta -1+\mathcal N-i\tilde{z}%
)^{2\beta -1+\mathcal N-i\tilde{z}}}\\
&\qquad\qquad\cdot e^{-2(\alpha -\beta )}\sqrt{\frac{(\mathcal N-i%
\tilde{z})(\alpha +\beta +\mathcal N-1-i\tilde{z})}{(\mathcal N-\alpha +\beta -i%
\tilde{z})(2\beta -1+\mathcal N-i\tilde{z})}}\\
&\qquad\sim \mathcal N^{2(\alpha -\beta )}e^{-2i(\alpha -\beta )\tilde{z}}e^{-2(\alpha
-\beta )},
\end{split}
\end{equation*}
we arrive at the relation%
\begin{equation*}
\frac{\sin (-\pi i\tilde{z})}{\sin \left( \pi (\beta -\alpha -i%
\tilde{z})\right) }e^{-2i(\alpha -\beta )\tilde{z}}\sim \frac{%
\kappa}{\mathcal N^{2(\alpha -\beta )}e^{-2(\alpha -\beta )}},
\end{equation*}%
which implies%
\begin{equation*}
\tilde{z}\sim \frac{i}{\pi }\frac{\kappa \sin \left( \pi (\beta -\alpha
)\right) }{\mathcal N^{2(\alpha -\beta )}e^{-2(\alpha -\beta )}}.
\end{equation*}%
In fact, it is easy to see from \eqref{a2} and the estimates above that a
purely imaginary solution $\tilde{z}$ does exist and a standard fixed
point argument in each of the balls in the previous lemma would show that it
is unique.

Finally, the half-ball of radius $R_{0}$ around the origin in the upper
half-plane is a compact set. Since the function at the left hand side of
\eqref{a1} is meromorphic, there cannot exist accumulation points of zeros
and this necessarily implies that the number of zeros in that half-ball is
finite.

We conclude then that the set of solutions to \eqref{a1} consists of a
finite number of solutions in a half ball of radius $R_{0}$ around the
origin in the upper half-plane together with an infinite sequence of roots
at the imaginary axis located at
\begin{equation}\label{zN}
z_{\mathcal N}=i(\beta +\mathcal N)+O\left( \frac{\kappa
\sin \left( \pi (\beta -\alpha )\right) }{\mathcal N^{2(\alpha -\beta )}e^{-2(\alpha
-\beta )}}\right) \quad\text{for}\quad \mathcal N>R_{0},
\end{equation}
as desired.

\bigskip

Now we consider the asymptotics for the residues. We define%
\begin{equation*}
g(z):=\frac{\Gamma (\alpha +iz)\Gamma (\alpha -iz)}{\Gamma (\beta
+iz)\Gamma (\beta -iz)}-\kappa.
\end{equation*}
We will estimate the residue of the function $\frac{1}{g(z)}$ at
the poles $z_{\mathcal N}$ when $\mathcal N$ is sufficiently large.
Given the fact that the poles are simple and the function $1/g(z)$
is analytic outside its poles, we have
\begin{equation*}
\Res\Big(\frac{1}{g(z)},z_{\mathcal N}\Big)=\lim_{z\rightarrow z_{\mathcal N}}\left( (z-z_{\mathcal N})\frac{%
1}{g(z)}\right) =\frac{1}{g^{\prime }(z_{\mathcal N})}\ .
\end{equation*}
Hence%
\begin{equation*}
\begin{split}
g^{\prime }(z)&=\frac{d}{dz}\left( \frac{\Gamma (\alpha +iz)\Gamma (\alpha
-iz)}{\Gamma (\beta +iz)\Gamma (\beta -iz)}\right) \\
&=-i\frac{\Gamma ^{\prime
}(\beta +iz)}{\Gamma ^{2}(\beta +iz)}\left( \frac{\Gamma (\alpha +iz)\Gamma
(\alpha -iz)}{\Gamma (\beta -iz)}\right) +\frac{1}{\Gamma (\beta +iz)}\left(
\frac{\Gamma (\alpha +iz)\Gamma (\alpha -iz)}{\Gamma (\beta -iz)}\right)
^{\prime }
\\&=: S_{1}+S_{2}.
\end{split}
\end{equation*}
Notice that $g(z_{\mathcal N})=0$ implies that
\begin{equation*}
 \Gamma (\beta +iz_{\mathcal N})=\frac{\Gamma (\alpha
+iz_{\mathcal N})\Gamma (\alpha -iz_{\mathcal N})}{\kappa \Gamma (\beta -iz_{\mathcal N})}
\end{equation*}%
and therefore,%
\begin{equation*}
S_{1}=-i\frac{\Gamma ^{\prime }(\beta +iz_{\mathcal N})}{\Gamma ^{2}(\beta +iz_{\mathcal N})}%
\left( \frac{\Gamma (\alpha +iz_{\mathcal N})\Gamma (\alpha -iz_{\mathcal N})}{\Gamma (\beta
-iz_{\mathcal N})}\right) =-i\kappa \frac{\Gamma ^{\prime }(\beta +iz_{\mathcal N})}{\Gamma
(\beta +iz_{\mathcal N})}=-i\kappa \psi (\beta +iz_{\mathcal N}),
\end{equation*}
where $\psi (z)$ is the digamma function. We recall the expansion \eqref{expansion-digamma},
\begin{equation*}
\psi(z)=-{\gamma}+\sum_{l=0}^\infty \left(\tfrac{1}{l+1}-\tfrac{1}{l+z}\right).
\end{equation*}
In this section, $\gamma$ denotes the Euler constant. Then%
\begin{equation}
S_{1}=i\kappa \left( \gamma +\sum_{l=0}^{\infty }\left( \frac{1}{l+\beta
+iz_{\mathcal N}}-\frac{1}{l+1}\right) \right) =\frac{\pi e^{-2(\alpha -\beta )}}{%
i\sin (\pi (\alpha -\beta ))}\mathcal N^{2(\alpha -\beta )}+O(1),  \label{s1}
\end{equation}%
where we have used the asymptotics of $l+\beta +iz_{\mathcal N}$ when $l=\mathcal N$ from \eqref{zN}. Next, using again \eqref{zN} we
estimate
\begin{equation*}\begin{split}
S_{2}&=\left. \frac{1}{\Gamma (\beta +iz)}\left( \frac{\Gamma (\alpha
+iz)\Gamma (\alpha -iz)}{\Gamma (\beta -iz)}\right) ^{\prime }\right\vert
_{z=z_{\mathcal N}}=\kappa i\left( \frac{\Gamma ^{\prime }(\alpha +iz)}{\Gamma
(\alpha +iz)}-\frac{\Gamma ^{\prime }(\alpha -iz)}{\Gamma (\alpha -iz)}+%
\frac{\Gamma ^{\prime }(\beta -iz)}{\Gamma (\beta -iz)}\right)\\
&=\kappa i\left( \psi (\alpha -\beta -\mathcal N+O(\mathcal N^{-2(\alpha -\beta )}))-\psi
(\alpha +\beta +\mathcal N+O(\mathcal N^{-2(\alpha -\beta )}))\right.\\
&\,\qquad\left.+\psi (2\beta +\mathcal N+O(\mathcal N^{-2(\alpha
-\beta )}))\right).
\end{split}
\end{equation*}%
By using the relations%
\begin{eqnarray*}
&&\psi (1-z)-\psi (z)=\pi \cot (\pi z)\\
&&\psi (z)\sim \ln (z-\gamma )+2\gamma \text{, as }\left\vert z\right\vert
\rightarrow \infty \text{, }\Real z>0,
\end{eqnarray*}
we conclude, as $\mathcal N\to \infty$,%
\begin{equation*}
\begin{split}
\psi (\alpha -\beta &-\mathcal N+O(\mathcal N^{-2(\alpha -\beta )}))\\
 &=\psi (1-\alpha +\beta
+\mathcal N+O(\mathcal N^{-2(\alpha -\beta )}))+\pi \cot (\pi (1-\alpha +\beta
+\mathcal N+O(\mathcal N^{-2(\alpha -\beta )}))) \\
&= \ln (\mathcal N)+O(1),
\end{split}
\end{equation*}%
and hence%
\begin{equation}
S_{2}=i\kappa \ln \mathcal N+O(1).  \label{s2}
\end{equation}%
Putting together \eqref{s1} and \eqref{s2} we find
\begin{equation*}
S_1+S_2=\frac{\pi e^{-2(\alpha -\beta )}}{i\sin (\pi (\alpha -\beta ))}\mathcal N^{2(\alpha
-\beta )}+i\kappa \ln \mathcal N+O(1),
\end{equation*}
and hence%
\begin{equation}\label{asymptotics-residue}
\begin{split}
\Res\Big(\frac{1}{g(z)},z_{\mathcal N}\Big)&= \frac{i}{\frac{\pi e^{-2(\alpha -\beta )}}{%
\sin (\pi (\alpha -\beta ))}\mathcal N^{2(\alpha -\beta )}-\kappa \ln \mathcal N+O(1)}\\
&=i\frac{%
\sin (\pi (\alpha -\beta ))e^{2(\alpha -\beta )}}{\pi }\mathcal N^{-2(\alpha -\beta
)}+O\left( \frac{\ln \mathcal N}{\mathcal N^{4(\alpha -\beta )}}\right)
\end{split}
\end{equation}
as $\mathcal N\rightarrow \infty $.

\section{Linear theory - injectivity}\label{section:linear-theory}

Let  $\bar u_\ve$ be the approximate solution from the  Section \ref{section:function-spaces}.
In this  section we consider the  linearized operator
\begin{equation}\label{linearized-general}
L_\ve \phi:=(-\Delta_{\mathbb R^n})^\gamma \phi-p\bar{u}_\ve^{p-1}\phi, \quad\text{in }\mathbb R^n\setminus \Sigma,
\end{equation}
where $\Sigma$ is a sub-manifold of dimension $k$ (or a disjoint union of smooth $k$-dimensional manifolds), and
\begin{equation}\label{linearized-points}
L_\ve \phi:=(-\Delta_{\mathbb R^N})^\gamma \phi-pA_{N,p,\gamma}\bar{u}_\ve^{p-1}\phi, \quad\text{in }\mathbb R^N\setminus \{q_1,\ldots,q_K\}.
\end{equation}

For this, we first need to study the model linearization
\begin{equation}\label{model-linearization}
\mathcal L_1\phi:=(-\Delta_{\mathbb R^N})^\gamma \phi -p A_{N,p,\gamma}u_1^{p-1} \phi=0\quad\text{in}\quad\mathbb R^N\setminus\{0\}.
\end{equation}
We will show that any solution (in suitable weighted spaces) to this equation
must vanish everywhere (from which injectivity in $\mathbb R^n\setminus\mathbb R^k$ follows easily), and then we will prove injectivity for the operator $L_\ve$.

Let us rewrite \eqref{model-linearization} using conformal properties and the conjugation \eqref{conjugation1}. If we  define
\begin{equation}\label{w-phi}
w=r^{\frac{N-2\gamma}{2}}\phi,
\end{equation}
then this equation is equivalent to
\begin{equation}\label{eq0}
P_\gamma^{g_0}(w)-Vw=0,
\end{equation}
for the radial potential
\begin{equation}\label{potential}
V=V(r)=r^{2\gamma}pA_{N,p,\gamma}u_1^{p-1}.
\end{equation}
 The asymptotic behavior of this potential is easily calculated using Proposition \ref{existence}
 and, indeed, for $r=e^{-t}$,
\begin{equation}\label{asymptotics-potential}
V(t)=\begin{cases}
 pA_{N,p,\gamma} +O(e^{-q_1t} )&\mbox{ as }t\to +\infty,\\
 O(e^{q_0 t})&\mbox{ as }t\to -\infty,
\end{cases}\end{equation}
for $q_0=(N-2\gamma)(p-1)-2\gamma>0$ and $q_1>0$.

\bigskip

Let $\gamma\in (0,1)$. By the well known extension theorem for the fractional Laplacian  \eqref{CS1}-\eqref{CS2},  equation \eqref{model-linearization} is equivalent to the boundary reaction problem
\begin{equation*}\label{eqlinear1}
\left\{\begin{array}{r@{}l@{}l}
\partial_{\ell\ell}\Phi+\dfrac{1-2\gamma}{\ell}\partial_\ell \Phi+\Delta_{\R^N}\Phi
    &\,=0
    &\quad\mbox{in }\R^{N+1}_+,\medskip\\
-\tilde d_\gamma\lim\limits_{\ell\to0}\ell^{1-2\gamma}\partial_\ell \Phi
    &\,=pA_{N,p,\gamma}u_1^{p-1}\Phi
    &\quad\mbox{on }\R^N \setminus \{0\},
\end{array}\right.
\end{equation*}
where $\tilde d_\gamma$ is defined in \eqref{tilde-d} and $\Phi|_{\ell=0}=\phi$.

Keeping the notations of Section \ref{section:isolated-singularity} for the spherical harmonic decomposition of $\mathbb S^{N-1}$, by $\mu_m$ we denote the $m$-th eigenvalue for $-\Delta_{\mathbb S^{N-1}}$, repeated according to multiplicity, and by $E_m(\theta)$ the corresponding eigenfunction.  Then we can write $\Phi=\sum_{m=0}^\infty \Phi_m(r,\ell)E_m(\theta)$, where $\Phi_m$ satisfies the following:
\begin{equation}\label{eqlinear2}
\left\{\begin{array}{r@{}l@{}l}
\partial_{\ell\ell}\Phi_{m}+\dfrac{1-2\gamma}{\ell}\partial_\ell\Phi_{m}
+\Delta_{\R^N}\Phi_m-\dfrac{\mu_m}{r^2}\Phi_m
    &\,=0
    &\quad\mbox{in }\R^{N+1}_+,\medskip\\
-\tilde d_\gamma\lim\limits_{\ell\to0}\ell^{1-2\gamma}\partial_\ell \Phi_m
    &\,=pA_{N,p,\gamma}u_1^{p-1}\Phi_m
    &\quad\mbox{on }\R^N \setminus \{0\},
\end{array}
\right.
\end{equation}
or equivalently, from \eqref{eq0},
\begin{equation}\label{eq0m}
P_\gamma^m(w)-Vw=0,
\end{equation}
for $w=w_m=r^{\frac{N-2\gamma}{2}}\phi_m$, $\phi_m=\Phi_m(\cdot,0)$.

\subsection{Indicial roots}\label{subsection:indicial-roots}

Let us calculate the indicial roots for the model linearized operator defined in \eqref{model-linearization} as $r\to 0$ and as $r\to \infty$.
Recalling \eqref{asymptotics-potential}, $\mathcal L_1$ behaves like the Hardy operator \eqref{Hardy-operator} with $\kappa=pA_{N,p,\gamma}$ as $r\to 0$ and $\kappa=0$ as $r\to \infty$. Moreover, we can characterize very precisely the location of the poles in Theorem \ref{thm:Hardy-potential} and Proposition \eqref{prop:unstable}.

Here we find  a crucial difference from the local case $\gamma=1$, where the Fourier symbol for the $m$-th projection $\Theta_m(\xi)-\kappa$ is quadratic in $\xi$, implying that there are only two poles. In contrast, in the non-local case, we have just seen that there exist \emph{infinitely many} poles.  Surprisingly, even though $\mathcal L_1$ is a non-local operator, its behavior is controlled by just four indicial roots, so we obtain results analogous to the local case.

For the statement of the next result, recall the shift \eqref{w-phi}.

\begin{lemma}\label{indicial}
For the operator $\mathcal L_1$ we have that, for each fixed mode $m=0,1,\ldots,$
\begin{itemize}
\item[\emph{i.}] At $r=\infty$, there exist two sequences of indicial roots
\begin{equation*}
\{\tilde \sigma_j^{(m)}\pm i\tilde\tau_j^{(m)}-\tfrac{N-2\gamma}{2}\}_{j=0}^\infty\quad \text{and} \quad\{-\tilde\sigma_j^{(m)}\pm i\tilde\tau_j^{(m)}-\tfrac{N-2\gamma}{2}\}_{j=0}^\infty.
\end{equation*}
Moreover,
\begin{equation*}
\tilde\gamma_m^{\pm}:=\pm\tilde\sigma_0^{(m)}-\tfrac{N-2\gamma}{2}=-\tfrac{N-2\gamma}{2}\pm \left[1-\gamma+\sqrt{(\tfrac{N-2}{2})^2+\mu_m}\right],\quad m=0,1,\ldots,
\end{equation*}
and
$\tilde{\gamma}_m^+$
is an increasing sequence (except for multiplicity repetitions).

\item[\emph{ii.}] At $r=0$,  there exist two sequences of indicial roots
\begin{equation*}
\{\sigma_j^{(m)}\pm i\tau_j^{(m)}-\tfrac{N-2\gamma}{2}\}_{j=0}^\infty\quad \text{and} \quad\{-\sigma_j^{(m)}\pm i\tau_j^{(m)}-\tfrac{N-2\gamma}{2}\}_{j=0}^\infty.
\end{equation*}
Moreover,
\begin{itemize}
\item[\emph{a)}] For the mode $m=0$, there exists $p_1$ with $\frac{N}{N-2\gamma}<p_1<\frac{N+2\gamma}{N-2\gamma}$ (and it is given by \eqref{p1}), such that for $\frac{N}{N-2\gamma}<p<p_1$ (the \underline{stable} case), the indicial roots $\gamma_0^{\pm}:=\pm\sigma_0^{(0)}-\frac{N-2\gamma}{2}$ are real with
\begin{equation*}\label{gamma0}
-\tfrac{2\gamma}{p-1}<\gamma_0^-<-\tfrac{N-2\gamma}{2}<\gamma_0^+,
\end{equation*}
while if $p_1<p<\frac{N+2\gamma}{N-2\gamma}$ (the \underline{unstable} case), then $\gamma_0^\pm$ are a pair of complex conjugates with real part $-\frac{N-2\gamma}{2}$ and imaginary part $\pm\tau_0^{(0)}$.

\item[\emph{b)}] In addition, for all $j\geq 1$,
\begin{equation*}
\sigma_j^{(0)}>\tfrac{N-2\gamma}{2}.
\end{equation*}

\item[\emph{c)}] For the mode $m=1$,
\begin{equation*}\label{indicial1}
\gamma_1^-:=-\sigma_0^{(1)}-\tfrac{N-2\gamma}{2}=-\tfrac{2\gamma}{p-1}-1.
\end{equation*}
\end{itemize}
\end{itemize}
\end{lemma}

\begin{proof}
First we consider statement \emph{ii.} and  calculate the indicial roots at $r=0$.   Recalling the shift \eqref{w-phi}, let $\mathcal L_1$ act on the function $r^{-\frac{N-2\gamma}{2}+\delta}$, and consider instead the operator in \eqref{eq0}. Because of Proposition  \ref{prop:symbol}, for each $m=0,1,\ldots$, the indicial root $\gamma_m:=-\frac{N-2\gamma}{2}+\delta$ satisfies
\begin{equation}\label{indicialm}
2^{2\gamma}\frac{\Gamma \big(A_m+\frac{\delta}{2}\big)\Gamma \big(A_m-\frac{\delta}{2}\big)}
{\Gamma  \big(B_m+\frac{\delta}{2}\big)\Gamma \big(B_m
-\frac{\delta}{2}\big)}=pA_{N,p,\gamma},
\end{equation}
where $A_m,B_m$ are defined in \eqref{Am}.

Note that if $\delta\in\mathbb C$ is a solution, then $-\delta$ and $\pm \overline{\delta}$ are also solutions. Let us write $\frac{\delta}{2}=\alpha+i\beta$, and denote
\begin{equation*}
\Phi_m(\alpha,\beta)=2^{2\gamma}\frac{\Gamma \big(A_m+\frac{\delta}{2}\big)\Gamma \big(A_m-\frac{\delta}{2}\big)}
{\Gamma \big(B_m+\frac{\delta}{2}\big)\Gamma \big(B_m
-\frac{\delta}{2}\big)}.
\end{equation*}
From the expression, one can see that $\Phi_m(\alpha,0)$ and $\Phi_m(0,\beta)$ are real functions.

We first claim that on the $\alpha\beta$-plane, provided that $|\alpha|\leq{B_m}$, any solution of \eqref{indicialm} must satisfy $\alpha=0$ or $\beta=0$, i.e., $\delta$ must be real or purely imaginary. Observing that the right hand side of \eqref{indicialm} is real and so is $\Phi_m(0,\beta)$ for $\beta\neq0$, the claim follows from the strict monotonicity of the imaginary part with respect to $\alpha$, namely
\begin{equation}\label{claim-1}
\begin{split}
\frac{\partial}{\partial \alpha}\Imag(\Phi_m(\alpha, \beta))&=-\frac{i}{2}\frac{\partial }{\partial \alpha}\Big[\Phi_m(\alpha,\beta)-\Phi_m(\alpha,-\beta) \Big]\\
&=\sum_{j=0}^\infty \Imag\Big[ \tfrac{1}{j+A_m+\alpha-i\beta} +\tfrac{1}{j+A_m-\alpha-i\beta}+\tfrac{1}{j+B_m+\alpha+i\beta}+\tfrac{1}{j+B_m-\alpha+i\beta}\Big]\\
&=\beta\sum_{j=0}^\infty \Big[ \tfrac{1}{(j+A_m+\alpha)^2+\beta^2}+\tfrac{1}{(j+A_m-\alpha)^2+\beta^2} -\tfrac{1}{(j+B_m+\alpha)^2+\beta^2}-\tfrac{1}{(j+B_m-\alpha)^2+\beta^2}\Big],
\end{split}
\end{equation}
the summands being strictly negative since $A_m>B_m$.
If $\beta\neq 0$ and $|\alpha|\leq B_m$, it is easy to see that the above expression is not zero. This yields the  proof of the claim.

Moreover, $\Phi_m(\alpha, 0)$ and $\Phi_m(0,\beta)$ are even functions in $\alpha, \beta$, respectively. Using the properties of the digamma function again, one can check that
\begin{equation}\label{derivative-a}
\frac{\partial{\Phi_m(\alpha,0)}}{\partial \alpha}<0 \mbox{ for }\alpha>0
\end{equation}
and
\begin{equation}\label{derivative-b}
\frac{\partial{\Phi_m(0,\beta)}}{\partial \beta}>0 \mbox{ for }\beta>0.
\end{equation}

Let us consider now the case $m=0$. Using the explicit expression for $A_{N,p,\gamma}$ from \eqref{Apn}, then $\delta$ must be a solution of
\begin{equation}\label{eq:indicial}
\frac{\Gamma \big(\frac{N}{4}+\frac{\gamma}{2}+\frac{\delta}{2}\big)
\Gamma \big(\frac{N}{4}+\frac{\gamma}{2}-\frac{\delta}{2}\big)}
{\Gamma\big(\frac{N}{4}-\frac{\gamma}{2}+\frac{\delta}{2}\big)
\Gamma\big(\frac{N}{4}-\frac{\gamma}{2}-\frac{\delta}{2}\big)}
=p\frac{\Gamma\big(\frac{N}{2}-\frac{\gamma}{p-1}\big)\Gamma\big(\frac{\gamma}{p-1}+\gamma\big)}
{\Gamma\big(\frac{\gamma}{p-1}\big)\Gamma\big(\frac{N}{2}-\gamma-\frac{\gamma}{p-1}\big)}=:\lambda(p).
\end{equation}
From the arguments in \cite{Ao-Chan-Gonzalez-Wei} (see also the definition of $p_1$ in \eqref{p1}), there exists a unique $p_1$ satisfying $\frac{N}{N-2\gamma}<p_1<\frac{N+2\gamma}{N-2\gamma}$ such that $\Phi_0(0,0)=\lambda(p_1)$, and $\Phi_0(0,0)>\lambda(p)$ when $\frac{N}{N-2\gamma}<p<p_1$, and $\Phi_0(0,0)<\lambda(p)$ when $p_1<p<\frac{N+2\gamma}{N-2\gamma}$.

\bigskip

Assume first that $\frac{N}{N-2\gamma}<p<p_1$ (the stable case). Then from (\ref{derivative-b}), we know that there are no indicial roots on the imaginary axis. Next we consider the real axis. Since $\Phi_0(B_0, 0)=0$, by (\ref{derivative-a}), there exists an unique root $\alpha^*\in (0, B_0)$ such that $\Phi_0(\pm \alpha^*, 0)=\lambda(p)$. We now show that $\alpha^*\in (0, \frac{2\gamma}{p-1}-\frac{N-2\gamma}{2})$. Note that
\begin{equation*}
\Phi_0\big(\tfrac{2\gamma}{p-1}-\tfrac{N-2\gamma}{2},0\big)-pA_{N,p,\gamma}
=(1-p)\frac{\Gamma\big(\frac{N}{2}-\frac{\gamma}{p-1}\big)\Gamma\big(\frac{\gamma}{p-1}+\gamma\big)}
{\Gamma\big(\frac{\gamma}{p-1}\big)\Gamma\big(\frac{N}{2}-\gamma-\frac{\gamma}{p-1}\big)}<0.
\end{equation*}
We conclude using the monotonicity of $\Phi_0(\alpha,0)$ in $\alpha$.

Now we consider the unstable case, i.e., for $p>p_1$. First by \eqref{derivative-a}, there are no indicial roots on the real axis. Then by \eqref{claim-1}, in the region $|\alpha|\leq B_0$, if a solution exists, then $\delta$ must stay in the imaginary axis. Since $\Phi_0(0,\beta)$ is increasing in $\beta$ and $\lim_{\beta\to \infty}\Phi_0(0,\beta)=+\infty$, we get an unique $\beta^*>0$ such that $\Phi_0(0,\pm \beta^*)=\lambda(p)$.

In the notation of Section \ref{section:Hardy}, we denote all the solutions to \eqref{eq:indicial} to be $\sigma_j^{(0)}\pm i\tau_j^{(0)}$ and $-\sigma_j^{(0)}\pm i\tau_j^{(0)}$, such that $\sigma_j$ is increasing sequence, then from the above argument, one has the following properties:
\begin{equation*}\left\{\begin{array}{r@{}lr@{}ll}
\sigma_0^{(0)}
    &\,\in \big(0,\tfrac{2\gamma}{p-1}-\tfrac{N-2\gamma}{2}\big), &\tau_0^{(0)}
    &\,=0, \quad
    &\mbox{for }\tfrac{N}{N-2\gamma}<p<p_1,\medskip\\
\sigma_0^{(0)}
    &\,=0, \quad
    &\tau_0^{(0)}
    &\,\in (0,\infty), \quad
    &\mbox{for }p_1\leq p<\tfrac{N+2\gamma}{N-2\gamma},
\end{array}\right.\end{equation*}
and
\begin{equation*}
\sigma_j^{(0)}>2B_0=\tfrac{N-2\gamma}{2} \quad\mbox{ for }j\geq 1.\\
\end{equation*}

For the next mode $m=1$, one can check by direct calculation that $\alpha_1=\frac{2\gamma}{p-1}+1-\frac{N-2\gamma}{2}$ is a solution to \eqref{indicialm}. By the monotonicity \eqref{derivative-a}, there are no other real solutions in $(0,\alpha_1)$. This also implies that $\Phi_1(0,0)>\lambda(p)$, by \eqref{derivative-b}, there are no solutions in the imaginary axis.

\bigskip

Moreover, using the fact that $\Phi_m(\alpha,0)$ is increasing in $m$, and $\Phi_m(\pm B_m, 0)=0$, we can get a sequence of real solutions $\alpha_m\in (0, B_m)$ for $m\geq 1$ that is increasing. Moreover, from \eqref{claim-1}, one also has that in the region $|\alpha|\leq B_m$, all the solutions to \eqref{indicialm} are real.

Then, denoting the solutions to \eqref{indicialm} by $\sigma_j^{(m)}\pm i\tau_j^{(m)}$ and $-\sigma_j^{(m)}\pm i\tau_j^{(m)}$ for $m\geq 1$, we conclude that:
\begin{equation*}
\sigma_0^{(1)}=\tfrac{2\gamma}{p-1}+1-\tfrac{N-2\gamma}{2}, \quad \{\sigma_0^{(m)}\} \mbox{ is \ increasing}, \quad \tau_0^{(m)}=0.\\
\end{equation*}

We finally consider statement \emph{i.} in the Lemma and look for the indicial roots of $\mathcal L_1$ at $r=+\infty$. In this case, $\delta$ will satisfy the following equation:
\begin{equation*}
2^{2\gamma}\frac{\Big|\Gamma \Big(\tfrac{1}{2}+\tfrac{\gamma}{2}
+\tfrac{1}{2}\sqrt{\big(\tfrac{N}{2}-1\big)^2+\mu_m}+\frac{\delta}{2}\Big)\Big|^2}
{\Big|\Gamma\Big(\tfrac{1}{2}-\tfrac{\gamma}{2}+\tfrac{1}{2}\sqrt{\big(\tfrac{N}{2}-1\big)^2+\mu_m}
+\frac{\delta}{2}\Big)\Big|^2}=0.
\end{equation*}
For each fixed $m=0,1,\ldots$, the indicial roots occur when
\begin{equation*}
\tfrac{1}{2}-\tfrac{\gamma}{2}+\tfrac{1}{2}\sqrt{\big(\tfrac{N}{2}-1\big)^2+\mu_m}
\pm \frac{\delta}{2}=j, \quad\mbox{for }j=0, -1, -2, \dots, -\infty,
\end{equation*}
or
\begin{equation*}
\pm \delta=(1-\gamma)+\sqrt{(\tfrac{N}{2}-1)^2+\mu_m}+2j, \quad j=0,1,2,\dots.
\end{equation*}
Thus, the indicial roots for $\mathcal{L}_1$ at $r=+\infty$ are given by
\begin{equation*}
-\tfrac{N-2\gamma}{2}\pm \Big\{(1-\gamma)+\sqrt{(\tfrac{N}{2}-1)^2+\mu_m}\Big\}\pm 2j, \ \ j=0,1,\dots.
\end{equation*}
This finishes the proof of the Lemma.
\end{proof}

\subsection{Injectivity of $\mathcal L_1$ in the weighted space $\mathcal C^{2\gamma+\alpha}_{\mu, \nu_1}$}

The arguments in this section rely heavily on the results from Section \ref{section:Hardy}.  We  fix
\begin{equation}\label{restriction-mu-nu-0}
\mu>\Real({\gamma^+_0})\geq -\frac{N-2\gamma}{2},\quad \nu_1\leq\min\{ 0,\mu\}.
\end{equation}
Such $\mu$ is chosen to exclude the trivial solution $\phi_*=r\partial_r u_1+\frac{p-1}{2\gamma}u_1$.

\begin{proposition}\label{injectivity}
Under the hypothesis \eqref{restriction-mu-nu-0}, the only solution $\phi\in \mathcal C^{2\gamma+\alpha}_{\mu,\nu_1}(\R^N \setminus \{0\})$ of the equation $\mathcal L_1 \phi=0$ is the trivial solution $\phi\equiv 0$.
\end{proposition}

\begin{proof}
We would like to classify solutions to the following equation:
\begin{equation*}\label{eqlinear}
(-\Delta_{\mathbb R^N})^\gamma \phi=pA_{N,p,\gamma}u_1^{p-1}\phi \mbox{ in }\R^N\setminus \{0\},
\end{equation*}
or equivalently, \eqref{eqlinear2} or \eqref{eq0m} for each $m=0,1,\ldots$.

\bigskip

{\bf Step 1:} the mode $m=0$. Define the constant $\tau=pA_{N,p,\gamma}$ and rewrite equation \eqref{eq0}, i.e. equation \eqref{eq0m} for $m=0$, as
\begin{equation}\label{choice-h}
P_\gamma^0(w)-\tau w=(V-\tau)w=:h,
\end{equation}
for $w=w(t)$, the conjugation of $\phi_0$ according to \eqref{w-phi}, and the right hand side $h=h(t)$, both written in the variable $t=-\log{r}\in\R$. By definition,
\begin{equation}\label{eq:prop7.2w}
w(t)=
\begin{cases}
O(e^{-(\mu+\frac{N-2\gamma}{2})t})
    & \text{ as } t\to+\infty,\\
O(e^{-(\nu_1+\frac{N-2\gamma}{2})t})
    & \text{ as } t\to-\infty.\\
\end{cases}
\end{equation}
We use also \eqref{asymptotics-potential}
to estimate the right hand side,
\begin{equation*}\label{asymptotics-h}
h(t)=\begin{cases}
 O(e^{-(q_1+\mu+\frac{N-2\gamma}{2})t} )&\mbox{ as }t\to +\infty,\\
 O(e^{-(\nu_1+\frac{N-2\gamma}{2})t})&\mbox{ as }t\to -\infty.
\end{cases}\end{equation*}
By our choice of weights $\mu$, $\nu_1$ from \eqref{restriction-mu-nu-0}, we have that $\mu+\frac{N-2\gamma}{2}>\sigma_0^{(0)}$ and $\nu_1+\frac{N-2\gamma}{2}<\sigma_1^{(0)}$ (for this, recall statements \emph{a)} and \emph{b)} in Lemma \ref{indicial}).

We use Theorem \ref{thm:Hardy-potential} and Proposition \ref{prop:unstable} with $\delta=q_1+\mu+\frac{N-2\gamma}{2}>\Real(\gamma_0^++\frac{N-2\gamma}{2})=\sigma_0^{(0)}$ and $\delta_0=-(\nu_1+\frac{N-2\gamma}{2})>-\sigma_1^{(0)}$. Obviously, $\delta+\delta_0=q_1+(\mu-\nu_1)>0$. Let $J\geq0$ be such that $\sigma_J^{(0)}<\delta<\sigma_{J+1}^{(0)}$. Then we can find a particular solution $w_p$ (depending on $J$) such that
\begin{equation*}
w_p(t)=O(e^{-\delta t}), \quad\text{as}\quad t\to +\infty.
\end{equation*}
Our solution will be obtained from this particular solution by adding elements in the kernel. By our choice of $\mu$ (hence $\delta$), we could only add exponentials $e^{-\sigma_j^{(0)} t}$, $j>J$. But these grow fast at $-\infty$ and thus are not allowed by the choice of $\nu_1$, and we must have $w=w_p$. But again, by the theorem, either $w_p$ grows as $t\to -\infty$ like $O(e^{\delta_0 t})$,  or it has the asymptotic behavior $w_p\sim e^{-\sigma_j^{(0)} t}$ for some $j=0,\ldots,J$ (maybe after passing to a subsequence in $t$, in case $\tau_j^{(0)}\neq 0$). But by \eqref{eq:prop7.2w} this second scenario is not allowed unless $j=0$, so we must conclude that $w=w_p$ satisfies
\begin{equation*}
w(t)=
\begin{cases}
O(e^{-\delta t})
    &\text{ as }\quad t\to +\infty,\\
O(e^{\min\{-\sigma_0^{(0)},\delta_0\} t})
    &\text{ as }\quad t\to-\infty.
\end{cases}
\end{equation*}
With this and \eqref{eq:prop7.2w} we see that we have improved the decay of $w=w_p$ at $+\infty$ by $e^{-q_1t}$ (at the expense of slightly worsening the behavior at $-\infty$).

Now, by the definition of $h$ in \eqref{choice-h}, we can iterate this process with $\delta'=lq_1+\mu+\frac{N-2\gamma}{2}$, $l\geq 2$, and $\delta_0'=\min\{-\sigma_0,\delta_0\}$, to obtain better  decay when $t\to +\infty$. Indeed, with these choices of the parameters we still have $\delta'+\delta_0'=(l-1)q_1+\delta+\min\set{-\sigma_0,\delta_0}>0$, so that the results in Section \ref{section:Hardy} remain applicable.  The Hamiltonian argument from Corollary \ref{cor:other-solutions} shows that there are no decaying solutions on both sides $t\to\pm\infty$ except for $\phi_*$. As a consequence, we have that $w$ decays faster than any $e^{-\delta t}$ as $t\to +\infty$, which translated to $\phi$ means that $\phi=o(r^a)$ as $r\to 0$ for every $a\in\mathbb N$. Note that the strong unique continuation result of \cite{Fall-Felli} (stable case) and \cite{Ruland} (unstable case) for the operator $P_\gamma^0-V$ implies that $\phi$ must vanish everywhere.

\bigskip

{\bf Step 2:} the modes $m=1,\ldots,N$.   Differentiating the equation \eqref{Lane-Emden} 
we get
\begin{equation*}
\mathcal L_1\frac{\partial u_1}{\partial x_m}=0.
\end{equation*}
Since $u_1$ only depends on $r$, we have $\frac{\partial u_1}{\partial x_m}=u_1'(r)E_m$, where $E_m=\frac{x_m}{|x|}$. Using the fact that $-\Delta_{\mathbb S^{N-1}} E_m=\mu_m E_m$, the extension for $u'_1(r)$ to $\mathbb R^{N+1}_+$
solves \eqref{eqlinear2} with eigenvalue $N-1$, and $w_1:=r^{\frac{N-2\gamma}{2}}u_1'$ satisfies $P_\gamma^m w-Vw=0$. Note that $u'_1$ decays like $r^{-(N+1-2\gamma)}$ as $r\to \infty$ and blows up like $r^{-\frac{2\gamma}{p-1}-1}$ as $r\to 0$.

We know that also $\phi_m$ solves \eqref{eq0m}. Assume it  decays like $r^{-(N+1-2\gamma)}$ as $r\to \infty$ and blows up like $r^{\gamma_m^+}$ as $r\to 0$. Then we can find a non-trivial combination of $u'_1$ and $\phi_m$ that decays faster than $r^{-(N+1-2\gamma)}$ at infinity. Since their singularities at $0$ cannot cancel, this combination is non-trivial.

Now we claim that no solution to \eqref{eq0m} can decay faster than $r^{-(N+1-2\gamma)}$ at $\infty$, which is a contradiction and yields that $\phi_m=0$ for $m=1, \ldots, N$.

To show this claim we argue as in Step 1, using the indicial roots at infinity (namely $-(N+1-2\gamma)$ and $1$) and interchanging the role of $+\infty$ and $-\infty$ in the decay estimate. Using the facts that the solution decays like $r^{\sigma}$ for some $\sigma<{-(N-2\gamma+1)}$, i.e. $\sigma+\frac{N-2\gamma}{2}<-\frac{N-2\gamma}{2}-1=-\sigma_0^{(1)}$ and $\Real(\gamma_m^+)+\frac{N-2\gamma}{2}<\sigma_1^{(1)}$,  one can show that the solution is identically zero.

\bigskip

{\bf Step 3:} the remaining modes $m\geq{N+1}$. We use an integral estimate involving the first mode which has a sign, as in \cite{Davila-delPino-Musso-Wei,Davila-delPino-Musso}.
We note that, in particular, $\phi_1(r)=-u_1'(r)>0$, which also implies that its extension $\Phi_1$ is positive. In general, the $\gamma$-harmonic extension $\Phi_m$ of $\phi_m$ satisfies
\[\left\{
\begin{array}{r@{}ll}
\divergence(\ell^{1-2\gamma}\nabla\Phi_m)
    &\,=\mu_{m}\dfrac{\ell^{1-2\gamma}}{r^2}\Phi_m
    &\text{ in } \R^{N+1}_+,\medskip\\
-\tilde{d}_\gamma\lim\limits_{\ell\to0}\ell^{1-2\gamma}\p_{\ell}\Phi_m
    &\,=pu_1^{p-1}\phi_m
    &\text { on } \R^{N+1}_+.
\end{array}
\right.\]
We multiply this equation by $\Phi_1$ and the one with $m=1$ by $\Phi_m$. Their difference gives the equality
\begin{equation*}\begin{split}
(\mu_m-\mu_1)\dfrac{\ell^{1-2\gamma}}{r^2}\Phi_m\Phi_1
&=\Phi_1\divergence(\ell^{1-2\gamma}\nabla\Phi_m)-\Phi_m\divergence(\ell^{1-2\gamma}\nabla\Phi_1)\\
&=\divergence(\ell^{1-2\gamma}(\Phi_1\nabla\Phi_m-\Phi_m\nabla\Phi_1)).
\end{split}\end{equation*}
Let us integrate over the region where $\Phi_m>0$. The functions are regular enough near $x=0$ by the restriction \eqref{restriction-mu-nu-0}. The boundary $\p\set{\Phi_m>0}$ is decomposed into a disjoint union of $\p^0\set{\Phi_m>0}$ and $\p^+\set{\Phi_m>0}$, on which the extension variable $\ell=0$ and $\ell>0$, respectively. Hence
\begin{equation*}\begin{split}
0&\leq\tilde{d}_\gamma(\mu_m-\mu_1)\int_{\set{\Phi_m>0}}\dfrac{\Phi_m\Phi_1}{r^2}\,dxd\ell\\
&=\int_{\p^0\set{\Phi_m>0}}\left(\phi_1\lim\limits_{\ell\to0}\ell^{1-2\gamma}\pnu{\Phi_m}-\phi_m\lim\limits_{\ell\to0}\ell^{1-2\gamma}\pnu{\Phi_1}\right)\,dx\\
&\qquad+\int_{\p^+\set{\Phi_m>0}}\ell^{1-2\gamma}\left(\Phi_1\pnu{\Phi_m}-\Phi_m\pnu{\Phi_1}\right)\,dxd\ell.
\end{split}\end{equation*}
The first integral on the right hand side vanishes due to the equations $\Phi_1$ and $\Phi_m$ satisfy. Then we observe that on $\p^+\set{\Phi_m>0}$, one has $\Phi_1>0$, $\pnu{\Phi_m}\leq0$ and $\Phi_m=0$. This forces (using $\mu_m>\mu_1$)
\[\int_{\set{\Phi_m>0}}\dfrac{\Phi_m\Phi_1}{r^2}\,dxd\ell=0,\] which in turn implies $\Phi_m\leq0$. Similarly $\Phi_m\geq0$ and, therefore, $\Phi_m\equiv0$ for $m\geq{N+1}$. This completes the proof of the Proposition \ref{injectivity}.
\end{proof}

\subsection{Injectivity of $\mathbb{L}_1$ on $\mathcal C^{2\gamma+\alpha}_{\mu,\nu_1}$}

In the following, we set $N=n-k$ and consider more general equation \eqref{equation0}. Set
\begin{equation*}
\mathbb{L}_1=(-\Delta_{\mathbb R^n})^\gamma-pA_{N,p,\gamma}u_1^{p-1} \mbox{ in }\R^n \setminus \R^k.
\end{equation*}

\begin{proposition}\label{injectivity1}
Choose the weights $\mu,\nu_1$ as in Proposition  \ref{injectivity}.
The only solution $\phi\in \mathcal C^{2\gamma+\alpha}_{\mu,\nu_1}(\R^n \setminus \R^k)$ of the linearized equation $\mathbb{L}_1 \phi=0$ is the trivial solution $\phi\equiv 0$.
\end{proposition}

\begin{proof}
The idea is to use the results from Section \ref{subsection:full-symbol} to reduce $\mathbb L_1$ to the simpler $\mathcal L_1$, taking into account that $u_1$ only depends on the variable $r$ but not on $y$. In the notation of Proposition \ref{prop:symbol}, define $w=r^{-\frac{N-2\gamma}{2}}\phi$, and $w_m$ its $m$-th projection over spherical harmonics. Set $\hat w_m(\lambda,\omega)$, $\lambda\in\mathbb R$, $\omega\in \mathbb S^{k}$ to denote its Fourier-Helgason transform. By observing that the full symbol \eqref{symbol}, for each fixed $\omega$, coincides with the symbol
\eqref{symbol-isolated}, we have reduced our problem to that of Proposition \ref{injectivity}. This completes the proof.
\end{proof}

\subsection{\textit{A priori} estimates}

Now we go back to the linearized operator $L_\ve$ from \eqref{linearized-points} for the point singularity case $\mathbb R^N\setminus \{q_1,\ldots,q_K\}$, or \eqref{linearized-general} for the general $\mathbb R^n\setminus \Sigma$, and consider the equation
\begin{equation}\label{linear1}
L_\ve \phi=h.
\end{equation}

For simplicity, we use the following notation for the weighted norms
\begin{equation}\label{norm}
\|\phi\|_*=\|\phi\|_{\mathcal C^{2\gamma+\alpha}_{\mu,\nu}}, \quad \|h\|_{**}=\|h\|_{\mathcal C^{0,\alpha}_{\mu-2\gamma,\nu-2\gamma}}.
\end{equation}
Moreover, for this subsection, we assume that $\mu, \nu$ satisfy
\begin{equation*}
\Real(\gamma_0^+)<\mu\leq 0,\quad -(n-2\gamma)<\nu.
\end{equation*}
For this choice of weights we have the following \emph{a priori} estimate:

\begin{lemma}\label{lemma:apriori-estimate}
Given $h$ with $\|h\|_{**}<\infty$, suppose that $\phi$ be a solution of \eqref{linear1}, then there exists a constant $C$ independent of $\ve$ such that
\begin{equation*}\label{apriori}
\|\phi\|_*\leq C\|h\|_{**}.
\end{equation*}
\end{lemma}

\begin{proof}
We will argue by contradiction. Assume that there exists $\ve_j\to 0$, and a sequence of solutions $\{\phi_j\}$ to $L_{\ve_j}\phi_j=h_j$ such that
\begin{equation*}
\|\phi_j\|_{*}=1, \quad\text{and}\quad \|h_j\|_{**}\to 0 \mbox{ as }j\to \infty.
\end{equation*}
In the following we will drop the index $j$ without confusion.

We first consider the point singularity case $\mathbb R^N\setminus\Sigma$ for $\Sigma=\{q_1,\ldots,q_K\}$. By Green's representation formula one has
\begin{equation*}
\begin{split}
\phi(x)&=\int_{\R^N}G(x,\tilde x)[h+pA_{N,p,\gamma}\bar{u}_\ve^{p-1}\phi]\,d\tilde{x}
=:I_1+I_2,
\end{split}
\end{equation*}
where $G$ is the Green's function for the fractional Laplacian $(-\Delta_{\mathbb R^N})^\gamma$ given by $G(x,\tilde x)=C_{N,\gamma}|x-\tilde x|^{-(N-2\gamma)}$ for some normalization constant $C_{N,\gamma}$.
In the first step, let $x\in \R^{N} \setminus \bigcup_i B_\sigma(q_i)$. In this case
\begin{equation*}
\begin{split}
I_1&\lesssim\int_{\mathbb R^N} G(x,\tilde x)h(\tilde x)\,d\tilde x\\
&=\int_{\{\dist(\tilde x,\Sigma)<\frac{\sigma}{2}\}}\cdots
+\int_{\{\frac{\sigma}{2}<\dist(\tilde x,\Sigma)<\frac{|x|}{2}\}}\cdots
+\int_{\{\frac{|x|}{2}<\dist(\tilde x,\Sigma)<2|x|\}}\cdots
+\int_{\{\dist(\tilde x,\Sigma)>2|x|\}}\cdots\\
&\leq C\|h\|_{**}(|x|^{-(N-2\gamma)}+|x|^{\nu})\\
&\leq C\|h\|_{**}|x|^{\nu},
\end{split}
\end{equation*}
because of our restriction of $\nu$. Moreover,
\begin{equation*}
\begin{split}
I_2&=\int G(x,\tilde x)p\bar{u}_\ve(\tilde x)^{p-1}\phi(\tilde x)\,d\tilde x\\
&=\int_{\{\dist(\tilde x,\Sigma)<\ve\}}\cdots
+\int_{\{\frac{\sigma}{2}>\dist(\tilde x,\Sigma)>\ve\}}\cdots
+\int_{\{\dist(\tilde x,\Sigma)>\frac{\sigma}{2}\}}\cdots=:I_{21}+I_{22}+I_{23}.
\end{split}
\end{equation*}
Calculate
\begin{equation*}
\begin{split}
I_{21}&\leq \int_{\{\dist(\tilde x,\Sigma)<\ve\}}|x-\tilde x|^{-(N-2\gamma)}\varrho(\tilde x)^{-2\gamma}\phi
\leq \|\phi\|_*\int_{\{\dist(\tilde x,\Sigma)<\ve\}}|x-\tilde x|^{-(N-2\gamma)}\varrho(\tilde x)^{\mu-2\gamma}\,d\tilde x,\\
&\leq C\ve^{N+\mu-2\gamma}\|\phi\|_*\varrho(x)^{-(N-2\gamma)},
\end{split}
\end{equation*}
where $\varrho $ is the weight function defined in Section \ref{section:function-spaces}, and

\begin{equation*}
\begin{split}
I_{22}&=\int_{\{\dist(\tilde x,\Sigma)>\frac{\sigma}{2}\}}|x-\tilde x|^{-(N-2\gamma)}\ve^{N(p-1)-2p\gamma}\varrho(\tilde x)^{-(N-2\gamma)(p-1)}\phi\,d\tilde x\\
&\leq \|\phi\|_*\int_{\{R>\dist(\tilde x,\Sigma)>\frac{\sigma}{2}\}}|x-\tilde x|^{-(N-2\gamma)}\ve^{N(p-1)-2p\gamma}\varrho(\tilde x)^{\mu-(N-2\gamma)(p-1)}\,d\tilde x\\
&\quad\,+\|\phi\|_*\int_{\{\dist(\tilde x,\Sigma)>R\}}|x-\tilde x|^{-(N-2\gamma)}\ve^{N(p-1)-2p\gamma}\varrho(\tilde x)^{\nu-(N-2\gamma)(p-1)}\Big]\,d\tilde x\\
&\lesssim \ve^{N(p-1)-2p\gamma}\varrho(x)^{-(N-2\gamma)}\|\phi\|_*.
\end{split}
\end{equation*}
Next for $I_{23}$,
\begin{equation*}
\begin{split}
I_{23}&
\lesssim\int_{\{\ve<\dist(\tilde x,\Sigma)<\frac{\sigma}{2}\}}|x-\tilde x|^{-(N-2\gamma)}\ve^{N(p-1)-2p\gamma}\varrho(\tilde x)^{-(N-2\gamma)(p-1)}\phi\,d\tilde x\\
&\leq \ve^{N(p-1)-2p\gamma}\varrho(x)^{-(N-2\gamma)}\|\phi\|_*\int_{\{\ve<|\tilde x|<\frac{\sigma}{2}\}}|\tilde x|^{\mu-(N-2\gamma)(p-1)}\,d\tilde x\\
&\lesssim (\ve^{N(p-1)-2p\gamma}+\ve^{\mu-2\gamma+N})\|\phi\|_*\varrho(x)^{-(N-2\gamma)}.
\end{split}
\end{equation*}
Combining the above estimates, one has
\begin{equation*}
\begin{split}
I_2&\leq C(\ve^{N(p-1)-2p\gamma}+\ve^{\mu-2\gamma+N})\varrho(x)^{-(N-2\gamma)}\|\phi\|_*\\
&\lesssim (\ve^{N(p-1)-2p\gamma}+\ve^{\mu-2\gamma+N})\varrho(x)^\nu \|\phi\|_*,
\end{split}
\end{equation*}
and thus
\begin{equation*}
\sup_{\{\dist(x,\Sigma)>\sigma\}}\{\varrho(x)^{-\nu}|\phi|\}\leq C(\|h\|_{**}+o(1)\|\phi\|_*),
\end{equation*}
which implies, because our initial assumptions on $\phi$, that there exists $q_i$ such that
\begin{equation}\label{contradiction}
\sup_{\{|x-q_i|<\sigma\}}|x-q_i|^{-\mu}|\phi|\geq \frac{1}{2}.
\end{equation}

In the second step we study the region $|x-q_i|<\sigma$. Without loss of generality, assume $q_i=0$. Recall that we are writing $\phi=I_1+I_2$. On the one hand,
\begin{equation*}
\begin{split}
I_1&=\int_{\R^N}G(x,\tilde x)h(\tilde{x})\,d\tilde x\\
&=\int_{\{|\tilde x|>2\sigma\}}\cdots+\int_{\{|\tilde x|<\frac{|x|}{2}\}}\cdots
+\int_{\{\frac{|x|}{2}<|\tilde x|<2|x|\}}\cdots+\int_{\{2|x|<|\tilde x|<2\sigma\}}\cdots\\
&\leq c\|h\|_{**}\Big[ \int_{\{|\tilde x|>2\sigma\}}|x-\tilde x|^{-(N-2\gamma)}|\tilde x|^{\nu-2\gamma}\,d\tilde x
+\int_{\{|\tilde x|<\frac{|x|}{2}\}}|x-\tilde x|^{-(N-2\gamma)}|\tilde x|^{\mu-2\gamma}\,d\tilde x\\
&\quad+\int_{\{\frac{|x|}{2}<|\tilde x|<2|x|\}}
|x-\tilde x|^{-(N-2\gamma)}|\tilde x|^{\mu-2\gamma}\,d\tilde x
+\int_{\{2|x|<|\tilde x|<2\sigma\}}|x-\tilde x|^{-(N-2\gamma)}|\tilde x|^{\mu-2\gamma}  \,d\tilde x\Big]\\
&\leq C\|h\|_{**}|x|^\mu.
\end{split}
\end{equation*}
On the other hand, for $I_2$, recall that $\phi$ is a solution to
\begin{equation*}
(-\Delta_{\mathbb R^N})^\gamma \phi-pA_{N,p,\gamma}\bar{u}_\ve^{p-1}\phi=h.
\end{equation*}
Define $\bar{\phi}(\tilde x)=\ve^{-\mu}\phi(\ve \tilde x)$, then $\bar{\phi}$ satisfies
\begin{equation*}
(-\Delta_{\mathbb R^N})^\gamma \bar{\phi}-pA_{N,p,\gamma}u_1^{p-1}\bar{\phi}=\ve^{2\gamma-\mu}h(\ve\tilde{x}).
\end{equation*}
By the assumption that $\|h\|_{**}\to 0$, one has that the right hand side tends to $0$ as $j\to \infty$.
Since $|\bar{\phi}(\tilde{x})|\leq C\|\phi\|_{*}|\tilde{x}|^{\mu}$ locally uniformly, and by regularity theory,
$\bar{\phi}\in \mathcal C^{\eta}_{loc}(\R^N\setminus \{0\})$ for some $\eta\in (0,1)$, thus passing to a subsequence, $\bar{\phi}\to \phi_\infty$ locally uniformly in any compact set, where $\phi_\infty\in \mathcal C_{\mu,\mu}^{\alpha+2\gamma}(\R^N \setminus \{0\})$ is a solution of
\begin{equation}\label{limit-equation}
(-\Delta_{\mathbb R^N})^\gamma \phi_\infty-pA_{N,p,\gamma}u_1^{p-1}\phi_\infty=0\quad\mbox{in }\mathbb R^N\setminus\{0\}
\end{equation}
(to handle the non-locality we may pass to the extension in a standard way). Since $\mu\leq 0$, it satisfies the condition in Proposition \ref{injectivity}, from which we get that $\phi_\infty \equiv 0$, so $\bar{\phi}\to 0$.


Now we go back to the calculation of $I_2$. Here we use the change of variable $x=\ve x_1$.
\begin{equation*}
\begin{split}
I_2&=\int_{\{|\tilde x|<\sigma\}}p|x-\tilde x|^{-(N-2\gamma)}\bar{u}_\ve^{p-1}\phi\,d\tilde x
=\ve^{\mu}\int_{\{|\tilde x|<\frac{\sigma}{\ve}\}}|x_1-\tilde x|^{-(N-2\gamma)}u_1^{p-1}(\tilde x)\bar{\phi}(\tilde x)\,d\tilde x\\
&=\ve^{\mu}\Big[\int_{\{|\tilde x|<\frac{1}{R}\}}\cdots+\int_{\{\frac{1}{R}<|\tilde x|<R\}}\cdots
+\int_{\{R<|\tilde x|<\frac{\sigma}{\ve}\}}\cdots \Big]=:J_1+J_2+J_3,
\end{split}
\end{equation*}
for some positive constant $R$ large enough to be determined later. For $J_1$, fix $x$, when $\ve\to 0$ one has
\begin{equation*}
\begin{split}
J_1&=\ve^{\mu} \int_{\{|\tilde x|<\frac{1}{R}\}}|x_1-\tilde x|^{-(N-2\gamma)}u_1^{p-1}(\tilde x)\bar{\phi}(\tilde x)\,d\tilde x\leq \ve^{\mu} \|\phi\|_*|x_1|^{-(N-2\gamma)}\int_{\{|\tilde x|<\frac{1}{R}\}}|\tilde x|^{\mu-2\gamma}\,d\tilde x\\
&\leq CR^{-(N-2\gamma+\mu)}\|\phi\|_*|x|^{\mu}.
\end{split}
\end{equation*}
For $J_2$ we use the fact that in this region $\bar{\phi}\to 0$, so
\begin{equation*}
\begin{split}
J_2&=\ve^{\mu} \int_{\{\frac{1}{R}<|\tilde x|<R\}}|x_1-\tilde x|^{-(N-2\gamma)}u_1^{p-1}(\tilde x)\bar{\phi}(\tilde x)\,d\tilde x\\
&=o(1)\ve^{\mu}\int_{\{\frac{1}{R}<|\tilde x|<R\}}\frac{1}{|x_1-\tilde x|^{N-2\gamma}}\,d\tilde x
=o(1)\ve^\mu |x_1|^{-(N-2\gamma)}=o(1)|x|^\mu,
\end{split}
\end{equation*}
and finally,
\begin{equation*}
\begin{split}
J_3&
=\ve^\mu\int_{\{R<|\tilde x|<\frac{\sigma}{\ve}\}}|x_1-\tilde x|^{-(N-2\gamma)}u_1^{p-1}(\tilde x)\bar{\phi}(\tilde x)\,d\tilde x\\
&=\ve^\mu \|\phi\|_*\Big[\int_{\{R<|\tilde x|<\frac{|x_1|}{2}\}}|x_1-\tilde x|^{-(N-2\gamma)}|\tilde x|^\mu u_1^{p-1}\,d\tilde x+\int_{\{\frac{|x_1|}{2}<|\tilde x|<2|x_1|\}}|x_1-\tilde x|^{-(N-2\gamma)}u_1^{p-1}|\tilde x|^\mu\,d\tilde x\\
&\quad+\int_{\{2|x_1|<|\tilde x|<\frac{\sigma}{\ve}\}}|x_1-\tilde x|^{-(N-2\gamma)}u_1^{p-1}|\tilde x|^\mu\,d\tilde x\Big]\\
&\leq C\ve^{\mu}|x_1|^\mu \|\phi\|_*|x_1|^{-\tau}
\leq o(1)\|\phi\|_*|x|^\mu
\end{split}
\end{equation*}
for some $\tau>0$.

Combining all the above estimates, one has
\begin{equation*}
||x|^{-\mu}I_1|\leq o(1)(\|\phi\|_*+1),
\end{equation*}
which implies
\begin{equation*}
\|\phi\|_*\leq o(1)\|\phi\|_*+o(1)+\|h\|_{**}=o(1).
\end{equation*}
This is a contradiction with \eqref{contradiction}.

\bigskip

For the more general case $\mathbb R^n\setminus \Sigma$ when $\Sigma$ is a smooth  $k$-dimensional sub-manifold, the argument is similar as above, the only difference is that one arrives to the analogous to \eqref{limit-equation} in the estimate for $I_2$ near $\Sigma$:
\begin{equation*}
(-\Delta_{\mathbb R^n})^\gamma \phi_\infty-pu_1^{p-1}\phi_\infty=0\quad\mbox{in }\mathbb R^n\setminus\mathbb R^k.
\end{equation*}
After the obvious rescaling by the constant $A_{N,p,\gamma}$, where $N=n-k$, one uses Remark \ref{remark} and the injectivity result in Proposition \ref{injectivity1} instead of the one in Proposition \ref{injectivity}. This completes the proof of Lemma \ref{lemma:apriori-estimate}.
\end{proof}

\section{Fredholm properties - surjectivity}\label{section:Fredholm}


Our analysis here follows closely the one in \cite{Pacard:lectures} for the local case. These lecture notes are available online but, unfortunately, yet to be published.

For the rest of the paper, we will take the pair of dual weights $\mu,\tilde\mu$ such that $\mu+\tilde\mu=-(N-2\gamma)$ and $\nu+\tilde\nu=-(n-2\gamma)$  satisfying
\begin{equation}\label{choice-mu-tilde-mu}
\begin{split}
&-\frac{2\gamma}{p-1}<\tilde\mu<\Real(\gamma_0^-)\leq -\frac{N-2\gamma}{2}\leq \Real(\gamma_0^+)<\mu< 0,\\
&-(n-2\gamma)<\tilde\nu<0.
\end{split}
\end{equation}

In order to consider the invertibility of the linear operators \eqref{linearized-general} and \eqref{linearized-points}, defined in the spaces
\begin{equation*}
L_\ve:\mathcal C^{2\gamma+\alpha}_{\tilde\mu,\tilde\nu}\to \mathcal C^{0,\alpha}_{\tilde\mu-2\gamma,\tilde\nu-2\gamma},
\end{equation*}
it is simpler to consider the conjugate operator
\begin{equation}\label{conjugate-operator}
\tilde L_\ve(w):=f^{-1}_1L_\ve(f_2 w),\quad \tilde L_\ve:\mathcal C^{2\alpha+\gamma}_{\tilde\mu+\frac{N-2\gamma}{2},\tilde\nu+\frac{n-2\gamma}{2}}\to \mathcal C^{2\alpha+\gamma}_{\tilde\mu+\frac{N-2\gamma}{2},\tilde\nu+\frac{n-2\gamma}{2}},
\end{equation}
where $f_2$ is a weight $\varrho^{-\frac{n-2\gamma}{2}}$ near infinity, and $\varrho^{-\frac{N-2\gamma}{2}}$ near the singular set $\Sigma$, while $f_1$ is $\varrho^{-\frac{n+2\gamma}{2}}$ near infinity, and $\varrho^{-\frac{N+2\gamma}{2}}$ near the singular set. Recall that $\varrho$ is defined in Section \ref{section:function-spaces}. In addition, in each neighborhood use Fermi coordinates as in formula \eqref{norm-conjugate} below. This conjugate operator is better behaved in  weighted Hilbert spaces and simplifies the notation in the proof of Fredholm properties.

\subsection{Fredholm properties}

Fredholm properties for extension problems related to this type of operators were considered in \cite{Mazzeo:edge,Mazzeo:edge2}.

In the notation of Section \ref{section:function-spaces},  and following the paper \cite{Mazzeo-Pacard}, we define the weighted Lebesgue space $L^2_{\delta,\vartheta}(\mathbb R^n\setminus\Sigma)$.
These are $L^2_{\text{loc}}$ functions for which the norm
\begin{equation}\label{L2-norm}
\|\phi\|^2_{L^2_{\delta,\vartheta}(\mathbb R^n\setminus\Sigma)}=\int_{\mathbb R^n\setminus B_R} |\phi|^2\varrho^{ -2\gamma-2\vartheta}\,dz +\int_{B_R\setminus\mathcal T_\sigma} |\phi|^2\,dz+\int_{\mathcal T_\sigma}|\phi|^2\varrho^{N-1-2\gamma-2\delta}\,drdyd\theta
\end{equation}
is finite. Here $drdyd\theta$ denotes the corresponding measure in Fermi coordinates $r>0$, $y\in\Sigma$, $\theta\in\mathbb S^{N-1}$. One defines accordingly, for $\gamma>0$, weighted Sobolev spaces  $W^{2\gamma,2}_{\delta,\vartheta}$  with respect to the vector fields from Remark \ref{remark:edge} (see \cite{Mazzeo:edge} for the precise definitions).

The seemingly unusual normalization in the integrals in \eqref{L2-norm} is explained by
the change of variable $\phi=f_2w$. Indeed,
\begin{equation}\label{norm-conjugate}
\begin{split}
\|w\|^2_{L^2_{\delta,\vartheta}}&=\int_{-\infty}^{-\log R}\int_{\mathbb S^{n-1}}|w|^2e^{2\vartheta \tilde t}\,d\tilde \theta d\tilde t +\int_{\{\dist(\cdot,\Sigma)>\sigma,|z|<R\}}|w|^2\,dz\\
&+\int_{-\log \sigma}^{+\infty} \int_{\Sigma}\int_{\mathbb S^{N-1}}|w|^2e^{2\delta t}\, d\theta dy  dt.
\end{split}
\end{equation}

We have
\begin{lemma}\label{lemma:inclusions}
For the choice of parameters
\begin{equation*}\label{choice-delta-vartheta}
-\delta<\tilde\mu+\tfrac{N-2\gamma}{2},\quad-\vartheta>\tilde\nu+\tfrac{n-2\gamma}{2},
\end{equation*}
we have the continuous inclusions
\begin{equation*}
\mathcal C^{2\gamma+\alpha}_{\tilde\mu,\tilde\nu}(\mathbb R^n\setminus\Sigma) \hookrightarrow  L^2_{-\delta,-\vartheta} (\mathbb R^n\setminus \Sigma).
\end{equation*}
\end{lemma}

The spaces $L^2_{\delta,\vartheta}$ and $L^2_{-\delta,-\vartheta}$ are dual with respect to the natural pairing
\begin{equation*}
\langle\phi_1,\phi_2\rangle_*=\int_{\mathbb R^n} \phi_1 \phi_2,
\end{equation*}
for  $\phi_1\in L^2_{\delta,\vartheta}$, $\phi_2\in L^2_{-\delta,-\vartheta}$. Now, let
$\tilde L_\ve$ be the operator defined in \eqref{conjugate-operator}.
It is a densely defined, closed graph operator (this is a consequence of elliptic estimates).
Then, relative to this pairing, the adjoint of
\begin{equation}\label{operator1}
\tilde L_\ve:L^2_{-\delta,-\vartheta}\to L^2_{-\delta,-\vartheta}
\end{equation}
is precisely
\begin{equation}\label{operator2}
(\tilde L_\ve)^*=\tilde L_\ve:L^2_{\delta,\vartheta}\to L^2_{\delta,\vartheta}.
\end{equation}

Now we fix $\mu$, $\tilde\mu$, $\nu$, $\tilde\nu$ as in \eqref{choice-mu-tilde-mu}, and choose $-\delta<0$ slightly smaller than $\tilde \mu+\frac{N-2\gamma}{2}$ and $-\vartheta<0$ just slightly larger than $\tilde\nu+\tfrac{n-2\gamma}{2}$ so that, in particular,
\begin{equation}\label{choice-2}\begin{split}
&-\tfrac{2\gamma}{p-1}+\tfrac{N-2\gamma}{2}<-\delta<\tilde\mu+\tfrac{N-2\gamma}{2}<0
<\mu+\tfrac{N-2\gamma}{2}<\delta<\tfrac{N-2\gamma}{2},\\
&-\tfrac{n-2\gamma}{2}<\tilde\nu
+\tfrac{n-2\gamma}{2}<-\vartheta<0<\vartheta<\tfrac{n-2\gamma}{2},
\end{split}\end{equation}
and we have the inclusions from Lemma \ref{lemma:inclusions}. In addition, we can choose $\delta,\vartheta$
different from
the corresponding indicial roots. Higher order regularity is guaranteed by the results in Section \ref{subsection:estimates-Sobolev}. We will show:

\begin{proposition}\label{prop:Fredholm}

Let $\delta\in (-\frac{N+2\gamma}{2}-2\gamma, \tfrac{N-2\gamma}{2})$ and $\vartheta\in(-\frac{n-2\gamma}{2},\frac{n-2\gamma}{2})$ be real numbers satisfying \eqref{choice-2}.

Assume that $w\in L^2_{\delta,\vartheta}$ is a solution to
\begin{equation*}
\tilde L_\ve w= \tilde h \mbox{ on }\mathbb R^n\setminus\Sigma
\end{equation*}
for $\tilde h\in L^2_{\delta,\vartheta}$. Then we have the \emph{a priori} estimate
\begin{equation}\label{estimate1}
\|w\|_{L^2_{\delta,\vartheta}}\lesssim \|\tilde h\|_{L^2_{\delta,\vartheta}}+\|w\|_{L^2(\mathcal K)},
\end{equation}
where $\mathcal K$ is a compact set in $\mathbb R^n\setminus \Sigma$. Translating back to the original operator $L_\ve$, if $\phi$ is a solution to $L_\ve\phi=h$ in $\mathbb R\setminus\Sigma$, then \eqref{estimate1} is rewritten as
\begin{equation}\label{estimate11}
\|\phi\|_{L^2_{\delta+2\gamma,\vartheta+2\gamma}}\lesssim \| h\|_{L^2_{\delta,\vartheta}}+\|\phi\|_{L^2(\mathcal K)}.
\end{equation}

As a consequence, $L_\ve$ has good Fredholm properties. The same is true for the linear operators from \eqref{operator1} and \eqref{operator2}.
\end{proposition}

\begin{proof}
The proof here  goes by subtracting suitable parametrices near $\Sigma$ and near infinity thanks to Theorem \ref{thm:Hardy-potential}. Then the remainder is a compact operator. For simplicity we set $\ve=1$.

 \bigskip

 We first consider the point singularity case, i.e., $k=0$, $n=N$.

\textbf{Step 1:} (Localization) Let us study  how the operator $L_1:L^2_{\delta+2\gamma,\vartheta+2\gamma}\to L^2_{\delta,\vartheta}$  is affected by localization, so that it is enough to work with functions supported near infinity and near the singular set.

 In the first step,  assume that the singularity happens only at $r=\infty$ (but not at $r=0$). We would like to patch a suitable parametrix at $r=\infty$. Let $\chi$ be a cut-off function such that $\chi=1$ in $\mathbb R^N\setminus B_R$, $\chi=0$ in $B_{R/2}$. Let $\mathcal K=B_{2R}$, and set
\begin{equation*}
h_1:=L_1 (\chi\phi)=\chi L_1\phi+[L_1,\chi]\phi,
\end{equation*}
where $[\cdot,\cdot]$ denotes the commutator operator. Contrary to the local case, the commutator term does not have compact support, but can still give good estimates in  weighted Lebesgue spaces by carefully controlling the the tail terms. Let
\begin{equation*}
I(x):=[L_1,\chi]\phi(x)=(-\Delta_{\mathbb R^N})^\gamma (\chi\phi)(x)-\chi(x)(-\Delta_{\mathbb R^N})^\gamma \phi (x)=k_{N,\gamma}\int_{\mathbb R^N} \frac{\chi(x)-\chi(\tilde x)}{|x-\tilde x|^{N+2\gamma}}\phi(\tilde x)\,d\tilde x.
\end{equation*}
Let us bound this integral in $L^2_{\delta,\vartheta}$.

We first consider the case $|x|\ll1$. Note that
\begin{equation*}
\begin{split}
I(x)&=\int_{B_{2R}\setminus B_{R/2}} \frac{\chi(x)-\chi(\tilde x)}{|x-\tilde x|^{N+2\gamma}}\phi(\tilde x)\,d\tilde x+\int_{\mathbb R^N\setminus B_{2R}} \frac{\phi(\tilde x)}{|x-\tilde x|^{N+2\gamma}}\,d\tilde x\\
&\leq CR^{\frac{2\vartheta+2\gamma-N}{2}}
\|\phi\|_{L_{\delta+2\gamma,\vartheta+2\gamma}^2}+\|\phi\|_{L^2(\mathcal{K})}.
\end{split}
\end{equation*}
We have that $\|I\|_{L_{\delta}^2(B_1)}$ can be bounded by  $o(1)\|\phi\|_{L^2_{\delta+2\gamma,\vartheta+ 2\gamma}}+\|\phi\|_{L^2(\mathcal{K})}$ for large $R$ if $\vartheta<\frac{N-2\gamma}{2}$ and $\delta<\frac{N-2\gamma}{2}$.

Next, for $|x|\geq 2R$, we need to add the weight at infinity and calculate $\|I\|_{L^2_{\vartheta}(\mathbb R^N\setminus B_{2R})}$. But
\begin{equation*}
\begin{split}
I(x)&=\int_{B_{R/2}} \frac{\phi(\tilde x)}{|x-\tilde x|^{N+2\gamma}}\,d\tilde x+\int_{B_{2R}\setminus B_{R/2}} \frac{\chi(x)-\chi(\tilde x)}{|x-\tilde x|^{N+2\gamma}}\phi(\tilde x)\,d\tilde x\\
&\leq C\|\phi\|_{L_{\delta+2\gamma,\vartheta+2\gamma}^2}R^{\frac{2\vartheta+
6\gamma+N}{2}}|x|^{-(N+2\gamma)} \mbox{ if } \delta>-\tfrac{N+2\gamma}{2}-2\gamma.
\end{split}
\end{equation*}
One has that $\|I\|_{L^2_{\vartheta}(\R^N\setminus B_R)}=o(1)\|\phi\|_{L^2_{\delta+2\gamma,\vartheta+2\gamma}}$ if $\vartheta>-\frac{N+2\gamma}{2}-2\gamma$ .

Now let $1\leq|x|< 2R$, and calculate $\|I\|_{L^2(B_{2R}\setminus B_{1})}$. Again, we split
\begin{equation*}
\begin{split}
I(x)&=\int_{B_{R/2}} \frac{\chi(x)\phi(\tilde x)}{|x-\tilde x|^{N+2\gamma}}\,d\tilde x+\int_{B_{2R}\setminus B_{R/2}} \frac{\chi(x)-\chi(\tilde x)}{|x-\tilde x|^{N+2\gamma}}\phi(\tilde x)\,d\tilde x+\int_{\mathbb R^N\setminus B_{2R}} \frac{(\chi(x)-1)\phi(\tilde x)}{|x-\tilde x|^{N+2\gamma}}\,d\tilde x\\
&=:I_{21}+I_{22}+I_{23}.
\end{split}
\end{equation*}
Similar to the above estimates, we can get that the $L^2$ norm can be bounded by $$\|\phi\|_{L^2(\mathcal K)}+o(1)\|\phi\|_{L^2_{\delta+2\gamma,\vartheta+2\gamma}}$$ if $\vartheta<\frac{N-2\gamma}{2}$.
Thus we have shown that
\begin{equation*}
\|h_1\|_{L^2_{\delta,\vartheta}}\lesssim \|h\|_{L^2_{\delta,\vartheta}}+\|\phi\|_{L^2(\mathcal K)}+o(1)\|\phi\|_{L^2_{\delta+2\gamma,\vartheta+2\gamma}},
\end{equation*}
so localization does not worsen estimate \eqref{estimate11}.

In addition, the localization at $r=0$ is similar. One just needs to interchange the role of $r$ and $1/r$, and $\vartheta$ by $\delta$. Similar to the estimates in Step 5 below, one can get that the error caused by the localization can be bounded by  $\|\phi\|_{L^2(\mathcal K)}+o(1)\|\phi\|_{L^2_{\delta+2\gamma,\vartheta+2\gamma}} $ if $-\frac{N+2\gamma}{2}-2\gamma<\vartheta<\frac{N-2\gamma}{2}$, $-\frac{N+2\gamma}{2}-2\gamma<\delta<\frac{N-2\gamma}{2}$.

\bigskip

\textbf{Step 2:} (The model operator) After localization around one of the singular points, say $q_1=0$, the operator $L_1$ can be approximated by the model operator $\mathcal L_1$ from \eqref{model-linearization}, or by its conjugate given in \eqref{eq0}. Moreover, recalling the notation \eqref{potential} for the potential term and its asymptotics \eqref{asymptotics-potential}, it is enough to show that
\begin{equation}\label{inequality1}
\|w\|_{L^2_{\delta}}\lesssim \|\tilde h\|_{L^2_{\delta}},
\end{equation}
if  $w=w(t,\theta)$  is a solution of
\begin{equation}\label{equation10-2}
P_\gamma^{g_0} w-\kappa w=\tilde h,\quad t\in\mathbb R,\,\theta\in\mathbb S^{N-1},
\end{equation}
that has compact support in $t\in(0,\infty)$. Here we have denoted $\kappa=pA_{N,p,\gamma}$.

Now project over spherical harmonics, so that $w=\sum_m w_m(t) E_m(\theta)$, and $w_m$ satisfies
\begin{equation*}
P_\gamma^{m} w_m-\kappa w_m=h_m,\quad t\in\mathbb R.
\end{equation*}
Our choice of weights \eqref{choice-2} implies that there are no additional solutions to the homogeneous problem and that we can simply write our solution as
\eqref{w_m:Fourier}, in Fourier variables. Then
\begin{equation}\label{L2-estimate-Fourier}\begin{split}
\int_{\mathbb R} e^{2\delta t}|w_m(t)|^2\,dt&=\int_{\mathbb R} |w_m(\xi+\delta i)|^2\,d\xi
=\int_{\mathbb R} \frac{1}{|\Theta_\gamma^m(\xi+\delta i)-\kappa|^2} |\hat h_m(\xi+\delta i)|^2\,d\xi \\&\leq C\int_{\mathbb R} |\hat h_m(\xi+\delta i)|^2\,d\xi
=\int_{\mathbb R} e^{2\delta t}|h_m(t)|^2\,dt,
\end{split}\end{equation}
where we have used \eqref{symbol-limit}.
(note that there are no poles on the $\mathbb R+\delta i$ line).
Estimate \eqref{inequality1} follows after taking sum in $m$ and the fact that $\{E_m\}$ is an orthonormal basis.

For the estimate near infinity, we proceed in a similar manner, just approximating the potential by $\tau=0$.

\bigskip

\textbf{Step 3:} (Compactness)  Let $\{w_j\}$ be a sequence of solutions to $\tilde L_1 w_j=\tilde h_j$
with $\tilde h_j\in L^2_{\delta,\vartheta}$. Assume that we have the uniform bound $\|w_j\|_{L^2_{\delta,\vartheta}} \leq C$. Then there exists a subsequence, still denoted by $\{w_j\}$, that is convergent in $L^2_{\delta, \vartheta}$ norm. Indeed, by the regularity properties of the equation, $\|w_j\|_{W^{2\gamma,2}_{\delta,\vartheta}} \leq C$, which in particular, implies a uniform $W^{2\gamma,2}$ in every compact set $\mathcal K$. But this is enough to conclude that
$\{w_j\}$ has a convergent subsequence in $W^{2\gamma,2}(\mathcal K)$. Finally, estimate \eqref{estimate1} implies that this convergence is also  true in $L^2_{\delta,\vartheta}$, as we claimed.

\bigskip

\textbf{Step 4:} (Fredholm properties for $\tilde L_1$) This is a rather standard argument. First, assume that the kernel is infinite dimensional, and take an orthonormal basis $\{w_j\}$ for this kernel. Then, by the claim in Step 3, we can find a Cauchy subsequence. But, for this,
\begin{equation*}
\|w_j-w_{j'}\|^2=\|w_j\|^2+\|w_{j'}\|^2=2,
\end{equation*}
a contradiction.

Second, we show that the operator has closed range. Let $\{w_j\}$, $\{\tilde h_j\}$ be two sequences such that
\begin{equation}\label{equation20}
\tilde L_1 w_j=\tilde h_j\quad\text{and}\quad \tilde h_j\to h \text{ in }L^2_{\delta,\vartheta}.
\end{equation} Since $\Ker \tilde L_1$ is closed, we can use the projection theorem to write $w_j=w_j^0+w_j^1$ for $w_j^0\in\Ker \tilde L_1$ and $w_j^1\in (\Ker \tilde L_1)^\bot$. We have that $\tilde L_1 w_j^1=\tilde h_j$.

Now we claim that this sequence is uniformly bounded, i.e., $\|w_j^1\|_{L^2_{\delta,\vartheta}}\leq C$ for every $j$. By contradiction, assume that
$\|w_j^1\|_{L^2_{\delta,\vartheta}}\to \infty$ as $j\to\infty$, and rescale
\begin{equation*}
\tilde w_j=\frac{w_j^1}{\|w_j^1\|_{L^2_{\delta,\vartheta}}}
\end{equation*}
so that the new sequence has norm one in $L^2_{\delta,\vartheta}$. From the previous remark, there is a convergent subsequence, still denoted by $\{\tilde w_j\}$, i.e., $\tilde w_j\to\tilde w$ in $L^2_{\delta,\vartheta}$. Moreover, we know that $\tilde L_1\tilde w=0$. However, by assumption we have that $w_j^1\in(\Ker \tilde L_1)^\bot$, therefore so does $\tilde w$. We conclude that $\tilde w$ must vanish identically, which is a contradiction with the fact that $\|w_j^1\|_{L^2_{\delta,\vartheta}}=1$. The claim is proved.

Now, using the remark in Step 3 again, we know that there exists a convergent subsequence $w_j^1\to w^1$ in $L^2_{\delta,\vartheta}$. This $w^1$ must be regular, so we can pass to the limit in  \eqref{equation20} to conclude that $\tilde L_1(w^1)=h$, as desired.

\bigskip

\textbf{Step 5.}
Now we consider the case that $\Sigma$ is a sub-manifold of dimension $k$, and study the localization near a point in $z_0\in\Sigma$. In Fermi coordinates $z=(x,y)$, this is a similar estimate to that of \eqref{error2}.

First let $\chi$ be a cut-off function such that $\chi(r)=1$ for $r\leq d$ and $\chi(r)=0$ for $r\geq 2d$. Define $\tilde{\chi}(z)=\chi(\dist(z,\Sigma))$, and consider $\tilde{\phi}=\tilde{\chi}\phi$. Using Fermi coordinates near $\Sigma$ and around a point $z_0\in \Sigma$, that can be taken to be $z_0=(0,0)$ without loss of generality, then, for $z=(x,y)$ satisfying $|x|\ll1, |y|\ll1$, by checking the estimates in the proof of Lemma \ref{lemma:error2}, one can get that
\begin{equation}\label{equation60}
\begin{split}
(-\Delta_z)^\gamma \tilde{\phi}(z)&=(1+|x|^{\frac{1}{2}})(-\Delta_{\R^n\setminus \R^k})^\gamma \tilde{\phi}(x,y)+|x|^{-\gamma}\tilde{\phi}+\mathcal{R}_2
\end{split}
\end{equation}
where
\begin{equation*}
\begin{split}
\mathcal{R}_2&=\int_{\Sigma}\int_{\{|x|^\beta<|\tilde{x}|<2d\}}
\frac{\tilde{\phi}}{|\tilde{x}|^{N+2\gamma}}\,d\tilde{x}dy+|x|^{-\beta(N+2\gamma)}
\int_\Sigma\int_{\{|\tilde{x}|<|x|^\beta\}}\tilde{\phi}\,d\tilde{x}dy\\
&\quad+|x|^{\beta k}\int_\Sigma\int_{\{|x|^\beta<|\tilde{x}|<2d\}}\frac{\tilde{\phi}}{|\tilde{x}|^{n+2\gamma}}
\,d\tilde{x}dy\\
&\leq \|\tilde{\phi}\|_{L_{\delta+2\gamma}^2}|x|^{-\frac{1}{4}(N-2\gamma-2\delta)},
\end{split}
\end{equation*}
where we have used  H\"older inequality and that $\beta=\frac{1}{2}$. Here $\|\tilde{\phi}\|_{L^2_{\delta+2\gamma}}$ is the weighted norm near $\Sigma$.
One can easily check that the $L_{\delta}^2$ norm of $\mathcal{R}_2$ and $|x|^{-\gamma}\tilde{\phi}$ are bounded by $o(1)\|\phi\|_{L_{\delta+2\gamma,\vartheta+2\gamma}}$  for small $d$ if {$\delta<\frac{N-2\gamma}{2}$}.

Next we consider the effect of the localization. Let
\begin{equation*}
I_1(z)=[L_1,\tilde{\chi}]\phi=k_{n,\gamma}
\int_{\mathbb R^n}\frac{\tilde{\chi}(z)-\tilde{\chi}(\tilde{z})}{|z-\tilde{z}|^{n+2\gamma}}\phi(\tilde z)\,d\tilde{z}.\\
\end{equation*}
For $|z|\gg1$, one has
\begin{equation*}
\begin{split}
I_1(z)\lesssim |z|^{-(n+2\gamma)}\int_{\mathcal{T}_{2d}}\phi(\tilde{z})\,d\tilde{z}
\lesssim d^{\delta+\frac{N}{2}+3\gamma}\|\phi\|_{L_{\delta+2\gamma,\vartheta+2\gamma}^2}|z|^{-(n+2\gamma)}.
\end{split}
\end{equation*}
Adding the weight at infinity we get that $\|I_1\|_{L_{\vartheta}^2(\R^n\setminus B_R)}$ can be bounded by $o(1)\|\phi\|_{L_{\delta+2\gamma,\vartheta+2\gamma}^2}$ for $d$ small and $R$ large whenever {$\delta>-\frac{N+2\gamma}{2}-2\gamma, \ \vartheta>-\frac{n+2\gamma}{2}-2\gamma$}.

For $|x|\ll1$, one has
\begin{equation*}
\begin{split}
I_1(z)\lesssim \int_{\mathcal{T}_{2d}\setminus \mathcal{T}_d}\frac{\phi}{|z-\tilde{z}|^{n+2\gamma}}\,d\tilde{z}
+\int_{\mathcal{T}_{2d}^c}\frac{\phi}{|z-\tilde{z}|^{n+2\gamma}}\,d\tilde{z}.\\
\end{split}
\end{equation*}
One can check that the $L_{\delta}^2$ term can be bounded by
\begin{equation*}
\begin{split}
\|I_1\|_{L_{\delta}^2}&\leq C[\|\phi\|_{L^2(\mathcal{K})}
+R^{\frac{-n+2\gamma+2\vartheta}{2}}\|\phi\|_{L^2_{\vartheta+2\gamma}(B_R^c)}]\\
&\leq C[\|\phi\|_{L^2(\mathcal{K})}
+o(1)\|\phi\|_{L^2_{\delta+2\gamma,\vartheta+2\gamma}}]
\end{split}
\end{equation*}
for $d$ small and $R$ large if {$\delta<\frac{N-2\gamma}{2}, \ \vartheta<\frac{n-2\gamma}{2}$}.

For $z\in \mathcal{K}$, the estimate goes similarly for $\delta>-\frac{N+2\gamma}{2}-2\gamma$. In conclusion, we have
\begin{equation}
\|I_1\|_{L^2_{\delta,\vartheta}}\leq C\,\big[\|\phi\|_{L^2(\mathcal{K})}
+o(1)\|\phi\|_{L^2_{\delta+2\gamma,\vartheta+2\gamma}}\big].
\end{equation}
Note that this estimate only uses the values of the function $\phi$ when $|y|\ll1$. Indeed, by checking the arguments in Lemma 5.7, the main term of the expansion for the fractional Laplacian in \eqref{equation60} comes from $I_{11}$, i.e. for $|y|\ll1$. The contribution when $|y|>|x|^{\beta}$ is included in the remainder term $|x|^\gamma\tilde{\phi}+\mathcal{R}_2$. The localization around the point $z_0=(0,0)$ is now complete.

\bigskip

\textbf{Step 6.} Next, after localization, we  can replace \eqref{equation10-2} by
\begin{equation*}
P_\gamma^{g_k} w- \tau w=\tilde h, \quad\text{in }\mathbb S^{N-1}\times\mathbb H^{k+1},
\end{equation*}
and $w$ is supported only near a point $z_0\in\partial\mathbb H^{k+1}$, that can be taken arbitrarily. We first consider the spherical harmonic decomposition for $\mathbb S^{N-1}$ and recall the symbol for each projection from Theorem \ref{thm:symbol}.

The $L^2_\delta$ estimate follows similarly as in the case of points, but one uses the Fourier-Helgason transform on hyperbolic space instead of the usual Fourier transform as in Theorem \ref{thm:symbol}. Note, however,  that the hyperbolic metric in \eqref{metric-gk} is written in half-space coordinates as $\frac{dr^2+|dy|^2}{r^2}=dt^2+e^{2t}|dy|^2$, so in order to account for a weight of the form $r^{\delta}$ one would need to use this transform written in rectangular coordinates. This is well known and comes Kontorovich-Lebedev  formulas (\cite{Terras}). Nevertheless, for our purposes it is more suitable to use this transform in geodesic polar coordinates as it is described in Section \ref{subsection:transform}. To account for the weight, we just recall the following relation between two different models for hyperbolic space $\mathbb H^{k+1}$, the half space model with metrics $\frac{dr^2+|dy|^2}{r^2}$ and the hyperboloid model with metric $ds^2+\sinh s\, g_{\mathbb S^k}$ in geodesic polar coordinates:
\begin{equation*}
\cosh s=1+\frac{|y|^2+(r-2)^2}{4r}.
\end{equation*}
Since we are working locally near a point $z_0\in\partial\mathbb H^{k+1}$, we can choose $y=0$ in this relation, which yields that  $e^{-\delta t}=r^\delta=2^{\delta}e^{-\delta s}$. Thus we can use a weight of the form $e^{-\delta s}$ in replacement for $e^{-\delta t}$.

One could redo the theory of Section \ref{section:Hardy} using the Fourier-Helgason transform instead. Indeed, after projection over spherical harmonics, and following \eqref{inversion1}, we can write for $\zeta\in\mathbb H^{k+1}$,
\begin{equation*}
w_m(\zeta)=\int_{\mathbb H^{k+1}}\mathcal G(\zeta,\zeta')\bar h(\zeta')\,d\zeta',
\end{equation*}
where the Green's function is given by
\begin{equation*}
\mathcal G(\zeta,\zeta')=\int_{-\infty}^{+\infty} \frac{1}{\Theta_\gamma^m(\lambda)-\tau}k_\lambda(\zeta,\zeta')\,d\lambda.
\end{equation*}
The poles of $\frac{1}{\Theta_\gamma^m(\lambda)-\tau}$ are well characterized; in fact, they coincide with those in the point singularity case.

But instead, we can take one further reduction and consider the projection over spherical harmonics in $\mathbb S^k$. That is, in geodesic polar coordinates $\zeta=(s,\varsigma)$, $s>0$, $\varsigma\in\mathbb S^{k}$, we can write $w_m(s,\varsigma)=\sum_{j} w_{m,j}(s) E^{(k)}_j(\varsigma)$, where $E^{(k)}_j$ are the eigenfunctions for $-\Delta_{\mathbb S^k}$. Moreover, note that the symbol \eqref{symbol} is radial, so it commutes with this additional projection.

Now we can redo the estimate \eqref{L2-estimate-Fourier}, just by taking into account the following facts: first, one also has a simple Plancherel formula \eqref{Plancherel}. Second, for a radially symmetric function, the Fourier-Helgason transform takes the  form of a simple spherical transform \eqref{spherical-transform}. Third, the spherical function $\Phi_\lambda$ satisfies \eqref{spherical-estimate} and we are taking a weight of the form $e^{\delta s}$. Finally, the expression for the symbol \eqref{symbol} is the same as in the point singularity case \eqref{symbol-isolated}.

This yields estimate \eqref{estimate1} from which Fredholm properties follow immediately.
\end{proof}

\begin{remark}
We do not claim that our restrictions on $\delta,\vartheta$ in Proposition \ref{prop:Fredholm} are the sharpest possible (indeed, we chose them in the injectivity region for simplicity), but these are enough for our purposes.
\end{remark}

Gathering all restrictions on the weights we obtain:

\begin{corollary}\label{cor:surjectivity}
The operator in \eqref{operator2} is injective, both in
$\mathbb R^N\setminus \{q_1,\ldots,q_K\}$ and $\mathbb R^n\setminus \bigcup \Sigma_i$. As a consequence, its adjoint
$(\tilde L_\ve^*)^*=\tilde L_\ve$ given in \eqref{operator1} is surjective.
\end{corollary}

\begin{proof}
Lemma \ref{lemma:apriori-estimate} shows that, after conjugation, $\tilde L_\ve$ is injective in $\mathcal C^{2\gamma+\alpha}_{\tilde\mu+\frac{N-2\gamma}{2},\tilde\nu+\frac{n-2\gamma}{2}}$. By regularity estimates and our choice of $\delta,\vartheta$ from \eqref{choice-2}, we immediately obtain injectivity for \eqref{operator2}. Since, thanks to the Fredholm properties,
$$\Ker(\tilde L_\ve^*)^\bot=\rango(\tilde L_\ve),$$
the Corollary follows.
\end{proof}

\subsection{Uniform estimates}

Now we return to the operator $L_\ve$ defined in \eqref{linearized-points}, the adjoint of
\begin{equation*}
L_\ve: L_{-\delta,-\vartheta}^2\to L_{-\delta-2\gamma,-\vartheta-2\gamma}^2
\end{equation*}
is just
\begin{equation*}
L_\ve^*:  L_{\delta+2\gamma,\vartheta+2\gamma}^2\to L_{\delta,\vartheta}^2.
\end{equation*}
From the above results, one knows that $L_\ve^*$ is injective and  $L_\ve$ is surjective.

Fixing the isomorphisms
\begin{equation*}
\pi_{2\delta,2\vartheta}: L_{-\delta,-\vartheta}^2\to L_{\delta,\vartheta}^2,
\end{equation*}
we may identify the adjoint $L_\ve^*$ as
\begin{equation*}
L_\ve^*=\pi_{-2\delta,-2\vartheta}
\circ L_\ve\circ \pi_{2\delta,2\vartheta}: L_{-\delta+2\gamma,-\vartheta+2\gamma}^2\to L_{-\delta,-\vartheta}^2.
\end{equation*}
Now we have a new operator
\begin{equation*}
\mathbf{L}_\ve=L_\ve\circ L_\ve^*: L_\ve\circ \pi_{-2\delta,-2\vartheta}\circ L_\ve\circ \pi_{2\delta,2\vartheta}: L_{-\delta+2\gamma,-\vartheta+2\gamma}^2\to L_{-\delta-2\gamma,-\vartheta-2\gamma}^2.
\end{equation*}
This map is an isomorphism. Hence there exists a bounded two sided inverse
\begin{equation*}
\mathbf{G}_\ve:L_{-\delta-2\gamma,
-\vartheta-2\gamma}^2\to L_{-\delta+2\gamma,-\vartheta+2\gamma}^2.
\end{equation*}
Moreover, $G_\ve=L_\ve^*\circ \mathbf{G}_\ve$ is right inverse of $L_\ve$ which map into the range of $L_\ve^*$. We will fix our inverse to be this one.

From Corollary \ref{cor:surjectivity}
\begin{equation*}
\mathbf{G}_\ve: \mathcal C^{0,\alpha}_{\tilde{\mu}-2\gamma,\tilde{\nu}-2\gamma}\to \mathcal C^{4\gamma+\alpha}_{\tilde{\mu}+2\gamma, \tilde{\nu}+2\gamma}
\end{equation*}
and
\begin{equation*}
G_\ve:\mathcal C^{0,\alpha}_{\tilde{\mu}-2\gamma,\tilde{\nu}-2\gamma}\to \mathcal C^{2\gamma+\alpha}_{\tilde{\mu}, \tilde{\nu}}
\end{equation*}
are bounded.

\bigskip

We are now in the position to prove  uniform surjectivity. It is a consequence of the following two results:

\begin{lemma}
If $u\in \mathcal C^{2\gamma+\alpha}_{\tilde{\mu}, \tilde{\nu}}$ and $v\in \mathcal C^{4\gamma+\alpha}_{\tilde{\mu}+2\gamma, \tilde{\nu}+2\gamma}$ solve equations $L_\ve u=0, \ u=L_\ve^* v$, then $u\equiv v\equiv 0$.
\end{lemma}

\begin{proof}
Suppose $u,v$ satisfy the given system, then one has $L_\ve L_\ve^* v=0$. Consider $\tilde w=\pi_{2\delta,2\vartheta}v$. Multiply the equation by $w$; integration by parts in $\R^n$ yields
\begin{equation*}
0=\int w L_\ve\circ \pi_{-2\delta, -2\vartheta}\circ L_\ve w =\int \pi_{-2\delta, -2\vartheta}|L_\ve w|^2.
\end{equation*}
Thus  $L_\ve w=0$. Moreover, since $v\in \mathcal C^{4\gamma+\alpha}_{\tilde{\mu}+2\gamma, \tilde{\nu}+2\gamma}$, one has $w\in \mathcal C^{2\gamma+\alpha}_{\tilde{\mu}+2\gamma+2\delta_{\tilde{\mu}}, \tilde{\nu}+2\gamma+2\delta_{\tilde{\nu}}}\hookrightarrow \mathcal C^{2\gamma+\alpha}_{\mu',\nu'}$ for some $\mu'>\Real(\gamma_0^+),  \nu'>-(n-2\gamma)$ , thus by the injectivity property, one has $w\equiv 0$. We conclude then that $u\equiv v\equiv 0$.
\end{proof}

\begin{lemma}\label{lemma-uniform-surjectivity}
Let $G_\ve$ be the bounded inverse of $L_\ve$ introduced above, then for $\ve$ small, $G_\ve$ is uniformly bounded, i.e. for $h\in \mathcal C^{0,\alpha}_{\tilde{\mu}-2\gamma,\tilde{\nu}-2\gamma}$, if $u\in \mathcal C^{2\gamma+\alpha}_{\tilde{\mu}, \tilde{\nu}},v\in \mathcal C^{4\gamma+\alpha}_{\tilde{\mu}+2\gamma,\tilde{\nu}+2\gamma}$ solve the system $L_\ve u=h$ and $L_\ve^* v=u$, then one has
\begin{equation*}
\|u\|_{\mathcal C^{2\gamma+\alpha}_{\tilde{\mu}, \tilde{\nu}}(\R^n\setminus \Sigma)}\leq C\|h\|_{\mathcal C^{0,\alpha}_{\tilde{\mu}-2\gamma,\tilde{\nu}-2\gamma}(\R^n \setminus \Sigma)}
\end{equation*}
for some $C>0$ independent of $\ve$ small.
\end{lemma}

\begin{proof}
The proof is similar to the proof of Lemma
\ref{lemma:apriori-estimate}. So we just sketch the proof here and point out the differences. It is by contradiction argument. Assume that there exists $\{\ve^{(n)}\}\to 0$ and a sequence of functions $\{h^{(n)}\}$ and solutions $\{u^{(n)}\}$, $\{v^{(n)}\}$ such that
\begin{equation*}
\|u\|_{\mathcal C^{2\gamma+\alpha}_{\tilde{\mu}, \tilde{\nu}}(\R^n\setminus \Sigma)}=1, \quad \|h\|_{\mathcal C^{0,\alpha}_{\tilde{\mu}-2\gamma,\tilde{\nu}-2\gamma}(\R^n \setminus \Sigma)}\to 0,
\end{equation*}
and solve the equation
\begin{equation*}
L_\ve u=h, \ L_\ve^* v=u.
\end{equation*}
Here note that, for simplicity, we have dropped the superindex $(n)$. Then using the Green's representation formula, following the argument in Proposition \ref{lemma:apriori-estimate}, one can show that
\begin{equation*}
\sup_{\{\dist(x,\Sigma)>\sigma\}}\{\varrho(x)^{-\tilde{\nu}}
|u|\}\leq C(\|h\|_{\mathcal C^{0,\alpha}_{\tilde{\mu}-2\gamma, \tilde{\nu}-2\gamma}}
+o(1)\|u\|_{\mathcal C^{2\gamma+\alpha}_{\tilde{\mu},\tilde{\nu}}}),
\end{equation*}
which implies that there exists $q_i$ such that
\begin{equation}\label{contradiction2}
\sup_{\{|x-q_i|<\sigma\}}|x-q_i|^{-\tilde{\mu}}|u|\geq \frac{1}{2}.
\end{equation}

In the second step we study the region $\{|x-q_i|<\sigma\}$. Without loss of generality, assume $q_i=0$. Define the rescaled function as $\bar{u}=\ve^{-\tilde{\mu}}u(\ve x)$ and similarly for $\bar{v}$ and $\bar{h}$. Similarly to the argument in \ref{lemma:apriori-estimate}, $\bar{u}$ will tend to a limit $u_\infty$ that solves
\begin{equation*}
(-\Delta)^\gamma u_\infty-pA_{N,p,\gamma}u_1^{p-1}u_\infty=0 \quad \mbox{ in }\R^N.
\end{equation*}
If we show that this limit vanishes identically, $u_\infty\equiv 0$, then we will reach a contradiction with \eqref{contradiction2}.

For this, we wish to show that $\bar{v}$ also tends to a limit. If
\begin{equation}\label{estimate50}
\|v\|_{\mathcal C^{4\gamma+\alpha}_{\tilde{\mu}+2\gamma, \tilde{\nu}+2\gamma}}\leq C_0\|u\|_{\mathcal C^{2\gamma+\alpha}_{\tilde{\mu},\tilde{\nu}}},
\end{equation}
then it is true that the limit exists. If not, we can use the same contradiction argument to show that after some scaling, $\bar{v}$ will tend to a limit $v_\infty\in \mathcal C^{4\gamma+\alpha}_{\tilde{\mu}+2\gamma, \tilde{\mu}+2\gamma}$ which solves
\begin{equation*}
L_1^* v_\infty=0.
\end{equation*}
This implies that $v\equiv 0$. This will give a contradiction and yield that \eqref{estimate50} holds for some constant $C_0$.

By the above analysis we arrive at the limit problem, in which $u_\infty, v_\infty$ solve
\begin{equation*}
L_1 u_\infty=0, \ L_1^* v_\infty=u_\infty \quad \mbox{ in }\R^N.
\end{equation*}
Thus $L_1L_1^* v_\infty=0$. Multiply the equation by $v_\infty$ and integrate, one has $L_1^* v_\infty=0$, which implies that $v_\infty\equiv 0$. So also $u_\infty\equiv 0$. Then, following the argument in Lemma \ref{lemma:apriori-estimate}, one can get a contradiction. So the uniform surjectivity holds for all $\ve$ small.
\end{proof}

%
%
%

\section{Conclusion of the proof}\label{section:conclusion}

If $\phi$ is a solution to
\begin{equation*}
(-\Delta_{\mathbb R^n})^\gamma (\bar{u}_\ve+\phi)=|\bar{u}_\ve+\phi|^p \quad\mbox{ in }\R^n \setminus \Sigma,
\end{equation*}
we first show that $\bar{u}_\ve+\phi$ is positive in $\R^n \setminus \Sigma$.

Indeed, for $z$ near $\Sigma$, there exists $R>0$ such that if $\varrho(z)<R\ve$, then
\begin{equation*}
c_1\varrho(z)^{-\frac{2\gamma}{p-1}}<\bar{u}_\ve<c_2\varrho(z)^{-\frac{2\gamma}{p-1}}
\end{equation*}
for some $c_1,c_2>0$. Since $\phi\in \mathcal C^{2,\alpha}_{\tilde\mu,\tilde\nu}$, we have $|\phi|\leq c\varrho(z)^{\tilde\mu}$. But $\tilde\mu>-\frac{2\gamma}{p-1}$, so it follows that $\bar{u}_\ve+\phi>0$ near $\Sigma$. Since $\bar{u}_\ve+\phi \to 0$ as $|z|\to \infty$, by the maximum principle (see for example Lemma 4.13 of \cite{Cabre-Sire1}), we see that $\bar{u}_\ve+\phi>0$ in $\R^n \setminus \Sigma$, so it is a positive solution and it is singular at all points of $\Sigma$.

Next we will prove the existence of such $\phi$. For this, we take an additional restriction on $\tilde\nu$
$$-(n-2\gamma)<\tilde\nu<-\frac{2\gamma}{p-1}.$$

\subsection{Solution with isolated singularities ($\mathbb R^N\setminus \{q_1,\ldots,q_K\}$)}

We first treat the case where $\Sigma$ is a finite number of points.
Recall that equation
\begin{equation*}
(-\Delta_{\mathbb R^N})^\gamma (\bar{u}_\ve+\phi)=A_{N,p,\gamma}|\bar{u}_\ve+\phi|^p \mbox{ in }\R^N \setminus \{q_1,\ldots,q_K\}
\end{equation*}
is equivalent to the following:
\begin{equation}\label{nonlinear}
L_{\ve}(\phi)+Q_{\ve}(\phi)+f_{\ve}=0,
\end{equation}
where $f_\ve$ is defined in \eqref{f-epsilon}, $L_\ve$ is the linearized operator from (\ref{linearized-general}) and $Q_\ve$ contains the remaining higher order terms.
Because of Lemma \ref{lemma-uniform-surjectivity}, it is possible to construct a right inverse for $L_\ve$ with norm bounded independently of $\ve$. Define
\begin{equation}\label{F}
F(\phi):=G_{\ve}[-Q_{\ve}(\phi)-f_{\ve}],
\end{equation}
then equation \eqref{nonlinear} is reduced to
\begin{equation*}\label{nonlinear1}
\phi=F(\phi).
\end{equation*}
Our objective is to show that $F(\phi)$ is a contraction mapping from $\mathcal{B}$ to $\mathcal{B}$, where
\begin{equation*}
\mathcal{B}=\{\phi\in \mathcal C^{2\gamma+\alpha}_{\tilde\mu,\tilde\nu}(\R^n \setminus \Sigma):  \|\phi\|_*\leq \beta \ve^{N-\frac{2p\gamma}{p-1}}\}
\end{equation*}
for some large positive $\beta$.

In this section, $\|\cdot\|_*$ is the $\mathcal C^{2\gamma+\alpha}_{\tilde{\mu}, \tilde{\nu}}$ norm, and $\|\cdot\|_{**}$ is the $\mathcal C^{0,\alpha}_{\tilde{\mu}-2\gamma, \tilde{\nu}-2\gamma}$ norm where $\tilde{\mu}, \tilde{\nu}$ are taken as in the surjectivity section.

First we have the following lemma:
\begin{lemma}
We have that, independently of $\ve$ small,
\begin{equation*}
\|Q_{\ve}(\phi_1)-Q_{\ve}(\phi_2)\|_{**}\leq \frac{1}{2l_0}\|\phi_1-\phi_2\|_*
\end{equation*}
for all $\phi_1, \ \phi_2 \in \mathcal{B}$, where $l_0=\sup\|G_\ve\|$.
\end{lemma}
\begin{proof}
The estimates here are similar to Lemma 9 in \cite{mp}. For completeness, we give here the proof.

With some abuse of notation, in the following paragraphs the notation $\|\cdot\|_*$ and $\|\cdot\|_{**}$ will denote the weighted $\mathcal C^0$ norms and not the weighted $\mathcal C^{\alpha}$ norms (for the same weights) that was defined in \eqref{norm}.

First we show that there exists $\tau>0$ such that for $\phi\in \mathcal{B}$, we have
\begin{equation*}
 |\phi(x)|\leq \frac{1}{4}\bar{u}_\ve(x)\quad\text{for all}\quad x\in \bigcup_{i=1}^k B(q_i,\tau).
\end{equation*}
Indeed, from the asymptotic behaviour of $u_1$ in Proposition \ref{existence}
we know that
\begin{equation*}
\begin{split}
& c_1|x|^{-\frac{2\gamma}{p-1}}<u_{\ve_i}(x)<c_2|x|^{-\frac{2\gamma}{p-1}} \mbox{ if } |x|<R\ve_i,\\
&c_1\ve_i^{N-\frac{2p\gamma}{p-1}}|x|^{-(N-2\gamma)}<u_{\ve_i}(x)<c_2\ve_i^{N-\frac{2p\gamma}{p-1}}|x|^{-(N-2\gamma)} \mbox{ if }R\ve_i\leq |x|<\tau.
\end{split}
\end{equation*}
The claim follows because $\phi\in\mathcal B$ implies that
\begin{equation*}
|\phi(x)|<c\beta \ve^{N-\frac{2p\gamma}{p-1}}\varrho(x)^{\tilde\mu}.
\end{equation*}

Next, since $|\frac{\phi}{\bar{u}_\ve}|\leq \frac{1}{4}$ in $B(q_i,\tau)$, by Taylor' expansion,
\begin{equation*}
|Q_{\ve}(\phi_1)-Q_{\ve}(\phi_2)|\leq c|\bar{u}_{\ve}|^{p-2}(|\phi_1|+|\phi_2|)|\phi_1-\phi_2|.
\end{equation*}
Thus for $x\in B(q_i,\tau)$, we have
\begin{equation*}
\begin{split}
\varrho(x)^{2\gamma-\tilde\mu}|Q_{\ve}(\phi_1)-Q_{\ve}(\phi_2)|
&\leq c\varrho(x)^{\tilde\mu+\frac{2\gamma}{p-1}}(\|\phi_1\|_*+\|\phi_2\|_*)\|\phi_1-\phi_2\|_*\\
&\leq c\tau^{\tilde\mu+\frac{2\gamma}{p-1}}\beta\ve^{N-\frac{2\gamma}{p-1}}\|\phi_1-\phi_2\|_*.
\end{split}
\end{equation*}
The coefficient in front can be taken as small as desired by choosing $\ve$ small. Outside the union of the balls $B(q_i,\tau)$ we use the estimates
\begin{equation*}
\bar{u}_{\ve}(x)\leq c\ve^{N-\frac{2p\gamma}{p-1}}|x|^{-(N-2\gamma)}\quad \mbox{and}\quad|\phi|\leq c\ve^{N-\frac{2p\gamma}{p-1}}|x|^{\tilde\nu},
\end{equation*}
where $c$ depends on $\tau$ but not on $\ve$ nor $\phi$.

For $\varrho\geq \tau  $ and $|x|\leq R$, we can neglect all factors involving $\varrho(x)$, so
\begin{equation*}
\begin{split}
|Q_{\ve}(\phi_1)-Q_{\ve}(\phi_2)|&\leq c(|\bar{u}_{\ve}|^{p-1}+|\phi|^{p-1})|\phi_1-\phi_2|
\leq c\ve^{(p-1)(N-\frac{2p\gamma}{p-1})}|\phi_1-\phi_2|\\
&\leq c\ve^{p(N-2\gamma)-N}\|\phi_1-\phi_2\|_*,
\end{split}
\end{equation*}
for which the coefficient can be as small as desired since $p>\frac{N}{N-2\gamma}$.

Lastly, for $|x|\geq R$, in this region $\bar{u}_\ve=0$, so
\begin{equation*}
\begin{split}
\varrho(x)^{2\gamma-\tilde\nu}|Q_{\ve}(\phi_1)-Q_{\ve}(\phi_2)|&\leq c\varrho(x)^{2\gamma-\tilde\nu}(\phi_1^{p-1}+\phi_2^{p-1})|\phi_1-\phi_2|\\
&\leq c\varrho(x)^{2\gamma-\tilde\nu+p\tilde\nu}\ve^{N(p-1)-2p\gamma}\|\phi_1-\phi_2\|_*,
\end{split}
\end{equation*}
and here the coefficient can be also chosen as small as we wish because  $\tilde\nu<-\frac{2\gamma}{p-1}$  implies that $2\gamma-\tilde\nu+p\tilde\nu<0$.

Combining all the above estimates, one has
\begin{equation*}
\|Q_{\ve}(\phi_1)-Q_{\ve}(\phi_2)\|_{**}\leq \frac{1}{2l_0}\|\phi_1-\phi_2\|_*
\end{equation*}
as desired.

\bigskip

Now we go back to the original definition of the norms $\|\cdot\|_{*}$, $\|\cdot\|_{**}$ from \eqref{norm}.
For this, we need to estimate the H\"older norm of $Q_{\ve}(\phi_1)-Q_{\ve}(\phi_2)$. First in each $B(q_i,\tau)$,
\begin{equation*}
\begin{split}
\nabla Q_{\ve}(\phi)&=p\Big((\bar{u}_\ve+\phi)^{p-1}-\bar{u}_\ve^{p-1}-(p-1)\bar{u}_\ve^{p-1}\phi \Big)\nabla \bar{u}_\ve
+p((\bar{u}_\ve+\phi)^{p-1}-\bar{u}_\ve^{p-1})\nabla \phi,
\end{split}
\end{equation*}
and similarly as before, we can get that
\begin{equation*}
\varrho(x)^{2\gamma+1-\tilde\mu}|\nabla (Q_{\ve}(\phi_1)-Q_{\ve}(\phi_2))|\leq c\ve^{N-\frac{2p\gamma}{p-1}}\|\phi_1-\phi_2\|_*.
\end{equation*}
Moreover, for $\tau<\varrho(x)<R$,
\begin{equation*}
|\nabla (Q_{\ve}(\phi_1)-Q_{\ve}(\phi_2))|\leq c\ve^{p(N-2\gamma)-N}\|\phi_1-\phi_2\|_*.
\end{equation*}
Lastly, for $\varrho(x)>R$,
\begin{equation*}
\nabla Q_{\ve}(\phi_1-\phi_2)=p\phi_1^{p-1}\nabla(\phi_1-\phi_2)+p\nabla \phi_2(\phi_1^{p-1}-\phi_2^{p-1}),
\end{equation*}
which yields
\begin{equation*}
\begin{split}
&\varrho^{-\tilde\nu+2\gamma+1}|\nabla Q_{\ve}(\phi_1-\phi_2)|\\
&\leq \varrho^{2\gamma+1-\tilde\nu}\Big[(\|\phi_1\|_*+\|\phi_2\|_*)^{p-1}
\varrho^{(p-1)(\tilde\nu-2\gamma)}\|\phi_1-\phi_2\|_*\varrho^{\tilde\nu-2\gamma-1}
+\|\phi_2\|_*\varrho^{\tilde\nu-2\gamma-1}  \|\phi_1^{p-1}-\phi_2^{p-1}\|_*
\Big]\\
&\leq c\ve^{N-\frac{2p\gamma}{p-1}}\|\phi_1-\phi_2\|_*.
\end{split}
\end{equation*}
This completes the desired  estimate for $Q_{\ve}(\phi_1)-Q_{\ve}(\phi_2)$ and concludes the proof of the lemma.
\end{proof}

Recall that $\|f_\ve\|_{**}\leq C_0\ve^{N-\frac{2p\gamma}{p-1}}$ for some $C_0>0$ from \eqref{error1}. Then the lemma above gives an estimate for the map  \eqref{F}. Indeed,
\begin{equation*}
\begin{split}
\|F(\phi)\|_*&\leq l_0[\|Q_{\ve}(\phi)\|_{**}+\|f_\ve\|_{**}]
\leq l_0\|Q_{\ve}(\phi)\|_{**}+l_0C_0\ve^{N-\frac{2p\gamma}{p-1}}\\
&\leq \frac{1}{2}\|\phi\|_{*}+l_0C_0\ve^{N-\frac{2p\gamma}{p-1}}
\leq \beta \ve^{N-\frac{2p\gamma}{p-1}},
\end{split}
\end{equation*}
and
\begin{equation*}
\begin{split}
\|F(\phi_1)-F(\phi_2)\|_{*}&\leq l_0\|Q_{\ve}(\phi_1)-Q_{\ve}(\phi_2)\|_{**}\leq \frac{1}{2}\|\phi_1-\phi_2\|_{*}
\end{split}
\end{equation*}
if we choose $\beta>2l_0C_0$.  So $F(\phi)$ is a contraction mapping from $\mathcal{B}$ to $\mathcal{B}$. This implies the existence of a solution $\phi$ to \eqref{nonlinear}.

\subsection{The general case $\mathbb R^n\setminus \Sigma$, $\Sigma$ a sub-manifold of dimension $k$}
For the more general case, only minor changes need to be made in the above argument. The most important one comes from Lemma \ref{lemma:error2} and it says that the weight parameter $\mu$ must now lie in the smaller interval:
\begin{equation}
\label{choice-tilde-mu-final}-\frac{2\gamma}{p-1}<\tilde\mu<\min\left\{\gamma-\tfrac{2\gamma}{p-1}, \tfrac{1}{2}-\tfrac{2\gamma}{p-1}, \Real(\gamma_0^-)\right\}.
\end{equation}

In this case, we only need to replace the exponent $N-\frac{2p\gamma}{p-1}$ in the above argument by $q=\min\{\frac{(p-3)\gamma}{p-1}-\tilde\mu, \frac{1}{2}-\gamma+\frac{(p-3)\gamma}{p-1}-\tilde\mu\}$, then $q>0$ if $\tilde\mu$ is chosen to satisfy \eqref{choice-tilde-mu-final}. We get a solution to \eqref{nonlinear}, and this concludes the proof of Theorem \ref{teo}.

\section{Appendix}

\subsection{Some known results on special functions}

\begin{lemma}\label{propiedadeshiperg}
 \textnormal{ \cite{Abramowitz,SlavyanovWolfganglay}} Let $z\in\mathbb C$. The hypergeometric function is defined for $|z| < 1$ by the power series
$$ \Hyperg(a,b;c;z) = \sum_{n=0}^\infty \frac{(a)_n (b)_n}{(c)_n} \frac{z^n}{n!}=\frac{\Gamma(c)}{\Gamma(a)\Gamma(b)}\sum_{n=0}^\infty \frac{\Gamma(a+n)\Gamma(b+n)}{\Gamma(c+n)} \frac{z^n}{n!}.$$\label{hypergeo}
It is undefined (or infinite) if $c$ equals a non-positive integer.
Some properties are
\begin{itemize}
  \item[i.] The hypergeometric function evaluated at $z=0$ satisfies
  \begin{equation}\label{prop2}
  \Hyperg(a+j,b-j;c;0)=1; \  j=\pm1,\pm2,...
  \end{equation}
     \item[ii.] If $|arg(1-z)|<\pi$, then
\begin{equation}\label{prop4}
  \begin{split}
  \Hyperg&(a,b;c;z)=
                     \frac{\Gamma(c)\Gamma(c-a-b)}{\Gamma(c-a)\Gamma(c-b)}
                     \Hyperg\left(a,b;a+b-c+1;1-z\right)
                      \\
                     +&(1-z)^{c-a-b}\frac{\Gamma(c)\Gamma(a+b-c)}
                     {\Gamma(a)\Gamma(b)}\Hyperg(c-a,c-b;c-a-b+1;1-z).
     \end{split} \end{equation}
\item[iii.] The hypergeometric function is symmetric with respect to first and second arguments, i.e
\begin{equation}\label{prop5}
  \Hyperg(a,b;c;z)= \Hyperg(b,a;c;z).
  \end{equation}
\item[iv.] Let $m\in \n$. The $m$-derivative of the hypergeometric function is given by
\begin{equation}\label{prop6}
 \tfrac{d^m}{dz^m} \left[(1-z)^{a+m-1} \Hyperg(a,b;c;z)\right]= \tfrac{(-1)^m(a)_m(c-b)_m}{(c)_m} (1-z)^{a-1}\Hyperg(a+m,b;c+m;z).
  \end{equation}
\end{itemize}
\end{lemma}

\begin{lemma}\label{propiedadesgamma}
 \textnormal{ \cite{Abramowitz,SlavyanovWolfganglay}} Let $z\in\mathbb C$. Some well known properties of the Gamma function $\Gamma(z)$ are
\begin{align}
      &\Gamma(\bar{z})=\overline{\Gamma(z)},\label{prop1g}\\
  &\Gamma(z+1)=z\Gamma(z), \label{prop2g}\\
   &\Gamma(z)\Gamma\left(z+\tfrac{1}{2}\right)=2^{1-2z}\sqrt{\pi}\,\Gamma(2z).  \label{prop3g}
   \end{align}
It is a meromorphic function in $z\in\mathbb C$ and its residue at each poles is given by
\begin{equation}\label{residue-Gamma}
\Res(\Gamma(z),-j)=\frac{(-1)^n}{j!}, \quad j=0,1,\ldots.
\end{equation}

 Let $\psi(z)$ denote the Digamma function defined by $$\psi(z)=\frac{d\ln\Gamma (z)}{dz}=\frac{\Gamma'(z)}{\Gamma(z)}.$$
 This function has the expansion
\begin{equation}\label{expansion-digamma}
\psi(z)=\psi(1)+\sum_{l=0}^{\infty}\left(\frac{1}{l+1}-\frac{1}{l+z}\right).
\end{equation}

Let $B(z_1,z_2)$ denote the Beta function defined by
\begin{equation*}
B(z_1,z_2)=\frac{\Gamma(z_1)\Gamma(z_2)}{\Gamma(z_1+z_2)}.
\end{equation*}
If $z_2$ is a fixed number and $z_1>0$ is big enough, then this function behaves
\begin{equation*}\label{propbeta}
B(z_1,z_2)\sim \Gamma(z_2)(z_1)^{-z_2}.
\end{equation*}
\end{lemma}

\subsection{A review of the Fourier-Helgason transform on Hyperbolic space}\label{subsection:transform}

Consider hyperbolic space $\mathbb H^{k+1}$, parameterized with coordinates $\zeta$. It can be written as a symmetric space of rank one as the quotient $\mathbb H^{k+1}\approx \frac{SO(1,k+1)}{SO(k+1)}.$ Fourier transform on hyperbolic space is a particular case of the Helgason-Fourier transform on symmetric spaces. Some standard references are \cite{Bray,H,Terras:book1}; we mostly follow the exposition of Chapter 8 in \cite{Georgiev}.

Hyperbolic space $\mathbb H^{k+1}$ may be defined as the upper branch of a hyperboloid in $\mathbb R^{k+2}$ with the metric induced by the Lorentzian metric in $\mathbb R^{k+2}$ given by $-d\zeta_0^2+d\zeta_1^2+\ldots+d\zeta_{k+1}^2$, i.e., $\mathbb H^{k+1}  =\{(\zeta_0,\ldots,\zeta_{k+1})\in \mathbb R^{k+2}: \zeta_0^2-\zeta_1^2-\ldots-\zeta_{k+1}^2=1, \; \zeta_0>0\}$, which in polar coordinates may be parameterized as
$$\mathbb H^{k+1}= \{\zeta\in \mathbb R^{k+2}: \zeta= (\cosh s, \varsigma \sinh s ), \; s\geq 0, \; \varsigma\in \mathbb S^{k}\},$$
with the metric $g_{\mathbb H^{k+1}}=ds^2+\sinh^2 s\, g_{\mathbb S^{k}}$.
Under these definitions the Laplace-Beltrami operator is given by
$$\Delta_{\mathbb H^{k+1}}=\partial_{ss}+k \frac{\cosh s}{\sinh s}\, \partial_s+\frac{1}{\sinh^2 s}\,\Delta_{\mathbb S^{k}},$$
and the volume element is $$d\mu_\zeta=\sinh^{k}s\;ds \, d\varsigma.$$
We denote by $[\cdot, \cdot]$ the internal product induced by the Lorentzian metric, i.e.,
$$[\zeta,\zeta']=\zeta_0\zeta_0'-\zeta_1\zeta_1'-\ldots -\zeta_{k+1}\zeta_{k+1}'.$$
The hyperbolic distance between two arbitrary points is given by
$\dist(\zeta,\zeta')=\cosh^{-1}([\zeta,\zeta']),$
and in the particular case that $\zeta=(\cosh s, \varsigma\sinh s )$, $\zeta'=O$,
 $$\dist(\zeta,O)=s.$$
 The unit sphere $\mathbb S^{N-1}$ is identified with the subset $\{\zeta\in\mathbb R^{k+2} \,:\,[\zeta,\zeta]=0, \zeta_0=1\}$ via the map $b(\varsigma)=(1,\varsigma)$ for $\varsigma\in\mathbb S^{k}$.

Given $\lambda\in \mathbb R$ and $\omega \in \mathbb S^{k}$, let $h_{\lambda,\omega}(\zeta)$ be the generalized eigenfunctions of the Laplace-Beltrami operator. This is,
$$\Delta_{\mathbb H^{k+1}} h_{\lambda,\omega}=-\left(\lambda^2+\tfrac{k^2}{4}\right)h_{\lambda,\omega}.$$
These may be explicitly written as
$$h_{\lambda,\omega}(\zeta)=[\zeta,b(\omega)]^{i\lambda-\frac{k}{2}}=(\cosh s-\sinh s\langle \varsigma,\omega\rangle)^{i\lambda-\frac{k}{2}}, \quad \zeta\in \mathbb H^{k+1}.$$
In analogy to the Euclidean space, the Fourier transform on $\mathbb H^{k+1}$ is defined by
\begin{equation*}\hat{u}(\lambda, \omega)=\int_{\mathbb H^{k+1}} u(\zeta)\,h_{\lambda,\omega}(\zeta)\,d\mu_{\zeta}.
\end{equation*}
Moreover, the following inversion formula holds:
\begin{equation}\label{inversion}
u(\zeta)=\int_{-\infty}^{\infty}\int_{\mathbb S^{k}}\bar{h}_{\lambda,\omega}(\zeta)\hat{u}(\lambda,\omega)
\,\frac{d\omega \, d\lambda }{|c(\lambda)|^2},\end{equation}
where $c(\lambda)$ is the Harish-Chandra coefficient:
$$\frac{1}{|c(\lambda)|^2}=\frac{1}{2(2\pi)^{k+1}}
\frac{|\Gamma(i\lambda+(\frac{k}{2})|^2}{|\Gamma(i\lambda)|^2}.$$
There is also a Plancherel formula:
\begin{equation}\label{Plancherel}
\int_{\mathbb H^{k+1}}|u(\zeta)|^2\,d\mu_\zeta=\int_{\mathbb R\times \mathbb S^{N-1}}|\hat{u}(\lambda,\omega)|^2\frac{d\omega \; d\lambda }{|c(\lambda)|^2},
\end{equation}
which implies that the Fourier transform extends to an isometry between the Hilbert spaces $L^2(\mathbb H^{k+1})$ and $L^2(\mathbb R_+\times\mathbb S^{k},|c(\lambda)|^{-2}d\lambda \,d\omega)$.

If $u$ is a radial function, then $\hat u$ is also radial, and the above formulas simplify. In this setting, it is customary to normalize the measure of $\mathbb S^{k}$ to one in order not to account for multiplicative constants. Thus one defines the spherical Fourier transform as
\begin{equation}\label{spherical-transform}
\hat u(\lambda)=\int_{\mathbb H^{k+1}} u(\zeta)\Phi_{-\lambda}(\zeta)\,d\mu_\zeta,
\end{equation}
where
\begin{equation*}
\Phi_\lambda(\zeta)=\int_{\mathbb S^{k}} h_{-\lambda,\omega}(\zeta)\,d\omega
\end{equation*}
is known as the elementary spherical function. In addition, \eqref{inversion} reduces to
\begin{equation*}
u(\zeta)=\int_{-\infty}^{\infty}\hat{u}(\lambda)\Phi_\lambda(\zeta)
\,\frac{d\lambda }{|c(\lambda)|^2}.\end{equation*}
There are many explicit formulas for $\Phi_\lambda(\zeta)$ (we also write $\Phi_\lambda(s)$, since it is a radial function). In particular, $\Phi_{-\lambda}(s)=\Phi_\lambda(s)=\Phi_\lambda(-s),$ which yields regularity at the origin $s=0$.
Here we are interested in its asymptotic behavior. Indeed,
\begin{equation}\label{spherical-estimate}
\Phi_\lambda(s)\sim e^{(i\lambda-\frac{k}{2})s}\quad\text{as}\quad s\to +\infty.
\end{equation}
It is also interesting to observe that
$$\widehat{\Delta_{\mathbb H^{k+1}} u}=-\Big( \lambda^2+\tfrac{k^2}{4}\Big) \hat u.$$

We define the convolution operator as
$$u*v(\zeta)=\int_{\mathbb H^{k+1}} u(\zeta')v(\tau_\zeta^{-1}\zeta')\,d\mu_{\zeta'},$$
where $\tau_\zeta:\mathbb H^{k+1}\to \mathbb H^{k+1}$ is an isometry that takes $\zeta$ into $O$. If $v$ is a  radial function, then the convolution may be written as
$$u*v(\zeta)=\int_{\mathbb H^{k+1}} u(\zeta')v(\dist(\zeta, \zeta'))\,d\mu_{\zeta'},$$
and we have the property
\begin{equation*}\label{product-convolution}\widehat{u*v}=\hat u \,\hat v,\end{equation*}
in analogy to the usual Fourier transform.

On hyperbolic space there is a well developed theory of Fourier multipliers. In $L^2$ spaces everything may be written out explicitly. For instance, let $m(\lambda)$ be a multiplier in Fourier variables. A function $\hat u(\lambda,\omega)=\hat m(\lambda) u_0(\lambda,\omega)$, by the inversion formula for the Fourier transform \eqref{inversion} and expression \eqref{k}, may be written as
\begin{equation}\label{inversion1}\begin{split}u(x)=&
\int_{-\infty}^{\infty}\int_{\mathbb S^{k}}m(\lambda) \hat{u}_0(\lambda,\omega)\bar{h}_{\lambda,\omega}(\zeta)
\,\frac{d\omega \, d\lambda }{|c(\lambda)|^2}\\
=&\int_{-\infty}^{\infty}\int_{\mathbb H^{k+1}}m(\lambda) u_0(\zeta') k_\lambda(\zeta,\zeta')\,
d\mu_{\zeta'} \; d\lambda,
\end{split}\end{equation}
where we have denoted
\begin{equation*}\label{k}
k_\lambda(\zeta, \zeta') =\frac{1}{|c(\lambda)|^2}\int_{\mathbb S^{k}} \bar{h}_{\lambda,\omega}(\zeta) \,h_{\lambda,\omega}(\zeta')\,d\omega.
\end{equation*}
It is known that $k_\lambda$ is invariant under isometries, i.e.,
\begin{equation*}k_\lambda(\zeta, \zeta') =k_\lambda(\tau \zeta, \tau \zeta'),\end{equation*}
for all $\tau\in SO(1,k+1)$,
and in particular,
$$k_\lambda(\zeta, \zeta') =k_\lambda(\dist(\zeta,\zeta')),$$
so many times we will simply write $k_\lambda(\rho)$ for $\rho=\dist(\zeta,\zeta')$.
We recall the following formulas for $k_\lambda$:

\begin{lemma} \cite[Lemmas 8.5.2 y 8.5.3]{Georgiev}\label{lemma-k-lambda} For $k+1\geq 3$ odd,
\begin{equation*} k_\lambda(\rho)=c_k
\left(\frac{\partial_\rho}{\sinh \rho}\right)^\frac{k}{2}(\cos\lambda \rho),
\label{Kodd}\end{equation*}
and for $k+1\geq 2$ even,
\begin{equation*}\label{Keven}
k_\lambda(\rho)=c_k
\int_\rho^\infty\frac{\sinh \tilde\rho}{\sqrt{\cosh \tilde\rho-\cosh \rho}}
\left(\frac{\partial_{\tilde\rho}}{\sinh \tilde\rho}\right)^\frac{k+1}{2}(\cos\lambda \tilde\rho)\, d\tilde\rho.
\end{equation*}
\end{lemma}

\bigskip

\noindent\textbf{Acknowledgements.} H. Chan would like to thank Ali Hyder for useful discussions. M.d.M. Gonz\'alez is grateful to Rafe Mazzeo and Frank Pacard for many insightful discussions years ago. Part of the work was done when W. Ao was visiting the math department of Universidad Aut\'onoma de Madrid in July 2017, she thanks the hospitality and the financial support of the department. A. DelaTorre's research is  supported by Swiss National Science Foundation, projects nr. PP00P2-144669 and PP00P2-170588/1 while she was a postdoc under the supervision of Luca Martinazzi, and by the Spanish government grant MTM2014-52402-C3-1-P. M. Fontelos is supported by the Spanish government grant MTM2017-89423-P.  M.d.M. Gonz\'alez is supported by Spanish government grants MTM2014-52402-C3-1-P and MTM2017-85757-P, and the BBVA foundation grant for  Researchers and Cultural Creators, 2016. The research of H. Chan and J. Wei is supported by NSERC of Canada.

\end{document}